\documentclass[a4paper,11pt]{amsart}

\usepackage{amssymb}
\usepackage[all]{xy}
\usepackage{latexsym}
\usepackage{amsmath}
\usepackage{amsthm}
\usepackage{xcolor}
\usepackage{enumerate}
\usepackage{enumitem}
\usepackage{tikz}
\usepackage{tikz-cd}

\newcommand{\PU}{{\rm Pres}(\mathcal{U})}

\newcommand{\fp}{{\rm fp}}

\DeclareMathOperator{\Mor}{Mor}

\newcommand{\End}{\operatorname{End}}

\newcommand{\Ext}{\operatorname{Ext}}

\newcommand{\Cogen}{\operatorname{Cogen}}
\newcommand{\Gen}{\operatorname{Gen}}
\newcommand{\Ker}{\operatorname{Ker}}

\newcommand{\tr}{\operatorname{tr}}
\newcommand{\rej}{\operatorname{rej}}
\newcommand{\Img}{\operatorname{Im}}
\newcommand{\Coker}{\operatorname{Coker}}

\newcommand{\cone}{\operatorname{cone}}

\DeclareMathOperator{\Pres}{Pres}
\DeclareMathOperator{\Copres}{Copres}

\newcommand{\D}{\mathcal{D}}

\newcommand{\X}{\mathcal{X}}

\newcommand{\Modr}[1]{\mathrm{Mod}\textrm{-}{#1}}

\newcommand{\Inj}{\mathrm{Inj}}

\newcommand{\Proj}{\mathrm{Proj}}

\newcommand{\Ab}{\mathrm{Ab}}

\newcommand{\Prod}{\mathrm{Prod}}

\DeclareMathOperator{\Sum}{Sum}

\theoremstyle{plain}
\newtheorem{theorem}{Theorem}[section]
\newtheorem{lemma}[theorem]{Lemma}
\newtheorem{proposition}[theorem]{Proposition}
\newtheorem{corollary}[theorem]{Corollary}

\newtheorem{question}{Question}

\newtheorem{definition}[theorem]{Definition}

\newtheorem{remark}[theorem]{Remark}
\newtheorem{remarks}[theorem]{Remarks}

\newtheorem{example}[theorem]{Example}

\title{REFLECTIVE AND COREFLECTIVE SUBCATEGORIES}

\author{Manuel Cort\'es-Izurdiaga}
\address{Departamento de Matemática Aplicada, Universidad de Málaga, 29071, Málaga, SPAIN}
\thanks{The first named author is partially supported by the Spanish Government under grant PID2020-113552GB-I00 and by Junta de Andalucia under grant $P20_00770$}
\email{mizurdiaga@uma.es}
\author{Septimiu Crivei}
\address{Department of Mathematics, Babe\c s-Bolyai University, Str. Mihail Kog\u alniceanu 1, 400084, Cluj-Napoca, ROMANIA}
\thanks{The second named author would like to thank the Department of Mathematics of the University of Murcia for support and kind hospitality during his visit from July 2019, when part of this work was begun.}
\email{crivei@math.ubbcluj.ro}
\author{Manuel Saor\'{\i}n}
\address{Departamento de Matemáticas. Universidad de Murcia. Aptdo. 4021. 301000, Espinardo, MU (SPAIN)}
\thanks{The third named author is partially supported by the grant PID2020-113206GB-I00, funded by MCIN/AEI/10.13039/501100011033}
\email{msaorinc@um.es}

\thanks{The authors thank the referee for the careful reading of our manuscript and the comments that improved the presentation and readability of the paper.}

\begin{document}

\begin{abstract}
Given any additive category $\mathcal{C}$ with split idempotents, pseudokernels and pseudocokernels, we show that a subcategory $\mathcal{B}$ is coreflective if, and only if, it is precovering, closed under direct summands and each morphism in $\mathcal{B}$ has a pseudocokernel in $\mathcal{C}$ that belongs to $\mathcal{B}$. We  apply this result and its dual to, among others, preabelian and pretriangulated categories. As a consequence, we show that a subcategory of a preabelian category is coreflective if, and only it, it is precovering and closed under taking cokernels. On the other hand, if $\mathcal{C}$ is pretriangulated with split idempotents, then a subcategory $\mathcal{B}$ is coreflective and invariant under the suspension functor if, and only if, it is precovering and closed under taking direct summands and cones. These are extensions of well-known results for AB3 abelian and triangulated categories, respectively. 

By-side applications of these results allow us: a) To characterize the coreflective subcategories of a given AB3 abelian category which have a set of generators and are themselves abelian, abelian exact or  module categories; b) to extend to module categories over arbitrary small preadditive categories a result of Gabriel and De la Pe\~na stating that all fully exact subcategories are bireflective; c) to show that, in any Grothendieck category, the direct limit closure of its subcategory of finitely presented objects is a coreflective subcategory.

\end{abstract}

\maketitle


\subjclass{\textit{2020 AMS Subject Classification:} 18A40, 18E05, 18E10, 18C35, 18G80.}

\keywords{\textit{Keywords:} Additive category, (co)reflective subcategory, (pre)abelian category, (pre)triangulated category, exactly definable category, finitely accessible category, fully exact subcategory.}

\section{Introduction}
In all fields of Mathematics where Category Theory plays a significative role, it is a classical problem to determine when a given functor $F:\mathcal{B}\longrightarrow\mathcal{C}$ has a left or a right adjoint. When $\mathcal{B}$ is a (co)complete category, the most important tool is Freyd Adjoint Theorem (see \cite[Chapter 3]{Freyd} or \cite[Theorem 3.3.3]{Borceux}). This theorem states that, when $\mathcal{B}$ is complete, $F$ has a left adjoint if, and only if, it preserves all limits (equivalently, it preserves products and kernels) and satisfies a so-called solution-set condition.   The conjunction of product-preservation plus the solution-set condition is very much related to the  condition that  the essential image $\text{Im}(F)$ be a preenveloping subcategory of $\mathcal{C}$ (we refer the reader to Section \ref{sec.preliminares} for the unexplained terminology). We point out that avoiding the solution-set condition can be achieved sometimes by additional Set Theory axioms, something that  sheds light on the complexity of deciding when a functor has adjoints. For instance, assuming Vopenka's principle, every full subcategory of a locally presentable category which is closed
under (co)limits is (co)reflective, i.e. the inclusion functor has a (right) left adjoint \cite[Theorems 6.22 and 6.28]{AR} (notice that
every locally presentable category is bicomplete by \cite[Definition 1.17 and Corollary 1.28]{AR}).

In this paper, we are mainly interested in the case  when $\mathcal{C}$ is an arbitrary additive category and $F:\mathcal{B}\hookrightarrow\mathcal{C}$ is the inclusion functor from an (always full) subcategory, so that our goal is to give necessary and sufficient conditions for a subcategory to be (co)reflective. It is well-known that such a subcategory has to be preenveloping (resp. precovering) in $\mathcal{C}$. The natural question that arises asks whether this condition is also sufficient. Alternatively, one may ask what are the minimal conditions to add to the preenveloping (resp. precovering) condition of $\mathcal{B}$ for it to be a reflective (resp. coreflective) subcategory of $\mathcal{C}$. 

Our interest in this question stems from a result of Neeman (see \cite[Proposition 1.4]{Neeman10}) in the context of triangulated categories. It states that if $\mathcal{C}$ is a triangulated category with split idempotents, then a full triangulated subcategory $\mathcal{B}$ is coreflective if, and only if, it is precovering and closed under direct summands. Neeman's result was slightly generalized in \cite[Proposition 3.11]{SaorinStovicek}, where the same result is obtained, assuming only that $\mathcal{B}$ is a suspended subcategory, i.e. closed under extensions, direct summands and taking cones.  On the other hand, it is an immediate consequence of the dual of \cite[Theorem 3.1]{RSdV} that if $\mathcal{C}$ is  cocomplete abelian, a subcategory $\mathcal{B}$ is coreflective if, and only if, it is precovering and closed under taking cokernels.  

The common flavour of these results,  coming  from  the apparently distant triangulated and abelian worlds,  led us to study the above mentioned question in a  general context that includes both the triangulated and the abelian ones. Our basic result, Theorem \ref{t:CharacterizationCoreflective}, implies that if $\mathcal{C}$ is any additive category with split idempotents that has pseudokernels and pseudocokernels then a subcategory $\mathcal{B}$ is coreflective if, and only if, it is precovering, closed under direct summands and every morphism in $\mathcal{B}$ has a pseudocokernel in $\mathcal{C}$ that belongs to $\mathcal{B}$.  This result together with its dual allow not only to recover, in a generalized form, the results mentioned in the previous paragraph. They apply to the case when the category $\mathcal{C}$ is one-sided triangulated or pretriangulated, including the case of  the stable category of an abelian category modulo an appropriate subcategory, and also to the case when $\mathcal{C}$ is   exactly definable  or   preabelian. 

We pass to describe the organization of the paper and, in the process, we point out some of the main results of the paper. In Section 2 we introduce most of the concepts needed in the rest of the paper. In Section 3 we give the basic Theorem  \ref{t:CharacterizationCoreflective} mentioned above and its dual. 

In Section 4 we apply these theorems to one-sided triangulated categories and pretriangulated categories, in the sense of \cite{BeligiannisReiten} (see also \cite{BeligiannisMarmaridis}). The first half-dual part of Corollary \ref{c:ReflectiveAndCoreflectiveInPretriangulated} states that if $\mathcal{D}$ is a right triangulated category with split idempotents and pseudokernels, then a subcategory $\mathcal{Z}$ is coreflective and invariant under the suspension functor if, and only if, it is precovering, closed under direct summands and under taking cones. This strictly generalizes  \cite[Proposition 1.4]{Neeman10} and  \cite[Proposition 3.11]{SaorinStovicek}. 

In Section 5 we consider any abelian category $\mathcal{A}$ and any preenveloping (resp. precovering) subcategory $\mathcal{P}$, in which case the stable category $\underline{\mathcal{A}}=\mathcal{A}/\mathcal{P}$ is a right (resp. left) triangulated category. A natural extension of the concept of (co)resolving subcategory of $\mathcal{A}$, that we call weakly $\mathcal{P}$-(co)resolving, is defined. We show in Corollary \ref{cor.bijection-Presolving-reflectives} that, when $\underline{\mathcal{A}}$ has split idempotents, a situation that is very frequent in practice (see Proposition \ref{ex.stablecategs-splitidempotents}), and   $\mathcal{P}$ is both precovering and preenveloping in $\mathcal{A}$,  the assignment $\mathcal{X}\rightsquigarrow\underline{\mathcal{X}}$ gives a bijection between the preenveloping weakly $\mathcal{P}$-resolving subcategories of $\mathcal{A}$ closed under direct summands and the reflective subcategories of $\underline{\mathcal{A}}$ invariant under the loop functor. This bijection and its dual match very well in case $\mathcal{A}$ is a Frobenius abelian category and $\mathcal{P}$ is the subcategory of projective (=injective) objects. In this case, Corollary \ref{cor.bijections in Frobenius abelian} gives a bijection between: i) the preenveloping resolving subcategories of $\mathcal{A}$ closed under direct summands; ii) the precovering coresolving subcategories of $\mathcal{A}$ closed under direct summands; iii) the t-structures in the triangulated category $\underline{\mathcal{A}}$.  

In Section 6 we show that  exactly definable additive categories have split idempotents, pseudokernels and pseudocokernels (Proposition \ref{prop:sipp}), so that Theorem  \ref{t:CharacterizationCoreflective} applies to them. 

In Section 7, the first half-dual of Corollary \ref{c:AbelianCategories} shows that if $\mathcal{A}$ is a preabelian category, then a subcategory $\mathcal{B}$ is coreflective if, and only if, it is precovering and closed under taking cokernels. In the particular case when $\mathcal{A}$ is abelian, we give in Theorem \ref{t:CoreflectiveAndAbelian} a handy criterion for such a coreflective subcategory to be also abelian. 

Motivated by the fact that, when $\mathcal{A}$ is abelian, any coreflective abelian subcategory with a set of generators must be the subcategory $\text{Pres}(\mathcal{U})$ of objects presented by  some set $\mathcal{U}$ of objects, we study in Section 8 when the subcategory $\text{Pres}(\mathcal{U})$ is coreflective, coreflective abelian and coreflective abelian exact, respectively. Theorems \ref{t:PresCoreflective}, \ref{t.coreflective-abelian} and \ref{t:abexact} give  necessary and sufficient conditions for that to happen. Their duals are gathered in Section 8*.

In the final Section 9 we show several applications of the previous results. For instance, for a given AB3 abelian category $\mathcal{A}$, we determine the coreflective abelian and coreflective abelian exact subcategories of $\mathcal{A}$ that are module categories over a preadditive category or over a ring (Corollaries \ref{cor.coreflective-modulecategory} and \ref{cor.abexact-with-set-projgenerators}). In Theorem \ref{t:Gabriel-De la Pena} we prove that fully exact subcategories of module categories over small preadditive categories are always bireflective and they are exactly the ones induced by an epimorphism of preadditive categories. This extends the famous theorem of Gabriel and De la Pe\~na (see \cite{GabrieldelaPena}), who proved the result for module categories over rings. We end the section by showing that, given a Grothendieck category $\mathcal{G}$, the subcategory $\varinjlim(\text{fp}(\mathcal{G}))$ of the objects which are direct limits of finitely presented objects is a coreflective subcategory (Proposition \ref{prop.lim-fpG}).

\section{Preliminaries} \label{sec.preliminares}

\subsection{Additive, (pre)abelian and Grothendieck categories} 

Unless otherwise explicitly stated, all categories in this paper are additive and all subcategories are full. For any objects $M$ and $N$ of an additive category $\mathcal{C}$ we denote by $\mathcal{C}(M,N)$ the set of morphisms from $M$ to $N$. An additive subcategory of $\mathcal C$ is a subcategory $\mathcal B$ of $\mathcal C$ that contains the zero object and is closed under finite (co)products. All subcategories are closed under taking isomorphisms. A \emph{pseudokernel} of a morphism $f:A \rightarrow B$ is a morphism $k:K \rightarrow A$ such that $fk=0$ and any morphism $k':K' \rightarrow A$ with $fk'=0$ factors through $k$ (in a not necessarily unique way, unlike the case for kernels). 
In a dual fashion one  defines the concept of \emph{pseudocokernel}.

An additive category $\mathcal{C}$ \emph{has split idempotents} (or \emph{is idempotent complete}) if, for every object $A$ of $\mathcal{C}$ and every idempotent $e\in \mathcal{C}(A,A)$, there is an object $B$ and morphisms $p:A\to B$ and $i:B\to A$ in $\mathcal{C}$ such that $pi=1_B$ and $ip=e$ (e.g., see \cite[p.~6]{Prest}). Recall that the \emph{idempotent completion} $\hat{\mathcal{C}}$ of an additive category $\mathcal{C}$ is an additive category with split idempotents that contains $\mathcal{C}$ as a subcategory and such that, for any  additive functor $F:\mathcal{C} \rightarrow \mathcal{C}'$ in which $\mathcal{C}'$ is an additive category with split idempotents, there exists  an additive functor $\hat F: \hat{\mathcal{C}} \rightarrow \mathcal{C}'$, unique up to natural isomorphism, such that $\hat F i = F$, where $i:\mathcal{C} \rightarrow \hat{\mathcal{C}}$ is the inclusion. The idempotent completion $\hat{\mathcal{C}}$ of $\mathcal{C}$ always exists and has as objects the pairs $(A,e)$ with $A$ an object of $\mathcal{C}$ and $e$ an idempotent endomorphism of $A$. A morphism $\alpha:(A,e) \rightarrow (B,e')$ in $\hat{\mathcal{C}}$ is just a morphism $\alpha:A \rightarrow B$ in $\mathcal{C}$ satisfying $\alpha e = e' \alpha = \alpha$ (we use the same notation for the morphism in $\hat{\mathcal{C}}$ and the corresponding morphism in $\mathcal{C}$). We point out that many naturally arising additive categories are idempotent complete. A classical example of non-idempotent complete category is that of free modules over a ring which has a non-free projective module. Note that the idempotent completion of the category of free modules over any ring is equivalent to the category of projective modules. Other examples of categories with and without split idempotents will be given in the forthcoming Remark  \ref{4.6} and Proposition \ref{ex.stablecategs-splitidempotents}.
An additive category $\mathcal{C}$ is called \emph{Krull–Schmidt} if it has split idempotents and each object of $\mathcal{C}$ decomposes into a finite coproduct of (indecomposable) objects having local endomorphism rings.

A category $\mathcal{A}$ is called \emph{preabelian} if it is an additive category with kernels and cokernels. Note that an additive category is abelian if and only if it is preabelian, every monomorphism is a kernel, and every epimorphism is a cokernel. A typical example of preabelian category which is not abelian is the category of torsionfree abelian groups. Let $A$ and $B$  be objects of $\mathcal{A}$. If $B$ is a subobject of $A$ we write $B \leq A$ and we call any representing monomorphism of $B$ the inclusion. If $B'$ is another subobject of $A$, we write $B' \subseteq B$ if there exists a morphism $f:B' \rightarrow B$ such that $kf=k'$, where $k:B \rightarrow A$ and $k':B' \rightarrow A$ are the corresponding inclusions.

A subcategory $\mathcal{B}$ of an abelian category $\mathcal{A}$ is called \textit{abelian exact} if $\mathcal{B}$ is abelian and the inclusion functor is exact. This is equivalent to $\mathcal{B}$ being closed under taking finite coproducts, kernels and cokernels. When $\mathcal{A}$ is also bicomplete, a subcategory $\mathcal{B}$  is called \emph{fully exact} when it is closed under taking kernels, cokernels, coproducts and products. Equivalently, when $\mathcal{B}$ is an abelian exact subcategory closed under taking arbitrary products and coproducts.

An abelian category $\mathcal{A}$ is said to be \emph{AB3} if it is cocomplete, \emph{AB4} if it is AB3 and has exact coproducts, and \emph{AB5} if it is AB3 and has exact directed colimits. Dually, one has \emph{AB3$^*$}, \emph{AB4$^*$} and \emph{AB5$^*$} abelian categories. An abelian category is called \emph{Grothendieck} if it is AB5 and has a generator or, equivalently, a set of generators.

An object $F$ of an  additive category $\mathcal{A}$ with direct limits is \textit{finitely presented} if the functor $\mathcal{A}(F,-)$ preserves direct limits. In such case, we denote by $\fp(\mathcal{A})$ the  subcategory of all finitely presented objects of $\mathcal{A}$. The category $\mathcal{A}$ is called \emph{finitely accessible} (or \emph{locally finitely presented} in the terminology of \cite{CB}) when $\fp (\mathcal{A})$ is skeletally small and each object of $\mathcal{A}$ is a direct limit of objects in $\fp (\mathcal{A})$. When such a category is also abelian, it is automatically a Grothendieck category (see \cite[Section 2.4]{CB}). A finitely accessible Grothendieck category $\mathcal{A}$ on which $\fp (\mathcal{A})$ is an abelian exact subcategory (equivalently, closed under taking kernels) is called \emph{locally coherent}.

Recall that if $\mathcal{C}$ is any  category and $I$ is any small category, then a \emph{diagram of shape $I$} in $\mathcal{C}$  is just a  functor $I\rightarrow\mathcal{C}$. The category of diagrams of shape $I$ in $\mathcal{C}$ will be denoted by $[I,\mathcal{C}]$. It is well-known  that if $\mathcal{C}$ is AB3 abelian (resp. AB3$^*$ abelian), then so is $[I,\mathcal{C}]$ and colimits (resp. limits) in it are calculated pointwise.  

\begin{definition} \label{def.exactness-colimits in preabelian}
Let $\mathcal{C}$ be a cocomplete  category and $I$ be a small category. We will say that \emph{$I$-colimits preserve monomorphisms} in $\mathcal{C}$ when, given any morphism $f:D\rightarrow D'$ in $[I,\mathcal{C}]$ such that $f_i:D(i)\rightarrow D'(i)$ is a monomorphism, for all $i\in I$, the induced morphism ${\rm colim}\,f:{\rm colim}\,  D\rightarrow{\rm colim}\, D'$ is a monomorphism in $\mathcal{C}$.
\end{definition}

The notion that \emph{$I$-limits preserve epimorphisms} in a complete  category $\mathcal{A}$ is defined dually. The following is well known.

\begin{example} \label{ex.preservation of monos - AB conditions}
If $\mathcal{A}$ is an AB3 abelian category and $I$ is any small category, then $[I,\mathcal{A}]$ is also an AB3 abelian category. In this case, a morphism $f:D\rightarrow D'$ satisfies that $f_i:D(i)\rightarrow D'(i)$ is a monomorphism, for all $i\in I$, if and only if $f$ is a monomorphism in $[I,\mathcal{A}]$, if and only if it is a kernel in $[I,\mathcal{A}]$. Bearing in mind that the functor ${\rm colim}:[I,\mathcal{A}]\rightarrow\mathcal{A}$ is right exact, we conclude that $I$-colimits preserve monomorphisms if, and only if, the functor  ${\rm colim}:[I,\mathcal{A}]\rightarrow\mathcal{A}$ is exact. In particular:

\begin{enumerate}
\item $\mathcal{A}$ is AB4 if, and only if, for each set (=discrete small category) $I$,  the coproduct functor $\coprod_I:[I,\mathcal{A}]\rightarrow\mathcal{A}$ preserves monomorphisms.
\item $\mathcal{A}$ is AB5 if, and only if, for each directed set $I$,  the direct limit (=directed colimit) functor $\varinjlim: [I,\mathcal{A}]\rightarrow\mathcal{A}$ preserves monomorphisms.
\end{enumerate}
\end{example}

Of course, there is a dual result about AB3$^*$ abelian categories.

\subsection{Reflective and coreflective subcategories} \label{ss:reflective-coreflective}
A subcategory $\mathcal{B}$ of a category $\mathcal{C}$ is called \emph{coreflective} if the inclusion functor $i:\mathcal{B}\to \mathcal{C}$ has a right adjoint. Denote by $\nu$ and $\eta$ the counit and the unit of the adjunction respectively. 
Then $\eta$ is a natural isomorphism  (see \cite[Proposition II.7.5]{HiltonStambach}). Dually, a subcategory $\mathcal{B}$ is called \emph{reflective} if $i$ has a left adjoint. Also, $\mathcal{B}$ is called \emph{bireflective} if it is both coreflective and reflective.

It is well-known that a (co)reflective subcategory $\mathcal{B}$ is closed under all (co)limits that exist in $\mathcal{C}$ (see the proof of \cite[Proposition 3.5.3]{Borceux} for the limit case). In particular, when $\mathcal{C}$ is additive, $\mathcal{B}$ is an additive subcategory and  $i$ is an additive functor, which in turns implies that its right adjoint is also additive.

Every (co)reflective subcategory $\mathcal{B}$ of a preabelian category $\mathcal{A}$ is also preabelian. Indeed, by the previous paragraph, we know that $\mathcal{B}$ is  additive and closed under (co)kernels, which are then calculated in $\mathcal{B}$ as in $\mathcal{A}$. When $\mathcal{B}$ is coreflective and $f:B\to B'$ is a morphism in $\mathcal{B}$,  the kernel of $f$ in $\mathcal{B}$ is the composite morphism $q(\Ker f)\stackrel{\nu_{\Ker f}}\to \Ker f\to B$, where the latter arrow is the inclusion. A dual argument can be used to compute the cokernel of $f$ when $\mathcal{B}$ is reflective.

\subsection{Precovers and preenvelopes} Let $\mathcal{X}$ be a class of objects of a category $\mathcal{C}$. A morphism $f:A \rightarrow A'$ is called an \textit{$\mathcal{X}$-epimorphism} if $\mathcal{C}(X,f)$ is an epimorphism in the category of abelian groups for each $X \in \mathcal{X}$. An \textit{$\mathcal{X}$-precover} of an object $A$ of $\mathcal{C}$ is an $\mathcal{X}$-epimorphism $f:X \rightarrow A$ with $X \in \mathcal{X}$. An \textit{$\mathcal{X}$-coreflection} of $A$ is an $\mathcal{X}$-precover $f$ such that $\mathcal{C}(X,f)$ is an isomorphism for each $X \in \mathcal{X}$. Finally, an \textit{$\mathcal{X}$-cover} of $A$ is an $\mathcal{X}$-precover $f:X \rightarrow A$ such that any $g:X \rightarrow X$ satisfying $fg=f$ is an isomorphism.

The morphism $f$ is called an \textit{$\mathcal X$-monomorphism} if $\mathcal C(f,X)$ is an epimorphism in the category of abelian groups for each $X \in \mathcal X$. Using this definition, one can define the notions of $\mathcal{X}$-preenvelope, $\mathcal{X}$-reflection and $\mathcal{X}$-envelope in a dual fashion.

The class $\mathcal{X}$ is called \textit{precovering} (resp. \textit{preenveloping}) if every object in $\mathcal{C}$ has an $\mathcal{X}$-precover (resp. $\mathcal{X}$-preenvelope). It is well known that every precovering subcategory closed under taking direct summands is closed under taking any direct sum of objects that exists in $\mathcal{C}$. Moreover, $\mathcal X$ is (co)reflective if and only if every object of $\mathcal C$ has a (co)reflection.

We will need the following standard known fact for (pre)covers (e.g., see \cite[Theorem 1.2.7]{Xu} in module categories):

\begin{lemma} \label{lem.cover-directsummand-of-precover}
Let $\mathcal{A}$ be an abelian category, and let $\X$ be a class of objects of $\mathcal{A}$. Let $p_A:X_A\longrightarrow A$ and $q_A:\hat{X}_A\longrightarrow A$ be an $\X$-precover and an $\X$-cover, respectively, of an object $A$. Then, any morphism $f:X_A\longrightarrow\hat{X}_A$ such that $q_A f=p_A$ is a retraction. Moreover, there exists a subobject $X'_A$ of $X_A$ and an isomorphism $\varphi:X'_A \rightarrow \hat X_A$ such that $X_A=X'_A \coprod \Ker f$ and $p_A=q_A \varphi \pi$, where $\pi:X_A \rightarrow X_A'$ is the projection associated to this decomposition.
\end{lemma}

\subsection{Complete cotorsion pairs}

Let $\mathcal A$ be an abelian category. Given a class $\mathcal X$ of objects of $\mathcal A$, its right and left perpendicular class is
\begin{displaymath}
\mathcal X^{\perp}=\{Y \in \mathcal A \mid \Ext^1(X,Y)=0, \, \forall X \in \mathcal X\} \textrm{ and } {^\perp}{\mathcal X}=\{Y \in \mathcal A \mid \Ext^1(Y,X)=0, \, \forall X \in \mathcal X\},
\end{displaymath}
respectively. A \textit{cotorsion pair} in $\mathcal A$ is a pair of classes of objects of $\mathcal A$, $(\mathcal X,\mathcal Y)$, such that $\mathcal X^\perp = \mathcal Y$ and ${^\perp}{\mathcal Y}=\mathcal X$. The cotorsion pair is called \textit{complete} if for any object $A$ of $\mathcal A$, there exist short exact sequences
\begin{displaymath}
0 \rightarrow Y \rightarrow X \xrightarrow{f} A \rightarrow 0 \textrm{ and } 0 \rightarrow A \xrightarrow{g} Y' \rightarrow X' \rightarrow 0
\end{displaymath}
with $X,X' \in \mathcal X$ and $Y,Y' \in \mathcal Y$. Notice that $f$ is an $\mathcal X$-precover and $g$ is a $\mathcal Y$-preenvelope, so that a complete cotorsion pair gives a precovering and a preenveloping class.

\subsection{Trace and reject} \label{sec.trace-reject}

Let $\mathcal{A}$ be an AB3 abelian category and $\mathcal{X}$ a class of objects of $\mathcal{A}$. We use the following notation: 
${\rm Sum}(\X)$ (${\rm sum}(\X)$), $\Prod(\mathcal X)$, ${\rm Gen}(\mathcal{X})$ and ${\rm Pres}(\X)$ for the subcategories of $\mathcal{A}$ consisting, respectively, of (finite) coproducts of objects from $\X$, products of objects from $\mathcal X$, of  $\mathcal{X}$-generated objects, i.e. objects that are epimorphic image of coproducts of objects in $\X$,  and of $\X$-presented objects, i.e. objects that are isomorphic to cokernels of morphisms in ${\rm Sum}(\X)$. 

\begin{definition}\label{d:Trace}
Let $\mathcal{A}$ be an AB3 abelian category, $\mathcal{X}$
be a class of objects and $A$ be any object of $\mathcal{A}$. The \emph{trace} of
$\mathcal{X}$ in $A$, denoted $\tr_\mathcal{X}(A)$, is a subobject of $A$
satisfying the following two properties:

\begin{enumerate}
\item $\tr_\mathcal{X}(A)$ is $\mathcal{X}$-generated;

\item If $B$ is any $\mathcal{X}$-generated subobject of $A$, then
$B\subseteq \tr_\mathcal{X}(A)$.
\end{enumerate}
\end{definition}

In general $\tr_\mathcal{X}(A)$ may not exist. However, we have:

\begin{lemma}\label{l:Existence_trace}
In the situation of the last definition, if one of the following
conditions holds, then $\tr_\mathcal{X}(A)$ exists for all
objects $A$ in $\mathcal{A}$:

\begin{enumerate}
\item $\mathcal{X}$ is precovering. Then $\tr_\mathcal{X}(A)$ is the image of any $\mathcal{X}$-precover of $A$.

\item There is a set $\mathcal{X}_0\subseteq\mathcal{X}$ such that
$\mathcal{X}\subseteq{\rm Gen}(\mathcal{X}_0)$.

\item $\mathcal{A}$ is locally small, i.e. the subobjects of any object form a set (e.g., when $\mathcal{A}$ has a set of generators, see \cite[Proposition IV.6.6]{Stenstrom}).
\end{enumerate}
\end{lemma}

\begin{proof}
Let $A$ be an object of $\mathcal{A}$. Assertion (1) is trivial. If (2) holds, then the trace of $\mathcal{X}_0$ in $A$  is the image of the canonical morphism $\varepsilon_A:\coprod_{X\in\mathcal{X}_0}X^{(\mathcal{A}(X,A))}\to A$, and we have 
 $\tr_{\mathcal{X}}(A)=\tr_{\mathcal{X}_0}(A)$. Finally, if (3) holds, then $\tr_{\mathcal{X}}(A)$ is the sum of all $\mathcal{X}$-generated subobjects of $A$.
\end{proof}

Assuming that the trace $\tr_\mathcal{X}(A)$ exists, for all objects $A\in\mathcal{A}$, one inductively defines an ascending sequence $\tr_\mathcal{X}=:\tr^1_\mathcal{X}\subseteq...\subseteq \tr^n_\mathcal{X}\subseteq...$ of subfunctors of the identity  by the rule 
$\frac{\tr^n_{\mathcal{X}}(A)}{\tr^{n-1}_{\mathcal{X}(A)}}={\rm tr}_{\mathcal{X}}(\frac{A}{{\rm tr}^{n-1}_{\mathcal{X}}(A)})$, for all $A\in\mathcal{A}$,  whenever $n>1$. In this way, if we denote by $(1:{\rm tr}_\mathcal{X}^n)$ the endofunctor of $\mathcal{A}$ that takes $A\rightarrow\frac{A}{{\rm tr}^n_{\mathcal{X}}(A)}$, then we have a natural isomorphism $(1:{\rm tr}_\mathcal{X}^n)\cong (1:{\rm tr}_\mathcal{X})^n:= (1:{\rm tr}_\mathcal{X})\circ\stackrel{n}{.....}\circ (1:{\rm tr}_\mathcal{X})$, for all $n>0$.

We can make the dual considerations with an AB3$^*$ abelian category $\mathcal{A}$. In this case, we use the following notation: ${\rm Prod}(\X)$ (${\rm prod}(\X)$), ${\rm Cogen}(\mathcal{X})$ and ${\rm Copres}(\X)$ are the subcategories of $\mathcal{A}$ consisting of (finite) products of objects from $\X$, $\mathcal{X}$-cogenerated objects and $\X$-copresented objects respectively. One then defines the \emph{reject} of 
$\mathcal{X}$ in an object $A$, denoted ${\rm rej}_{\X}(A)$, as the smallest of the subobjects $B$ of $A$ such that $A/B$ is $\mathcal{X}$-cogenerated. Note that $(1:{\rm rej}_{\X}))(A):=\frac{A}{{\rm rej}_{\X}(A)}$ is precisely the trace of $\mathcal{X}^{op}$ in $A$ within $\mathcal{A}^{op}$. Hence one has a dual of Lemma \ref{l:Existence_trace}, whose statement is left to the reader. Furthermore, when ${\rm rej}_{\X}(A)$ exists for all objects $A\in\mathcal{A}$, one has a descending chain of subfunctors of the identity ${\rm rej}_{\X}=:{\rm rej}_{\X}^1\supseteq ...\supseteq {\rm rej}_{\X}^n\supseteq...$, where ${\rm rej}_{\X}^n={\rm rej}_{\X}\circ\stackrel{n}{....}{\rm rej}_{\X}$, for all $n>0$. Then $ {\rm rej}_{\X}^n$ and $((1:{\rm tr}_\mathcal{X}^n))$ are dual concepts, for all $n>0$.

\section{A general criterion for reflectivity and coreflectivity} \label{sec.main theorem}


Let $\mathcal{C}$ be a category and $\mathcal{B}$ a coreflective subcategory of $\mathcal{C}$. It is well-known that $\mathcal{B}$ is precovering since, if $q$ is the right adjoint of the inclusion functor $i$ and $\nu$ is the counit, then $\nu_M:q(M) \rightarrow M$ is a  $\mathcal{B}$-coreflection (see \cite[Theorem IV.2]{MacLane98}, where the author uses the terminology \textit{universal arrow}), whence a $\mathcal{B}$-precover, for any object $M$ of $\mathcal{C}$. Hence every (co)reflective subcategory of an additive category $\mathcal{C}$ is preenveloping (precovering) in $\mathcal{C}$, but the converse is not true in general. For example, in \cite[Theorem A]{Holm}, it is shown that over a local Cohen-Macaulay ring $R$ with a dualizing module, the class $MCM(R)$ of finitely generated maximal Cohen-Macaulay modules is preenveloping in the abelian category ${\rm mod}(R)$ (which is neither complete, nor cocomplete). However, by \cite[Theorem B]{Holm}, the subcategory $MCM(R)\subseteq {\rm mod}(R)$ is reflective if and only if the Krull dimension of $R$ is less or equal than 2. On the other hand, the classes of injective, fp-injective and finite injective-dimensional right $R$-modules over a ring $R$ are always preenveloping (see \cite[Theorems 4.1.6. and 4.1.7]{GobelTrlifaj}), but they are reflective if and only if the ring is semisimple, von Neumann regular and has finite right global dimension, respectively \cite[Proposition 4.1]{RSdV}. Dually, the classes of projective, flat and finite projective-dimensional right $R$-modules are precovering for any ring $R$ (see \cite[Theorems 4.1.1 and 4.1.12]{GobelTrlifaj}), but they are coreflective if and only if $R$ is semisimple, von Neumann regular and has finite right global dimension, respectively \cite[Proposition 4.1]{RSdV}.

In the following results, we find the conditions to be added to a precovering (preenveloping) subcategory to be coreflective (reflective).

\begin{theorem}\label{t:CharacterizationCoreflective}
Let $\mathcal{C}$ be an additive category such that every morphism has a pseudokernel and a pseudocokernel. Let $\mathcal{B}$ be a subcategory of $\mathcal{C}$. Consider the following sentences:
\begin{enumerate}
\item $\mathcal{B}$ is a coreflective subcategory of $\mathcal{C}$.

\item $\mathcal{B}$ is precovering, closed under taking direct summands and every morphism in $\mathcal{B}$ has a pseudocokernel in $\mathcal{C}$ which belongs to $\mathcal{B}$.
\end{enumerate}
Then $(1)\Longrightarrow (2)$. If $\mathcal{C}$ has split idempotents, then $(2)\Longrightarrow (1)$ as well.
\end{theorem}

\begin{proof}
$(1)\Longrightarrow (2)$ We follow the notation of the preceding paragraph: $q$ is the right adjoint of the inclusion $i:\mathcal{B} \rightarrow \mathcal{C}$, $\nu$ is the counit and $\eta$ the unit, and recall that $\eta$ is a natural isomorphism. Then, $\mathcal{B}$ consists precisely of the objects $B\in\mathcal{C}$ such that $\nu_B:q(B)\rightarrow B$ is an isomorphism. From this and the fact that $i$ and $q$ are additive functors (see Subsection \ref{ss:reflective-coreflective}), it is clear that $\mathcal{B}$ is closed under taking direct summands. 

On the other hand, as we have seen before, $\mathcal{B}$ is precovering. Let now $f:B \rightarrow B'$ be a morphism in $\mathcal{B}$ and let us prove that $f$ has a pseudocokernel in $\mathcal{C}$ belonging to $\mathcal{B}$. Take $c:B' \rightarrow C$ a pseudocokernel of $f$ in $\mathcal{C}$, which exists by assumption. We claim that $q(c)\nu_{B'}^{-1}:B'\rightarrow q(C)$ is also a pseudocokernel of $f$ in $\mathcal{C}$, which will end the proof of the implication. Indeed,  we have that $q(c)\nu_{B'}^{-1} f=q(c) q(f)\nu_B^{-1}=q(c f)\nu_B^{-1}=0$, due to the naturality of $\nu$ and the previous paragraph.  Let $g:B'\rightarrow A$ be a morphism in $\mathcal{C}$ such that $g f=0$ and, by the pseudocokernel condition of $c$, choose $h:C \rightarrow A$ with $h c=g$. We then have that $g\nu_{B'}=\nu_A q(g)$, which implies that $g=\nu_A q(g)\nu_{B'}^{-1}=\nu_A q(h) q(c) \nu_{B'}^{-1},$ so that $g$ factors through $q(c) \nu_{B'}^{-1}$.

$(2)\Longrightarrow (1)$ By \cite[IV.Theorem 2]{MacLane98} we only have to check that each $A \in \mathcal{C}$ has a $\mathcal{B}$-coreflection.

Take, using the assumption, a $\mathcal{B}$-precover $f:B \rightarrow A$ of $A$. Let $k:K \rightarrow B$ be a pseudokernel of $f$ in $\mathcal{C}$, which exists by hypothesis, and fix a $\mathcal{B}$-precover of $K$, $g:B' \rightarrow K$. Moreover, let $c:B \rightarrow C$ be a pseudocokernel of $kg$ in $\mathcal{C}$ such that $C$ belongs to $\mathcal{B}$. We can construct the following commutative diagram:
\begin{displaymath}
\begin{tikzcd}
 & B' \arrow{r}{kg} \arrow{d}{g} & B \arrow[equal]{d} \arrow{r}{c} & C \arrow{r} \arrow{d}{h} & 0\\
 0 \arrow{r} & K \arrow{r}{k} & B \arrow{r}{f} & A & 
\end{tikzcd}
\end{displaymath}
Since $f$ is a $\mathcal{B}$-precover, there exists $a:C \rightarrow B$ with $fa=h$. Then $fac = f$, hence, $f(1_A-ac)=0$. Now we prove:
\begin{equation}\label{e:Fact}\tag{*}
D \in \mathcal{B}, t:D \rightarrow B, ft=0 \Rightarrow act=0.
\end{equation}
Let $t:D \rightarrow B$ be a morphism with $D \in \mathcal{B}$ and $ft=0$. Then there exists $b:D \rightarrow K$ with $kb=t$. Since $g$ is a $\mathcal{C}$-precover, there exists $d:D \rightarrow B'$ with $gd=b$. Then $act=ackb=ackgd=0$.

As a consequence of (\ref{e:Fact}), $ac(1-ac)=0$, that is, $e=ac$ is an idempotent. Now use that idempotents split in $\mathcal{C}$ to get a factorization of $e$, 
\begin{displaymath}
\begin{tikzcd}
B \arrow{r}{u} & E \arrow{r}{v} & B
\end{tikzcd}
\end{displaymath}
with $uv=1_E$. Since $\mathcal{B}$ is closed under taking direct summands, $E \in \mathcal{B}$. 

We claim that $fv$ is a $\mathcal B$-coreflection of $A$. First, notice that $fv$ is a $\mathcal{B}$-precover: given $t:D \rightarrow A$ with $D \in \mathcal{B}$, there exists $t':D \rightarrow B$ with $ft'=t$. Since $fe=f$, we actually have $fvut'=t$, which says that $fv$ is a $\mathcal{B}$-precover. Now, we have to see that $ut'$ is the unique morphism satisfying $fvut'=t$. Take any other morphism $s:D \rightarrow A$ with $fvs=t$. Then $fv(ut'-s)=0$ and, by (\ref{e:Fact}), $ev(ut'-s)=0$. This means that $v(ut'-s)=0$, so that $ut'-s=0$ since $v$ is a monomorphism.
\end{proof}

We have the dual result too:

\begin{theorem} \label{t:CharacterizationReflective}
Let $\mathcal{C}$ be an additive category such that every morphism has a pseudokernel and a pseudocokernel. Let $\mathcal{B}$ be a subcategory of $\mathcal{C}$. Consider the following sentences:
\begin{enumerate}
\item $\mathcal{B}$ is a reflective subcategory of $\mathcal{C}$.

\item $\mathcal{B}$ is preenveloping, closed under taking direct summands and every morphism in $\mathcal{B}$ has a pseudokernel in $\mathcal{C}$ which belongs to $\mathcal{B}$.
\end{enumerate}
Then $(1)\Longrightarrow (2)$. If $\mathcal{C}$ has split idempotents, then $(2)\Longrightarrow (1)$ as well.
\end{theorem}

A crucial hypothesis for the equivalences in our last two theorems is that the category $\mathcal{C}$ has split idempotents. For categories $\mathcal{C}$ not having this property, we can consider their idempotent completion $\hat{\mathcal{C}}$. We state now slight generalizations of Theorems \ref{t:CharacterizationCoreflective} and \ref{t:CharacterizationReflective} to the case in which $\mathcal{C}$ does not have split idempotents using this idempotent completion. But first we need a couple of lemmas.

\begin{lemma}\label{l:PseudokernelsIdempotentCompletion}
Let $\mathcal{C}$ be an additive category and $f:(A,e) \rightarrow (B,e')$ be a morphism in $\hat{\mathcal{C}}$. Then:
\begin{enumerate}
\item If $c:K \rightarrow A$ is a pseudokernel of $f$ in $\mathcal{C}$, then $ec:(K,1_K) \rightarrow (A,e)$ is a pseudokernel of $f$ in $\hat{\mathcal{C}}$.  Moreover, $\mathcal{C}$ has pseudokernels if and only if  so does $\hat{\mathcal{C}}$.

\item If $c:B \rightarrow K'$ is a pseudocokernel of $f$ in $\mathcal{C}$, then $ce':(B,e') \rightarrow (K',1_{K'})$ is a pseudocokernel of $f$ in $\hat{\mathcal{C}}$. Moreover, $\mathcal{C}$ has pseudocokernels if and only if so does $\hat{\mathcal{C}}$.
\end{enumerate}
\end{lemma}

\begin{proof}
We just need to prove (1), and (2) will follow by duality. Since $e$ is idempotent $ec$  is a morphism in $\hat{\mathcal{C}}$. Moreover, $fec=fc=0$. Now, given $g:(K',e'') \rightarrow (A,e)$ with $fg=0$, there exists $h:K' \rightarrow K$ with $ch=g$. Then $he'':(K',e'')\rightarrow (K,1_K)$ is a morphism in $\hat{\mathcal{C}}$ and $eche''=ege'' = g$ since $eg=g=ge''$. Then, $ec:(K,1) \rightarrow (A,e)$ is a pseudokernel of $f$ in $\hat{\mathcal{C}}$.  
 
The last paragraph also proves the `only if' part of the last statement. Conversely, if $\hat{\mathcal{C}}$ has pseudokernels and $f:A\rightarrow B$ is any morphism in $\mathcal{C}$, and we take a pseudokernel $c:(K,e)\rightarrow (A,1_A)$ of $f:(A,1_A)\rightarrow (B,1_B)$ in $\hat{\mathcal{C}}$, then $c:K\rightarrow A$ is a pseudokernel of $f$ in $\mathcal{C}$.
\end{proof}

In the rest of the section we allow a certain abuse of notation. Given a (not necessarily additive) subcategory $\mathcal{B}$ of $\mathcal{C}$, we shall denote by $\hat{\mathcal{B}}$ the subcategory of $\hat{\mathcal{C}}$ that consists of the objects $(B,e)$ such that $B\in\mathcal{B}$. We call it the idempotent completion of $\mathcal{B}$.

\begin{lemma}\label{l:PrecoveringIdempotentCompletion}
Let $\mathcal{C}$ be an additive category and $\mathcal{B}$ be a subcategory. The following hold:
\begin{enumerate}
\item  $\mathcal{B}$ is precovering in $\mathcal{C}$ if and only if so is $\hat{\mathcal{B}}$  in $\hat{\mathcal{C}}$.

\item $\mathcal{B}$ is preenveloping in $\mathcal{C}$ if and only if so is $\hat{\mathcal{B}}$ in $\hat{\mathcal{C}}$.
\end{enumerate}
\end{lemma}

\begin{proof} Assertion (2) follows from (1) by duality, so we just prove the latter. The `if' part of the statement is clear. For the `only if' part, take an object $(A,e)$ of $\hat{\mathcal{C}}$ and let $f:B \rightarrow A$ be a $\mathcal{B}$-precover of $A$ in $\mathcal{C}$. It is easy to see that $ef:(B,1_B)\rightarrow (A,e)$ is an $\hat{\mathcal{B}}$-precover in $\hat{\mathcal{C}}$. 
\end{proof}

Now we prove the aforementioned extensions of Theorems \ref{t:CharacterizationCoreflective} and \ref{t:CharacterizationReflective} to additive categories without split idempotents.

\begin{corollary}\label{c:CoreflectiveIdempotentCompletion}
Let $\mathcal{C}$ be an additive category with pseudokernels and pseudocokernels, and $\mathcal{B}$ be a subcategory of $\mathcal{C}$. Then, the following are equivalent:
\begin{enumerate}
\item The idempotent completion $\hat{\mathcal{B}}$ of $\mathcal{B}$ is a coreflective subcategory of $\hat{\mathcal{C}}$.

\item $\mathcal{B}$ is precovering and every morphism in $\mathcal{B}$ has a pseudocokernel in $\mathcal{C}$ that belongs to $\mathcal{B}$.
\end{enumerate}
\end{corollary}

\begin{proof}
First, notice that $\hat{\mathcal{C}}$ has pseudokernels and pseudocokernels by Lemma \ref{l:PseudokernelsIdempotentCompletion}, and has split idempotents by construction. By Theorem \ref{t:CharacterizationCoreflective}, $\hat{\mathcal{B}}$ is coreflective if and only if it is precovering, closed under taking direct summands and every morphism in $\hat{\mathcal{B}}$ has a pseudocokernel in $\hat{\mathcal{C}}$ which belongs to $\hat{\mathcal{B}}$. Since $\hat{\mathcal{B}}$ is always closed under taking direct summands, this is equivalent to (2) by Lemmas \ref{l:PseudokernelsIdempotentCompletion} and \ref{l:PrecoveringIdempotentCompletion}.
\end{proof}

\begin{corollary}\label{c:ReflectiveIdempotentCompletion}
Let $\mathcal{C}$ be an additive category with pseudokernels and pseudocokernels, and $\mathcal{B}$ be a subcategory of $\mathcal{C}$. Then, the following are equivalent:
\begin{enumerate}
\item The idempotent completion $\hat{\mathcal{B}}$ of $\mathcal{B}$ is a reflective subcategory of $\hat{\mathcal{C}}$.

\item $\mathcal{B}$ is preenveloping and every morphism in $\mathcal{B}$ has a pseudokernel in $\mathcal{C}$ that belongs to $\mathcal{B}$.
\end{enumerate}
\end{corollary}

\section{Applications to triangulated categories}

In this section, we apply our results to left and right triangulated categories. Recall that in a left triangulated category $\mathcal{D}$ we have a left triangulation $(\Omega,\Delta)$ satisfying the axioms \textbf{(LT1)}, \textbf{(LT2)}, \textbf{(LT3)} and \textbf{(LT4)} of, for instance, \cite[Definition 3.1]{Zhi-Wei}. Dually, in a right triangulated category $\mathcal{D}'$ there exists a right triangulation $(\Sigma,\nabla)$ satisfying the dual axioms \textbf{(RT1)}, \textbf{(RT2)}, \textbf{(RT3)} and \textbf{(RT4)}. A \textit{pretriangulated category}
is a left and right triangulated category in which both triangulations are compatible in the sense of \cite[Definition II.1.1]{BeligiannisReiten}. The typical examples of left/right triangulated categories are the stable categories of abelian categories with respect to a precovering/preenveloping subcategory, see Section 5. Notice that the notion of pretriangulated category given here is different to the one considered in \cite[Definition 1.1.1]{Neeman01}, since the left/right translation (or loop/suspension) functor $\Omega/\Sigma$ need not be invertible. 

Recall that if $f:X \rightarrow Y$ is a morphism in a right triangulated category $\mathcal{D}$, then a \textit{cone} of $f$ is an object $Z$ appearing in a right triangle
\begin{displaymath}
\begin{tikzcd}
X \arrow{r}{f} & Y \arrow{r} & Z \arrow{r} & \Sigma(X).
\end{tikzcd}
\end{displaymath}
Since cones are unique up to isomorphism, we will denote by $\cone(f)$ the cone of the morphism $f$.

\textit{Cocones} in left triangulated categories are defined dually.

The following fact is well known:

\begin{lemma}\label{l:ConePseudocokernel}
Let $\mathcal{D}$ be an additive category.
\begin{enumerate}
\item Suppose that $\mathcal{D}$ is left triangulated and let
\begin{displaymath}
\begin{tikzcd}
\Omega(Z) \arrow{r}{f} & X \arrow{r}{g} & Y \arrow{r}{h} & Z
\end{tikzcd}
\end{displaymath}
be a left triangle. Then $g$ is a pseudokernel of $h$ and $f$ is a pseudokernel of $g$.

\item Suppose that $\mathcal{D}$ is right triangulated and let
\begin{displaymath}
\begin{tikzcd}
X \arrow{r}{f} & Y \arrow{r}{g} & Z \arrow{r}{h} & \Sigma(X)
\end{tikzcd}
\end{displaymath}
be a right triangle. Then $g$ is a pseudocokernel of $f$ and $h$ is a pseudocokernel of $g$.
\end{enumerate}
\end{lemma}

We next give the relation of any pseudokernel (resp. pseudocokernel) of a morphism with its cocone (resp. cone). 

\begin{lemma} \label{lem.cones-versus-pseudocokernels}
Let $\mathcal{D}$ be an additive category and $f:X\rightarrow Y$ be a morphism in $\mathcal{D}$. The following assertions hold:

 \begin{enumerate}
 \item If $(\mathcal{D},\Sigma)$ is right triangulated and $c:Y\rightarrow C$ is a pseudocokernel of $f$, then $\text{cone}(f)$ is a  retract of $C\oplus\Sigma (X)$.
 \item  If $(\mathcal{D},\Omega)$ is left triangulated and $k:K\rightarrow X$ is a pseudokernel of $f$, then $\text{cocone}(f)$ is a  retract of $K\oplus\Omega (Y)$.
 \end{enumerate} 
\end{lemma}

\begin{proof}
We just prove (1), since assertion (2) follows from it by duality.  We put $C(f):=\text{cone}(f)$ and consider the associated right triangle $X\stackrel{f}{\rightarrow}Y\stackrel{g}{\rightarrow}C(f)\stackrel{h}{\rightarrow}\Sigma (X)$. Since both $c$ and $g$ are pseudocokernels of $f$, we have morphisms $u:C(f)\rightarrow C$ and $v:C\rightarrow C(f)$ such that $ug=c$ and $vc=g$. It then follows that $vug=g$, which, by Lemma \ref{l:ConePseudocokernel},  implies that $1_{C(f)}-vu$ factors through $h$. That is, there is a morphism $w:\Sigma (X)\rightarrow C(f)$ such that $1_{C(f)}-vu=wh$. Then we have $1_{C(f)}=\begin{pmatrix}v & w \end{pmatrix} \begin{pmatrix} u\\ h\end{pmatrix}$, which implies that $\begin{pmatrix} v & w\end{pmatrix}:C\oplus\Sigma (X)\rightarrow C(f)$ is a retraction. 
\end{proof}

\begin{corollary}\label{c:ReflectiveAndCoreflectiveInPretriangulated0}
Let $\mathcal{D}$ be an additive category and $\mathcal{Z}$ be a subcategory containing the zero object.
\begin{enumerate}
\item Suppose that $(\mathcal{D},\Sigma )$ is right triangulated and has pseudokernels. The following assertions are equivalent:

\begin{enumerate}
\item   $\hat{\mathcal{Z}}$ is coreflective in $\hat{\mathcal{D}}$, where $\hat{(?)}$ denotes the idempotent completion, and $\Sigma (Z)$ is a retract of an object of $\mathcal{Z}$, for all $Z\in\mathcal{Z}$.
\item $\mathcal{Z}$ is precovering in $\mathcal{D}$ and the cone of any morphism in $\mathcal{Z}$ is a retract in $\mathcal{D}$ of an object of $\mathcal{Z}$.
\end{enumerate}
 
\item Suppose that $(\mathcal{D},\Omega )$ is left triangulated and has pseudocokernels. The following assertions are equivalent:

\begin{enumerate}
\item   $\hat{\mathcal{Z}}$ is reflective in $\hat{\mathcal{D}}$ and $\Omega (Z)$ is a retract of an object of $\mathcal{Z}$, for all $Z\in\mathcal{Z}$.
\item $\mathcal{Z}$ is preenveloping in $\mathcal{D}$ and the cocone of any morphism in $\mathcal{Z}$ is a retract in $\mathcal{D}$ of an object of $\mathcal{Z}$.
 \end{enumerate}
\end{enumerate}
\end{corollary}

\begin{proof}
We just prove (1), since assertion (2) follows from it by duality. 

$(1.a)\Longrightarrow (1.b)$ By Corollary \ref{c:CoreflectiveIdempotentCompletion}, we know  that $\mathcal Z$ is precovering and each morphism $f:Z\rightarrow Z'$ has a pseudocokernel in $\mathcal{D}$ that belongs to $\mathcal{Z}$. Then, by Lemma \ref{lem.cones-versus-pseudocokernels}, we conclude that $\cone(f)$ is a retract of an object of $\mathcal{Z}$. 

$(1.b)\Longrightarrow (1.a)$ Using again  Corollary \ref{c:CoreflectiveIdempotentCompletion}, in order to prove that $\hat{\mathcal{Z}}$ is coreflective, it is enough to prove that each morphism $f:Z\rightarrow Z'$ in $\mathcal{Z}$ has a pseudocokernel in $\mathcal{D}$ that belongs to $\mathcal{Z}$. But, by assumption we have an object $Z''\in\mathcal{Z}$ together with morphisms $s:C(f)\rightarrow Z''$ and $t:Z''\rightarrow C(f)$ such that $ts=1_{C(f)}$, where $C(f)=\cone(f)$. It immediately follows that if $Z\stackrel{f}{\rightarrow}Z'\stackrel{g}{\rightarrow}C(f)\stackrel{h}{\rightarrow}\Sigma (Z)$ is the associated right triangle, then $sg:Z'\rightarrow Z''$ is a pseudocokernel of $f$ as desired.

On the other hand, for each $Z\in\mathcal{Z}$, we have the trivial right triangle $Z\stackrel{1_Z}{\rightarrow}Z\rightarrow 0\rightarrow\Sigma (Z)$, and the morphism $Z\rightarrow 0$ is a morphism in $\mathcal{Z}$ whose cone is $\Sigma (Z)$. 
\end{proof}

We then get:

\begin{corollary}\label{c:ReflectiveAndCoreflectiveInPretriangulated}
Let $\mathcal{D}$ be a category with split idempotents and $\mathcal{Z}$ be a subcategory of $\mathcal{D}$.
\begin{enumerate}
\item If $\mathcal{D}$ is right triangulated and has pseudokernels, then $\mathcal{Z}$ is coreflective and $\Sigma(\mathcal{Z}) \subseteq \mathcal{Z}$ if and only if $\mathcal{Z}$ is precovering, closed under taking direct summands and cones.

\item If $\mathcal{D}$ is left triangulated and has pseudocokernels, then $\mathcal{Z}$ is reflective and $\Omega(\mathcal{Z}) \subseteq \mathcal{Z}$ if and only if $\mathcal{Z}$ is preenveloping, closed under taking direct summands and cocones.
\end{enumerate}
\end{corollary}
\begin{proof}
Just note that, in any of the equivalent conditions of each assertion, the subcategory $\mathcal{Z}$ contains the zero object. The result immediately follows from Corollary \ref{c:ReflectiveAndCoreflectiveInPretriangulated0}.
\end{proof}

This result is, in particular, true for pretriangulated categories. We establish explicitly the corresponding result for triangulated categories:

\begin{corollary}\label{c:ReflectiveAndCoreflectiveInTriangulated}
Let $\mathcal{D}$ be a triangulated category with split idempotents and $\mathcal{Z}$ be a subcategory. Then:
\begin{enumerate}
\item $\mathcal{Z}$ is coreflective and $\Sigma(\mathcal{Z}) \subseteq \mathcal{Z}$ if and only if $\mathcal{Z}$ is precovering and closed under taking direct summands and cones.

\item $\mathcal{Z}$ is reflective and $\Sigma^{-1}(\mathcal{Z}) \subseteq \mathcal{Z}$ if and only if $\mathcal{Z}$ is preenveloping and closed under taking direct summands and cocones.
\end{enumerate}
\end{corollary}

\begin{remarks} \label{4.6}
\begin{enumerate}
\item Notice that Corollary \ref{c:ReflectiveAndCoreflectiveInTriangulated} is  more general than \cite[Proposition 1.4]{Neeman10} and \cite[Proposition 3.11]{SaorinStovicek}, since the subcategory $\mathcal{Z}$ need be neither triangulated nor suspended.

\item There are lots of triangulated categories without split idempotents. For instance, if $\mathcal{A}$ is any additive category without split idempotents then so is its bounded homotopy category $\mathcal K^b(\mathcal{A})$. Indeed, if $e\in\mathcal{A}(A,A)$  is a non-split idempotent, then the same is true for $e[0]\in\mathcal K^b(\mathcal{A})(A[0],A[0])$. 

On the other hand, any triangulated category in which every countable family of objects has a coproduct, has split idempotents (see \cite[Proposition 1.6.8]{Neeman01}).

\item The idempotent completion of any triangulated category $\D$ has a canonical structure of triangulated category (see \cite{BalmerSchlichting}).
\end{enumerate}
\end{remarks}

\section{Applications to stable categories}

Let $\mathcal C$ be an additive category and $\mathcal P$ an additive subcategory of $\mathcal C$. \emph{The stable category of $\mathcal{C}$ modulo $\mathcal{P}$}, denoted $\mathcal{C}/\mathcal{P}$ or simply $\underline{\mathcal{C}}$ when $\mathcal P$ is understood, has as objects those of $\mathcal{C}$ and,  for every pair of objects $A,B$ in $\mathcal{C}$, $\underline{\mathcal{C}}(A,B)$ is the quotient group of $\mathcal{C}(A,B)$ by the subgroup of all morphisms factorizing through an object of $\mathcal{P}$. For any object $A$ and morphism $f$ in $\mathcal{C}$, we denote by $\underline A$ and $\underline f$ the corresponding object and morphism in $\underline{\mathcal{C}}$.

All throughout this section $\mathcal{A}$ will be an abelian category and $\mathcal{P}$ will be an additive subcategory of $\mathcal{A}$. When $\mathcal{P}$ is precovering (resp. preenveloping) the category $\underline{\mathcal{A}}$ has a structure of left (resp. right) triangulated category (see \cite[p. 28]{BeligiannisReiten}). We refer the reader to this latter monograph for the details. For our purposes here, we need to know that the loop (resp. suspension) functor $\Omega:\underline{\mathcal{A}}\rightarrow\underline{\mathcal{A}}$ (resp. $\Sigma: \underline{\mathcal{A}}\rightarrow\underline{\mathcal{A}}$) acts on objects by taking any $\underline A$ to $\underline{\Omega(A)}$, where $\Omega (A)$ is the kernel (resp. cokernel $\Sigma (A)$) of a $\mathcal{P}$-precover $\xi_A:P_A\rightarrow A$ (resp. a $\mathcal{P}$-preenvelope $\lambda^A:A\rightarrow P^A$). Up to isomorphism in $\underline{\mathcal{A}}$, this definition does not depend on the $\mathcal{P}$-precover (resp. $\mathcal{P}$-preenvelope) that we choose.  Note that any morphism $\underline{f}:\underline{A}\rightarrow\underline{B}$ is also represented by the morphism $\begin{pmatrix} f & \xi_B \end{pmatrix}:A\oplus P_B\rightarrow B$ (resp. $\begin{pmatrix} f \\ \lambda_A \end{pmatrix}:A\rightarrow B\oplus P^A$), which is clearly a $\mathcal{P}$-epimorphism (resp. $\mathcal{P}$-monomorphism). Then, in order to identify the cocone (resp. cone) of $\underline{f}$, we can assume that $f$ is a $\mathcal{P}$-epimorphism (resp. $\mathcal{P}$-monomorphism), in which case we have that $\text{cocone}(\underline{f})=\underline{\Ker f}$ (resp. $\text{cone}(\underline{f})=\underline{\Coker f}$) and the associated left (resp. right) triangle is $\Omega (\underline{B})\rightarrow\underline{\Ker f}\stackrel{\underline{i}}{\rightarrow}\underline{A}\stackrel{\underline{f}}{\rightarrow}\underline{B}$ (resp. $\underline{A}\stackrel{\underline{f}}{\rightarrow}\underline{B}\stackrel{\underline{p}}{\rightarrow}\underline{\Coker f}\rightarrow\Sigma (\underline{A})$), where $i:\Ker f\rightarrowtail A$ is the inclusion (resp. $p:B\twoheadrightarrow\Coker f$ is the projection).

We need to extend the notions of resolving and coresolving subcategories, introduced in \cite{AuslanderBridger}.

\begin{definition} \label{def.weakly (co)resolving}
Let $\mathcal{A}$ be an abelian category and let $\mathcal{P}$ and $\mathcal{X}$ be full additive subcategories. We say that $\mathcal{X}$ is \emph{weakly $\mathcal{P}$-resolving} when $\mathcal{P}\subseteq\mathcal{X}$,  $\mathcal{P}$ is precovering in $\mathcal{X}$ and $\mathcal{X}$ is closed under taking kernels of $\mathcal{P}$-epimorphisms. We will say that  $\mathcal{X}$ is $\mathcal{P}$-resolving when, in addition, for any exact sequence \begin{displaymath}
\begin{tikzcd}
0 \arrow{r} & A \arrow{r}{f} & B \arrow{r}{g} & C
\end{tikzcd}
\end{displaymath}
where $g$ is a $\mathcal{P}$-epimorphism, if $A,C\in\mathcal{X}$ then $B\in\mathcal{X}$.
\end{definition}

Note that, when $\mathcal{A}$ has enough projectives and $\mathcal{P}$ is the subcategory of projective objects, `$\mathcal{P}$-epimorphism' is synonymous of `epimorphism'. In such case a  weakly $\mathcal{P}$-resolving class will be simply called \emph{weakly resolving}, and it is just an $\X$ that contains $\mathcal{P}$ and is closed under taking kernels of epimorphisms. The $\mathcal{P}$-resolving subcategories are exactly the resolving subcategories, as defined in \cite{AuslanderBridger} (see also \cite{GobelTrlifaj}). 

Of course, we can define the dual notions of \emph{weakly ($\mathcal{P}$)-coresolving} and  \emph{($\mathcal{P}$)-coresolving} subcategories. The precise statements are left to the reader.

\begin{lemma}\label{l:PrecoversInStable}
Let $\mathcal{A}$ be an abelian category and $\mathcal{P}$ be a full additive subcategory of $\mathcal{A}$. Moreover, let $\mathcal{X}$ be an additive subcategory of $\mathcal{A}$ containing $\mathcal{P}$.
\begin{enumerate}
\item If $\mathcal{X}$ is precovering (resp. preenveloping) in $\mathcal{A}$, then $\underline{\mathcal{X}}:=\mathcal{X}/\mathcal{P}$ is precovering (resp. preenveloping) in $\underline{\mathcal{A}}:=\mathcal{A}/\mathcal{P}$.

\item If $\mathcal{P}$ is precovering (resp. preenveloping) in $\mathcal{A}$ and $\underline{\mathcal{X}}$ is precovering (resp. preenveloping) in $\underline{\mathcal{A}}$, then $\mathcal{X}$ is precovering (resp. preenveloping) in $\mathcal{A}$.
\end{enumerate}
\end{lemma}

\begin{proof}
We just do the precovering case, the other one following by duality. 

(1) is trivial.

(2) Suppose that $\underline{\mathcal{X}}$ is precovering in $\underline{\mathcal{A}}$ and let $A$ be any object of $\mathcal{A}$. Take an $\underline{\mathcal{X}}$-precover $\underline f_A:\underline X \rightarrow \underline A$ in $\underline{\mathcal{A}}$ and fix a $\mathcal{P}$-precover $\xi_A:P_A\rightarrow A$. Given any morphism in $\mathcal{A}$, $g:Y \rightarrow A$ with $Y \in \mathcal{X}$, there exists a morphism $h:Y \rightarrow X$ such that $g-f_Ah$ factors through a $P\in\mathcal{P}$, which in turn implies that it factors through $\xi_A$. It then follows that $g$ factors through $\begin{pmatrix} f_A & \xi_A\end{pmatrix}:X_A\oplus P_A\rightarrow A$, so that this latter morphism is an $\mathcal{X}$-precover. 
\end{proof}

The easy proof of the following lemma is left to the reader. 

\begin{lemma} \label{lem.retract-stable-versus-directsummand}
Let $\mathcal{A}$ be an abelian subcategory, $\mathcal{P}$ be an additive subcategory and $A,B\in\mathcal{A}$ be any two objects. Then $\underline{A}$ is a retract of $\underline{B}$ in $\underline{\mathcal{A}}:=\mathcal{A}/\mathcal{P}$ if, and only if, $A$ is a direct summand of $B\oplus P$ in $\mathcal{A}$, for some $P\in\mathcal{P}$.  
\end{lemma}

We are now ready for the two (dual) main results of the section.

\begin{theorem} \label{t:ReflectiveStable}
Let $\mathcal{A}$ be an abelian category and let $\mathcal{P}\subseteq\mathcal{X}$ be two full additive subcategories, such that $\mathcal{P}$ is precovering and preenveloping and $\mathcal{X}$ is closed under taking direct summands. The following assertions are equivalent:

\begin{enumerate}
\item The idempotent completion of $\underline{\mathcal{X}}:=\mathcal{X}/\mathcal{P}$ is a reflective subcategory of the idempotent completion of $\underline{\mathcal{A}}:=\mathcal{A}/\mathcal{P}$,  and each $X\in\mathcal{X}$ admits a $\mathcal{P}$-precover with kernel in $\mathcal{X}$.
\item $\mathcal{X}$ is a preenveloping weakly $\mathcal{P}$-resolving subcategory of $\mathcal{A}$.
\item $\mathcal{X}$ is preenveloping and closed under taking kernels of $\mathcal{P}$-epimorphisms in $\mathcal{A}$.
\end{enumerate}
\end{theorem}

\begin{proof}
$(2)\Longleftrightarrow (3)$ is clear since $\mathcal{P}\subseteq\mathcal{X}$ and $\mathcal{P}$ is precovering. 

$(1)\Longleftrightarrow (3)$ Since, up to isomorphism in $\underline{\mathcal{X}}$, the definition of $\Omega (\underline{X})$ does not depend on the $\mathcal{P}$-precover of $X$ that we choose to define it, the condition that each $X\in\mathcal{X}$ admits a $\mathcal{P}$-precover with kernel in $\mathcal{X}$ is equivalent to saying that $\Omega (\underline{\mathcal{X}})\subseteq\underline{\mathcal{X}}$.   Bearing in mind that $\mathcal{X}$ is closed under taking direct summands, by Lemma \ref{lem.retract-stable-versus-directsummand} we have that  $\underline{\mathcal{X}}$ is closed under taking retracts in $\underline{\mathcal{A}}$. 
Then, by applying Corollary \ref{c:ReflectiveAndCoreflectiveInPretriangulated0}(2) with $\mathcal{D}=\underline{\mathcal{A}}$,  we conclude that assertion (1) holds  if, and only if, $\underline{\mathcal{X}}$ is preenveloping and closed under taking cocones in $\underline{\mathcal{A}}$. 

By Lemma \ref{l:PrecoversInStable}, we get that $\underline{\mathcal{X}}$ is preenveloping in $\underline{\mathcal{A}}$ if, and only if, so is $\mathcal{X}$ in $\mathcal{A}$. On the other hand, by the two initial paragraphs of this section, saying that $\underline{\mathcal{X}}$ is closed under taking cocones in $\underline{\mathcal{A}}$ is equivalent to saying that $\mathcal{X}$ is closed under taking kernels of $\mathcal{P}$-epimorphisms in $\mathcal{A}$. 
Therefore assertions (1) and (3) are equivalent. 
\end{proof}

\begin{theorem} \label{t:CoreflectiveStable}
Let $\mathcal{A}$ be an abelian category and let $\mathcal{P}\subseteq\mathcal{X}$ be two full additive subcategories of $\mathcal A$, such that $\mathcal{P}$ is precovering and preenveloping and $\mathcal{X}$ is closed under taking direct summands. The following assertions are equivalent:

\begin{enumerate}
\item The idempotent completion of $\underline{\mathcal{X}}:=\mathcal{X}/\mathcal{P}$ is a coreflective subcategory of the idempotent completion of $\underline{\mathcal{A}}:=\mathcal{A}/\mathcal{P}$,  and each $X\in\mathcal{X}$ admits a $\mathcal{P}$-preenvelope with cokernel in $\mathcal{X}$.
\item $\mathcal{X}$ is a precovering weakly $\mathcal{P}$-coresolving subcategory of $\mathcal{A}$.
\item $\mathcal{X}$ is precovering and closed under taking cokernels of $\mathcal{P}$-monomorphisms in $\mathcal{A}$. 
\end{enumerate}
\end{theorem}

As an immediate consequence, we  get:

\begin{corollary} \label{cor.bijection-Presolving-reflectives}
Let $\mathcal{A}$ be an abelian category and $\mathcal{P}$ be a precovering and preenveloping subcategory such that $\underline{\mathcal{A}}=\mathcal{A}/\mathcal{P}$ has split idempotents (see Proposition \ref{ex.stablecategs-splitidempotents} below). The assignment $\mathcal{X}\rightsquigarrow\underline{\mathcal{X}}$ defines:

\begin{enumerate}
\item A bijection between the preenveloping weakly $\mathcal{P}$-resolving subcategories $\mathcal{X}$ of $\mathcal{A}$ closed under  direct summands and the  reflective subcategories  $\underline{\mathcal{X}}$ of $\underline{\mathcal{A}}$ such that $\Omega (\underline{\mathcal{X}})\subseteq\underline{\mathcal{X}}$.
\item A bijection between the precovering weakly $\mathcal{P}$-coresolving subcategories $\mathcal{X}$ of $\mathcal{A}$ closed under  direct summands and the  coreflective subcategories  $\underline{\mathcal{X}}$ of $\underline{\mathcal{A}}$ such that $\Sigma (\underline{\mathcal{X}})\subseteq\underline{\mathcal{X}}$.
\end{enumerate}
\end{corollary}

\begin{remark}
If in the bijection stated in (1) of the last corollary we assume that $\mathcal{A}$ has enough projectives and that  $\mathcal{P}:=\Proj(\mathcal{A})$ is preenveloping, then we get a  bijection between the preenveloping weakly resolving subcategories closed under direct summands  of $\mathcal{A}$ and the coreflective subcategories $\mathcal{Z}$ of $\underline{\mathcal{A}}:=\mathcal{A}/\Proj(\mathcal{A})$ which are closed under taking syzygies.  See Examples \ref{ex.Chase-situation} and \ref{ex.locally noetherian} for  instances of these situations. 

The dual result involving injective objects is true as well.
\end{remark}

This corollary can be pushed a little bit further when $\mathcal{A}$ is Frobenius and $\mathcal{P}$ is its subcategory of projective objects, since, in this case, the stable category is triangulated \cite[I \S2]{Happel}. Recall that a \textit{t-structure} in a triangulated category $\mathcal D$ \cite[p. 17]{BeligiannisReiten} is a pair $(\mathcal D^{\leq 0},\mathcal D^{\geq 0})$ of full subcategories of $\mathcal D$ such that, setting $\mathcal D^{\leq n}=\Sigma^{-n}(\mathcal D^{\leq 0})$ and $\mathcal D^{\geq n}=\Sigma^{-n}(\mathcal D^{\geq 0})$, for all $n \in \mathbb Z$, satisfies:
\begin{enumerate}
\item $\mathcal D(D,D')=0$ for any $D \in \mathcal D^{\leq 0}$ and $D' \in \mathcal D^{\geq 1}$.

\item $\mathcal D^{\leq 0} \subseteq \mathcal D^{\leq 1}$ and $\mathcal D^{\geq 1} \subseteq \mathcal D^{\geq 0}$.

\item For any object $D \in \mathcal D$, there exists a triangle
$D^{\leq 0} \rightarrow D \rightarrow D^{\geq 1} \rightarrow \Sigma(D^{\leq 0})$ with $D^{\leq 0} \in \mathcal D^{\leq 0}$ and $D^{\geq 1} \in \mathcal D^{\geq 1}$.
\end{enumerate}

\begin{corollary} \label{cor.bijections in Frobenius abelian}
Let $\mathcal{A}$ be a Frobenius abelian category and $\mathcal{P}=\Proj(\mathcal{A})$ be its subcategory of projective (=injective) objects. Suppose that $\underline{\mathcal{A}}=\mathcal{A}/\mathcal{P}$ has split idempotents (see Proposition \ref{ex.stablecategs-splitidempotents} below).
 For a subcategory $\mathcal{X}\subseteq\mathcal{A}$, the following assertions hold:

\begin{enumerate}
\item $\mathcal{X}$ is a preenveloping resolving subcategory of $\mathcal{A}$ closed under direct summands if, and only if, there is a complete cotorsion pair $(\mathcal{Z},\mathcal{X})$ in $\mathcal{A}$ such that $\Omega (\underline{\mathcal{X}})\subseteq\underline{\mathcal{X}}$.
\item $\mathcal{X}$ is a precovering coresolving subcategory of $\mathcal{A}$ closed under direct summands if, and only if, there is a complete cotorsion pair $(\mathcal{X},\mathcal{Y})$ in $\mathcal{A}$ such that $\Omega^{-1}(\underline{\mathcal{X}})\subseteq\underline{\mathcal{X}}$.
\end{enumerate}
In particular we have a bijection between:

\begin{enumerate}
\item[(a)] Precovering coresolving subcategories of $\mathcal{A}$ closed under direct summands;
\item[(b)] t-structures $(\mathcal{U},\Sigma (\mathcal{V}))$ in $\underline{\mathcal{A}}$;
\item[(c)] Preenveloping resolving subcategories of $\mathcal{A}$ closed under direct summands;
\item[(d)] Complete cotorsion pairs $(\mathcal{Z},\mathcal{W})$ such that  $\Omega^{-1}(\underline{\mathcal{Z}})\subseteq\underline{\mathcal{Z}}$ (or, equivalently, such that  $\Omega (\underline{\mathcal{W}})\subseteq\underline{\mathcal{W}}$).
\end{enumerate}
\end{corollary}

\begin{proof}
We just prove (2), since assertion (1) will follow by duality. It easily follows from  Theorem \ref{t:CoreflectiveStable} that $\mathcal{X}$ is a  precovering coresolving subcategory of $\mathcal{A}$ if, and only if, $\underline{\mathcal{X}}$ is a coreflective subcategory closed under taking extensions of $\underline{\mathcal{A}}$ such that $\Omega^{-1}(\underline{\mathcal{X}})\subseteq\underline{\mathcal{X}}$. But, bearing in mind that the suspension functor of the triangulated category $\underline{\mathcal{A}}$ is $\Omega^{-1}:\underline{\mathcal{A}}\rightarrow\underline{\mathcal{A}}$, the latter just says that $\underline{\mathcal{X}}$ is the aisle of a t-structure in $\underline{\mathcal{A}}$ (see \cite[Proposition 3.11]{SaorinStovicek}). That is, $\mathcal{X}$ is precovering coresolving in $\mathcal{A}$ if, and only if, $\underline{\mathcal{X}}$ is the aisle of a t-structure in $\underline{\mathcal{A}}$. But, by \cite[Proposition 3.9]{SaorinStovicek} and the fact that $\Omega$ is a self-equivalence of $\underline{\mathcal{A}}$, we get that $\underline{\mathcal{X}}$ is the aisle of a t-structure if, and only if, $\mathcal{X}$ is the first component of a complete cotorsion pair such that  $\Omega^{-1}(\underline{\mathcal{X}})\subseteq\underline{\mathcal{X}}$. 

Since a t-structure $(\mathcal{U},\Sigma (\mathcal{V}))$ is determined either by $\mathcal{U}$ or by $\mathcal{V}$, the bijections between (a) and (b) and between (b) and (c) are restrictions of the bijections in Corollary \ref{cor.bijection-Presolving-reflectives}. The bijection between (b) and (d), when taking the $\Omega^{-1}$-closed version of the latter, follows from \cite[Proposition 3.9]{SaorinStovicek}. But note that we have bifunctorial isomorphisms $\underline{\mathcal{A}}(\Omega (\underline{A}),\underline{B})\cong\Ext_{\mathcal{A}}^1(A,B)\cong\underline{\mathcal{A}}(\underline{A},\Omega^{-1}(\underline{B }))$, for all $A,B\in\mathcal{A}$. We then get the following chain of isomorphisms, for all $X,Y\in\mathcal{A}$: $$\Ext_{\mathcal{A}}^1(X,\Omega (Y))\cong\underline{\mathcal{A}}(\Omega (\underline{X}),\Omega (\underline{Y}))\cong\underline{\mathcal{A}}(X,Y)\cong\underline{\mathcal{A}}(\Omega (\Omega^{-1}(\underline{X})),\underline{Y})\cong\Ext_{\mathcal{A}}^1(\Omega^{-1}(X),Y).$$ This implies that a complete cotorsion pair $(\mathcal{X},\mathcal{Y})$ satisfies that $\Omega^{-1}(\mathcal{X})\subseteq\mathcal{X}$ if, and only if, the inclusion  $\Omega (\mathcal{Y})\subseteq\mathcal{Y}$ also holds.
\end{proof}

Now let us see that, in certain natural situations, stable categories have split idempotents.

\begin{proposition} \label{ex.stablecategs-splitidempotents}
Let $\mathcal{A}$ be an additive category and let $\mathcal{P}$ be a full additive subcategory of $\mathcal{A}$ that, without loss of generality, we assume to be closed under taking direct summands. Assume that one of the following conditions holds:
\begin{enumerate}
\item $\mathcal{A}$ is Krull-Schmidt.
\item $\mathcal{A}$ is bicomplete abelian and $\mathcal{P}$ is precovering and preenveloping. 
\end{enumerate}
Then the category $\underline{\mathcal{A}}=\mathcal{A}/\mathcal{P}$ has split idempotents.
\end{proposition}
 
\begin{proof} 
(1) If $\mathcal{A}$ is Krull-Schmidt, $A=A_1\oplus \cdots \oplus A_n$ is the decomposition of an object as a direct sum of indecomposable objects and we view the endomorphisms of $A$ as matrices $f=(f_{ij})$, where $f_{ij}:A_j\rightarrow A_i$ is a morphism in $\mathcal{A}$, for all $i,j \in \{1, \ldots, n\}$, then the Jacobson radical of $\End_{\mathcal{A}}(A)$ consists of those $f=(f_{ij})$ such that none of the $f_{ij}$ is an isomorphism. Moreover, $\End_{\mathcal{A}}(A)$ is a semiperfect ring, for each $A\in\mathcal{A}$ (see \cite{Krause}). 
 
On the other hand,  each object $A$ of $\mathcal A$ admits a decomposition $A=P\oplus A'$, where $P\in\mathcal{P}$ and $A'$ has no nonzero direct summand in $\mathcal{P}$. We then have that $\underline{A}\cong\underline{A'}$ in $\underline{\mathcal{A}}$ so that, in order to check the idempotent completion, we can assume that $A$ has no nonzero direct summand 
in $\mathcal{P}$, something that we assume from now on in this proof. We then consider the canonical ring homomorphism $\pi:R_A:=\End_{\mathcal{A}}(A)\rightarrow\End_{\underline{\mathcal{A}}}(A):=\underline{R}_A$, that takes $f\rightsquigarrow\underline{f}$. If $\pi (f)=0$, then $f$ belongs to the Jacobson radical $J(R_A)$ of $R_A$, due to the previous paragraph. In particular, $\pi$ is the projective cover of $\underline{R}_A$ in $\text{Mod-}R_A$.  Consider now an $f\in R_A$ such that $\underline{f}=\underline{f}^2$ in $\underline{R}_A$. Then we have a decomposition $\underline{R}_A=\underline{f}\underline{R}_A\oplus (\underline{1}_A-\underline{f})\underline{R}_A$, both in the category $\text{Mod-}\underline{R}_A$ and in the category $\text{Mod-}R_A$. Since $R_A$ is semiperfect we can apply \cite[Theorem 11.2.2]{Kasch} to get an idempotent $e\in R_A$, such that $\underline{e}\underline{R}_A=\pi (e R_A)=\underline{f}\underline{R}_A$ and $(\underline{1}_A-\underline{e})\underline{R}_A=\pi ((1_A-e)R_A)=(\underline{1}_A-\underline{f})\underline{R}_A$. It immediately follows that $\underline{e}=\underline{f}$. Since $\mathcal A$ has split idempotents, we get an object $X$ and morphisms $s:X\rightarrow A$ and $p:A\rightarrow X$ such that $e=sp$ and $1_X=ps$. By applying $\pi$ to these last equalities, we conclude that the idempotent $\underline{f}=\underline{e}$ splits in $\underline{\mathcal{A}}$. 

(2) Suppose that $\mathcal{A}$ is bicomplete abelian and $\mathcal{P}$ is precovering and preenveloping. Then $\mathcal{P}$ is closed under taking arbitrary products and coproducts, since it is preenveloping and precovering and closed under taking direct summands. It is well-known that the projection $\Pi :\mathcal{A}\rightarrow\underline{\mathcal{A}}$ preserves products and coproducts, so that $\underline{\mathcal{A}}$ has products and coproducts. Our goal is to adapt the proof of \cite[Proposition 1.6.8]{Neeman01} in our particular case. Neeman's proof is ultimately based on two facts: i) Coproducts of triangles are triangles; ii) The homotopy colimit of the sequence 
\begin{equation}\label{e:homseq}\tag{\#}
\begin{tikzcd}
X \arrow{r}{1_X} & X \arrow{r}{1_X} & \cdots
\end{tikzcd}
\end{equation}
is isomorphic to $X$.  We will consider the right triangulated structure of $\underline{\mathcal{A}}$. That a coproduct  of right triangles is a triangle is a direct consequence of the definition of the right triangulated structure in $\underline{\mathcal{A}}$ (see the initial paragraphs of this section), bearing in mind that  the coproduct is  right exact and that a  coproduct of $\mathcal{P}$-monomorphisms is obviously a $\mathcal{P}$-monomorphism. 

On the other hand, if $X\in\mathcal{A}$ and we consider the map $1-shift:X^{(\mathbb{N})}\rightarrow X^{(\mathbb{N})}$, but in the abelian category $\mathcal{A}$ for the moment, we claim that this map is a section, for which we just need to check that the map $\mathcal{A}(1-shift,A): \mathcal{A}(X^{(\mathbb N)},A)\rightarrow\mathcal{A}(X^{(\mathbb N)},A)$ is surjective. Indeed, if $A\in\mathcal A$ and we apply $\mathcal{A}(?,A)$ to $1-shift$, we get a morphism of abelian groups $\varphi:\mathcal{A}(X,A)^\mathbb{N}\cong\mathcal{A}(X^{(\mathbb N)},A)\rightarrow\mathcal{A}(X^{(\mathbb N)},A)\cong\mathcal{A}(X,A)^\mathbb{N}$ that it is easily seen to take $(f_n)_{n\in\mathbb{N}}$ to $(f_0-f_1,f_1-f_2,...,f_n-f_{n+1},...)$. This morphism is surjective since, given $(g_n)_{n\in\mathbb{N}}\in\mathcal{A}(X,A)^\mathbb{N}$, we can choose $(f'_n)_{n\in\mathbb{N}}$ by defining $f'_0=0$ and $f'_n=-\Sigma_{i=0}^{n-1}g_i$, for all $n>0$, to get an element $(f'_n)_{n \in \mathbb N}$ such that $\varphi ((f'_n)_{n \in \mathbb N})=(g_n)_{n \in \mathbb N}$. 

Now, using this claim, $1-shift$ is in particular a $\mathcal{P}$-monomorphism, so that it induces the following right triangle in $\underline{\mathcal{A}}$: $$\underline{X}^{(\mathbb{N})}\stackrel{1-shift}{\rightarrow}\underline{X}^{(\mathbb{N})}\rightarrow\underline{\Coker(1-shift)}\rightarrow\Sigma (\underline{X}^{(\mathbb{N})}).$$ But it is quite easy to see that the summation or codiagonal map $\nabla :X^{(\mathbb{N})}\rightarrow X$  is a cokernel of $1-shift$ in $\mathcal{A}$. This implies that the homotopy colimit of (\ref{e:homseq}) is $X$.

After having checked that the obvious adaptations of conditions i) and ii) of Neeman's proof are also true in our case, one can copy it mutatis mutandis to conclude that idempotents split in $\underline{\mathcal{A}}$.
\end{proof}


Now we give some examples involving stable categories to which the last results apply. 

\begin{example}
Let $K$ be a commutative ring and $\mathcal{A}$ a $\mathrm{Hom}$-finite abelian $K$-category (that is, an additive category such that $\mathcal{A}(A,B)$ is a finitely generated $K$-module, for all $A,B\in\mathcal{A}$, and the composition of morphisms is $K$-bilinear). For any object $P$ in $\mathcal{A}$, the subcategory $\mathcal{P}:=\mathrm{sum}(P)$ is both precovering and preenveloping in $\mathcal{A}$. Therefore, Theorems \ref{t:ReflectiveStable} and \ref{t:CoreflectiveStable} apply, and also Corollary \ref{cor.bijection-Presolving-reflectives} when $\mathcal{A}$ is, in addition, Krull-Schmidt. 

If $K$ is also noetherian, for any $K$-algebra $A$ that is finitely generated as a $K$-module (these algebras are called \emph{Noether algebras}), the category 
 $\mathcal{A}=\mathrm{mod}\textrm{-}A$ of finitely generated $A$-modules is $\mathrm{Hom}$-finite.  In the particular cases when $K$ is artinian (and hence $A$ is an \emph{Artin algebra}) or $K$ is a complete noetherian local domain (see \cite{Swan}) $\text{mod-}A$ is also a Krull-Schmidt category.

Another example of $\mathrm{Hom}$-finite Krull-Schmidt category appears when $K$ is a field and we take $\mathcal{A}:=\mathrm{coh}(\mathbb{X})$ to be the category of coherent sheaves over a smooth projective variety $\mathbb{X}$ (see \cite[Section 3]{BondalKapranov}). 
\end{example}

\begin{example} \label{ex.locally noetherian}
Let $\mathcal{G}$ be a locally noetherian Grothendieck category, i.e. a Grothendieck category with a set of noetherian generators. Such a category is characterized by the fact that the subcategory $\Inj(\mathcal{G})$ of injective objects is closed under taking arbitrary coproducts or, equivalently, by the fact that each injective object is a coproduct of indecomposable objects (see \cite[Theorems 1 and 2]{Roos}). In addition, in this case the class $\Inj(\mathcal{G})$ is precovering (actually covering). It is precovering because 
$\Inj(\mathcal{G})=\mathrm{Sum}(\mathcal{E})$, where $\mathcal{E}$ is a set of representatives of  isomorphism classes of indecomposable injective objects, and covering by \cite[Theorem 1.2]{ElB}. 
By Proposition \ref{ex.stablecategs-splitidempotents} above, we know that for $\mathcal{A}=\mathcal{G}$ and $\mathcal{P}=\Inj(\mathcal{G})$, the corresponding stable category $\underline{\mathcal{A}}$ has split idempotents, so that  Corollary \ref{cor.bijection-Presolving-reflectives} applies.

Examples of this situation are the module category $\mathcal{G}=\mathrm{Mod}\textrm{-}R$, where $R$ is a right noetherian ring, or the category $\mathrm{Qcoh}(\mathbb{X})$ of quasi-coherent sheaves over a noetherian scheme $\mathbb{X}$.  
\end{example}

\begin{example} \label{ex.Chase-situation}
By \cite[Corollary 3.5]{RadaSaorin}, given a ring $R$,  the class $\Proj(R)$ of projective right $R$-modules is preenveloping (and precovering) if, and only if, $R$ is a ring which is right perfect and left coherent. By Proposition \ref{ex.stablecategs-splitidempotents} again, we  get that Corollary \ref{cor.bijection-Presolving-reflectives} applies in this case to $\mathcal{A}=\mathrm{Mod}\textrm{-}R$ and $\mathcal{P}=\Proj(R)$. 
\end{example}

Recall that in any Grothendieck category $\mathcal{G}$, an object  $Y$ is called \emph{Gorenstein injective} when there is an acyclic (=exact) complex $E^\bullet$ of injective objects such that $Y=Z^0(E^\bullet)$ is the object of $0$-cocycles and the complex $\mathcal{G}(E,E^\bullet )$ is acyclic in $\text{Ab}$, for all $E\in\Inj(\mathcal{G})$. We denote by $\text{GInj}(\mathcal{G})$ the subcategory of Gorenstein injective objects, and we obviously have $\Inj(\mathcal{G})\subseteq\text{GInj}(\mathcal{G})$. The following is now a consequence of our results:

\begin{corollary} \label{c:Gorenstein-injectives}
Let $\mathcal{G}$ be a locally noetherian Grothendieck category an let us consider stable categories modulo $\mathcal{P}=\Inj(\mathcal{G})$. The following assertions are equivalent:

\begin{enumerate}
\item $\text{\rm GInj}(\mathcal{G})$ is a preenveloping subcategory of $\mathcal{G}$;
\item $\underline{\text{\rm GInj}(\mathcal{G})}$ is a reflective subcategory of $\underline{\mathcal{G}}$.
\end{enumerate}

When the derived category $\mathbf{D}(\mathcal{G})$ is compactly generated (e.g. when $\mathcal{G}=\text{\rm Mod-}R$ for a right noetherian ring $R$), these equivalent conditions hold.
\end{corollary}
\begin{proof}
$(2)\Longrightarrow (1)$ is a consequence of Lemma \ref{l:PrecoversInStable} and Theorem \ref{t:CharacterizationReflective}.

$(1)\Longrightarrow (2)$ It is well-known that   $\text{GInj}(\mathcal{G})$ is closed under taking extensions and direct summands (see \cite[Theorem 2.6]{Holm} for a proof in module categories, which works in any Grothendieck category). We will prove that it is also closed under taking kernels of $\mathcal{P}$-epimorphisms, so that it will be weakly $\mathcal{P}$-resolving and, by Corollary \ref{cor.bijection-Presolving-reflectives}, the proof will be completed.

Let $f:X\rightarrow Y$ be a $\mathcal{P}$-epimorphism, where $X,Y\in\text{GInj}(\mathcal{G})$. Consider an acyclic complex of injectives $E^\bullet$ such that $Z^0(E^\bullet)\cong Y$ and $\mathcal{G}(E,E^\bullet )$ is acyclic, for all $E\in\Inj(\mathcal{G})$. The projection $\pi:E_Y:=E^{-1}\twoheadrightarrow Y$ factors through $f$, which implies that $f$ is an epimorphism. Taking the pullback of $f$ and $\pi$ and using the properties of pullbacks, we get two exact sequences (*):
$$0\rightarrow\text{Ker}(f)\rightarrow\hat{X}\rightarrow E_Y\rightarrow 0$$
$$0\rightarrow\text{Ker}(\pi)\rightarrow\hat{X}\rightarrow X\rightarrow 0,$$
where $\hat{X}$ is the upper left corner of the corresponding cartesian square. Since $\text{Ker}(\pi)=Z^{-1}(E^\bullet )$ is also Gorenstein injective, the second sequence gives that $\hat{X}\in\text{GInj}(\mathcal{G})$. On the other hand, since $\pi$ factors through $f$ the first of the two sequences (*) splits. Then $\text{Ker}(f)$ is in  $\text{GInj}(\mathcal{G})$, as it is a direct summand of $\hat{X}$.

The last statement follows from \cite{Krause05} (see also \cite[Proposition 4.2]{EnochsIacob} for the module case).
\end{proof}

Notice that $\textrm{GInj}(\Modr R)$ is preenveloping for every ring $R$ by \cite[Theorem 5.6]{SarochStovicek}. It is not known if this result is true for any Grothendieck category.

\section{Applications to exactly definable categories}

For our purposes in this section, we will need the following elementary result.

\begin{lemma} \label{lem.passing of p(co)kernels to subcategories}
Let $F:\mathcal{C}\to\mathcal{D}$ be a fully faithful additive functor between additive categories and denote by ${\rm Im}(F)$ its essential image. The following assertions hold:

\begin{enumerate}
\item If $\mathcal{D}$ has split idempotents and $\mathrm{Im}(F)$ is closed under taking direct summands in $\mathcal{D}$, then $\mathcal{C}$ has split idempotents.
\item If $\mathcal{D}$ has pseudocokernels and $\mathrm{Im}(F)$ is preenveloping in $\mathcal{D}$, then $\mathcal{C}$ has pseudocokernels.
\item If $\mathcal{D}$ has pseudokernels and $\mathrm{Im}(F)$ is precovering in $\mathcal{D}$, then $\mathcal{C}$ has pseudokernels.
\end{enumerate}
\end{lemma}
\begin{proof}
Assertion (1) is clear, and assertion (3) follows from assertion (2) by duality.  We then prove (2). Let $f:C\to C'$ be a morphism in $\mathcal{C}$ and let $c:F(C')\to D$ be a pseudocokernel of $F(f)$ in $\mathcal{D}$. We then fix an $\text{Im}(F)$-preenvelope $u:D\to F(C'')$. 
By the fully faithful condition of $F$, we get a unique morphism $g:C\to C''$ such that $F(g)=uc$. It immediately follows that $g$ is a pseudocokernel of $f$ in $\mathcal{C}$. 
\end{proof}

Recall that an object $Y$ of a locally coherent Grothendieck category $\mathcal{G}$ is \emph{FP-injective} when $\Ext_\mathcal{G}^1(?,Y)$ vanishes on the subcategory $\text{fp}(\mathcal{G})$ of finitely presented objects. Recall also that  the additive categories satisfying any of the equivalent conditions of the next result (see \cite[Introduction and Proposition 2.2]{K98}) are called \emph{exactly definable}. In particular, every finitely accessible additive category with products is exactly definable (e.g., see \cite[Proposition 11.1]{Prest}). 

\begin{proposition} \label{prop.Krause}
Let $\mathcal{C}$ be an additive category. The following assertions are equivalent:

\begin{enumerate}
\item $\mathcal{C}$ is equivalent to the category $\mathrm{Ex}(\mathcal{A}^{\rm op},\Ab)$ of exact  functors $\mathcal{A}^{\rm op}\to\Ab$, for some (skeletally) small abelian category $\mathcal{A}$;
\item $\mathcal{C}$ is equivalent to the subcategory of FP-injective objects of a locally coherent Grothendieck category $\mathcal{G}$. 
\end{enumerate}
\end{proposition}

We then get:

\begin{proposition} \label{prop:sipp} Every exactly definable additive category  has split idempotents, pseudocokernels and pseudokernels. 
\end{proposition}
\begin{proof}
Let $\mathcal{C}$ be an exactly definable category. By Proposition \ref{prop.Krause}, we have a fully faithful functor $F:\mathcal{C}\to\mathcal{G}$, where $\mathcal{G}$ is a locally coherent Grothendieck category and $\text{Im}(F)=\mathrm{FPInj}(\mathcal{G})$ is the subcategory of FP-injective objects of $\mathcal{G}$. By  \cite[Corollaries ~3.5 and 4.2]{CPT}, $\mathrm{FPInj}(\mathcal{G})$ is  (pre)covering and preenveloping in $\mathcal{G}$, and it is also closed under taking direct summands. The result then follows from Lemma \ref{lem.passing of p(co)kernels to subcategories}.
\end{proof}

As a consequence, our Theorems \ref{t:CharacterizationCoreflective} and \ref{t:CharacterizationReflective} directly apply now to subcategories of exactly definable subcategories. 

\begin{corollary} \label{c:pureepi} Let $\mathcal{C}$ be an exactly definable additive category and let $\mathcal{B}$ be a subcategory of $\mathcal{C}$. The following assertions are equivalent:
\begin{enumerate}
\item $\mathcal{B}$ is a coreflective subcategory of $\mathcal{C}$.
\item $\mathcal{B}$ is precovering, closed under taking direct summands and every morphism in $\mathcal{B}$ has a pseudocokernel in $\mathcal{C}$ which belongs to $\mathcal{B}$.
\end{enumerate} 
\end{corollary}





\begin{corollary} \label{c:puremono} Let $\mathcal{C}$ be an exactly definable additive category and let $\mathcal{B}$ be a subcategory of $\mathcal{C}$. The following assertions are equivalent:
\begin{enumerate}
\item $\mathcal{B}$ is a reflective subcategory of $\mathcal{C}$.
\item $\mathcal{B}$ is preenveloping, closed under taking direct summands and every morphism in $\mathcal{B}$ has a pseudokernel in $\mathcal{C}$ which belongs to $\mathcal{B}$.
\end{enumerate} 
\end{corollary}




\section{Applications to preabelian and abelian categories}

In this section we apply the results of  Section \ref{sec.main theorem} to preabelian and abelian categories. The following lemma is well-known (see { \cite[Proposition 6.5.4]{Borceux}).
 
\begin{lemma} \label{lem.cokernels-implies-splitidempotents}
Let $\mathcal{C}$ be an additive category with cokernels (resp. kernels). Then $\mathcal{C}$ has split idempotents.
\end{lemma}

\begin{corollary}\label{c:AbelianCategories}
Let $\mathcal{A}$ be a preabelian category and $\mathcal{B}$ be an additive subcategory of $\mathcal{A}$. Then:
\begin{enumerate}
\item $\mathcal{B}$ is coreflective if and only if it is precovering and closed under taking cokernels. 

\item $\mathcal{B}$ is reflective if and only if it is preenveloping and closed under taking kernels.
\end{enumerate}
\end{corollary}

\begin{proof}
We just prove assertion (1), since (2) follows by duality. By Lemma \ref{lem.cokernels-implies-splitidempotents}, $\mathcal{A}$ has split idempotents. Moreover, direct summands are clearly cokernels. According to Theorem \ref{t:CharacterizationCoreflective}, it will be enough to prove that if $f:B\rightarrow B'$ is a morphism in $\mathcal{B}$ and $c:B'\rightarrow B''$ is a pseudocokernel in $\mathcal{A}$, with $B''\in\mathcal{B}$, then $\Coker f$ is a direct summand of $B''$. Indeed, if $p:B'\rightarrow\Coker f$ is the cokernel map of $f$ in $\mathcal{A}$, then there is a unique morphism $g:\Coker f\rightarrow B''$ such that $gp=c$ and, by the pseudocokernel condition of $c$, there is a morphism $h:B''\rightarrow\Coker f$ such that $hc=p$. It then follows that $hgp=hc=p$, which in turn implies that $hg=1_{\Coker f}$. Then $\Coker f$ is a retract of $B''$, and hence a direct summand since $\mathcal{A}$ has split idempotents.
\end{proof}


Now we study when the coreflective subcategory $\mathcal B$ of $\mathcal A$ is abelian and inherits the AB-conditions from $\mathcal A$.

\begin{proposition}\label{prop:exactness of colimits}
Let $\mathcal{A}$ be a cocomplete preabelian category and $\mathcal{B}$ be a coreflective subcategory of $\mathcal{A}$, in which case $\mathcal{B}$ is also  cocomplete preabelian (see Subsection \ref{ss:reflective-coreflective}). If   $I$ is a small category, then the following assertions are equivalent:

\begin{enumerate}
\item $I$-colimits preserve monomorphisms in $\mathcal{B}$.
\item For each morphism $f:D\rightarrow D'$ in $[I,\mathcal{B}]$ such that  $\mathcal{A}(?,\Ker f_i)_{| \mathcal{B}}=0$, for all $i\in I$, one has that $\mathcal{A} (?,\Ker ({\rm colim}(f)))_{| \mathcal{B}}=0$, where ${\rm colim}(f):{\rm colim}\,D\rightarrow{\rm colim}\,D'$ is the image of $f$ by the colimit functor $[I,\mathcal{B}]\rightarrow \mathcal{B}$.
\end{enumerate}
\end{proposition}

\begin{proof}
Let $\alpha :B\rightarrow B'$ be a morphism in $\mathcal{B}$ and  consider the exact sequence of additive contravariant functors $\mathcal{B}\rightarrow\Ab$  $$0\rightarrow\mathcal{A}(?,\Ker \alpha)_{| \mathcal{B}}\rightarrow\mathcal{B}(?,B)\stackrel{\alpha_*}{\rightarrow}\mathcal{B}(?,B').$$ Then $\alpha$ is a monomorphism in $\mathcal{B}$ if and only if $\mathcal{A}(?,\Ker \alpha)_{| \mathcal{B}}=0$.  

In view of this fact, assertion (2) can be then restated as saying that if $f:D\rightarrow D'$ is a morphism in $[I,\mathcal{B}]$ such that $f_i:D(i)\rightarrow D'(i)$ is a monomorphism in $\mathcal{B}$, for all $i\in I$, then ${\rm colim} f:{\rm colim}\,D\rightarrow{\rm colim}\,D'$ is a monomorphism in $\mathcal{B}$. This is exactly what assertion (1) says.
\end{proof}

As a consequence, we get:

\begin{corollary}\label{cor:exactness of colimits}
Let $\mathcal{A}$ be an  AB3 abelian category and $\mathcal{B}$ be a coreflective subcategory of $\mathcal{A}$. The following assertions hold:

\begin{enumerate}
\item Suppose that $\mathcal{A}$ is also AB4 (e.g., if $\mathcal{A}$ has an injective cogenerator). Then coproducts preserve monomorphisms in $\mathcal{B}$ if, and only if, for each family of morphisms in $\mathcal B$, $(f_i:B_i\rightarrow B'_i\mid i\in I)$, such that $\mathcal{A}(?,\Ker f_i)_{| \mathcal{B}}=0$ for all $i\in I$, one has that $\mathcal{A}(?,\coprod_{i\in I}\Ker f_i)_{| \mathcal{B}}=0$. In particular, when $\mathcal{B}$ is also abelian, it is AB4 if and only if the last mentioned condition holds. 
\item Suppose that $\mathcal{A}$ is also AB5. Then direct limits preserve monomorphisms in $\mathcal{B}$ if, and only if, for each direct system of morphisms in $\mathcal B$, $(f_i:B_i\rightarrow B'_i \mid i\in I)$, such that $\mathcal{A}(?,\Ker f_i)_{| \mathcal{B}}=0$, for all $i\in I$, one has that $\mathcal{A}(?,\varinjlim\Ker f_i)_{| \mathcal{B}}=0$. In particular, when $\mathcal{B}$ is abelian, it is AB5 if and only if the last mentioned condition holds.
\item When $\mathcal{A}$ is AB5 and complete (e.g., when $\mathcal{A}$ is a Grothendieck category), coproducts preserve monomorphisms in $\mathcal{B}$. In particular, when $\mathcal{B}$ is  also abelian, it is an AB4 abelian category.
\end{enumerate}
\end{corollary}

\begin{proof}
By Example \ref{ex.preservation of monos - AB conditions}, the particular cases of assertions (1) and (2) when $\mathcal{B}$ is abelian follow. So we just need to prove the first part of each assertion. 

(1) In this case we have an isomorphism $\coprod_{i\in I}\Ker f_i\cong\Ker (\coprod_{i\in I}f_i)$. The assertion follows by  Proposition \ref{prop:exactness of colimits}, where the index set $I$ is viewed as a discrete small category.

(2) Given any direct system $(f_i:B_i\rightarrow B'_i \mid i\in I)$ of morphisms in $\mathcal{B}$, we have an isomorphism $\varinjlim\Ker f_i\cong\Ker(\varinjlim f_i)$, due to the AB5 condition of $\mathcal{A}$. The result is then a direct consequence of applying Proposition \ref{prop:exactness of colimits}, where the directed set $I$ is viewed as a small category in the obvious way. 

(3) By the complete AB5 condition of $\mathcal{A}$, for each family $(A_i)_{i\in I}$ of objects of $\mathcal{A}$ the canonical morphism $\coprod_{i\in I}A_i\rightarrow\prod_{i\in I}A_i$ is a monomorphism, since it is a direct limit of sections. Let then $(f_i:B_i\rightarrow B'_i\mid i\in I)$ be a family of morphisms in $\mathcal{B}$ such that  $\mathcal{A}(?,\Ker f_i)_{| \mathcal{B}}=0$, for all $i\in I$. Again by the AB5 condition of $\mathcal{A}$, we have that $\Ker(\coprod_{i\in I}f_i)\cong\coprod_{i\in I}\Ker f_i$, where all kernels are considered in $\mathcal{A}$. We then have a monomorphism $$\mathcal{A}(?,\Ker(\coprod_{i\in I}f_i))_{| \mathcal{B}}\cong\mathcal{A}(?,\coprod_{i\in I}\Ker f_i)_{| \mathcal{B}}\rightarrowtail\mathcal{A}(?,\prod_{i\in I}\Ker f_i)_{| \mathcal{B}}\cong\prod_{i\in I}\mathcal{A}(?,\Ker f_i)_{| \mathcal{B}}=0.$$ By Proposition \ref{prop:exactness of colimits}, we conclude that coproducts preserve monomorphisms in $\mathcal{B}$.
\end{proof}

\begin{remark} \label{rem.reduction of exactness of colimits}
In the last two results the condition that $\mathcal{A}(?,A)_{| \mathcal{B}}=0$, for some object $A\in \mathcal{A}$ and coreflective subcategory $\mathcal{B}$, frequently appears. Note that when $\mathcal{B}$ has a set $\mathcal{U}$ of generators, i.e. a set $\mathcal{U}\subset\mathcal{B}$ such that each object of $\mathcal{B}$ is an epimorphic image of a coproduct of objects of $\mathcal{U}$, one has that  $\mathcal{A}(?,A)_{| \mathcal{B}}=0$ if and only if  $\mathcal{A}(?,A)_{| \mathcal{U}}=0$. We will face this situation in Section 8.
\end{remark}

Of course, the dual of the last two results are true as well.

All these results naturally raise the question of when a coreflective subcategory of an AB3 abelian category is itself abelian. We can give the following characterization of such categories. We refer to Definition \ref{d:Trace} for the concept of trace, and we note that every coreflective subcategory of an AB3 abelian category $\mathcal{A}$ does have a trace in any object of $\mathcal{A}$ by Lemma \ref{l:Existence_trace}.

\begin{theorem}\label{t:CoreflectiveAndAbelian}
Let $\mathcal{A}$ be an AB3 abelian category and $\mathcal{B}$ be a full additive coreflective subcategory. The following assertions are equivalent:
\begin{enumerate}
\item $\mathcal{B}$ is abelian.

\item For every morphism $f:B \rightarrow B'$ in $\mathcal{B}$, the canonical morphism $\overline{f}:
\frac{B}{\tr_{\mathcal{B}}(\Ker f)} \to \Img f$ is a $\mathcal{B}$-coreflection.
\end{enumerate}
\end{theorem}

\begin{proof}
Denote by $q$ the right adjoint of the inclusion functor and by $\nu$ the counit of the adjunction. Let $f:B \rightarrow B'$ be a morphism in $\mathcal{B}$ and denote by $j:\Img f \rightarrow B'$ the inclusion. We can construct the following commutative diagram in $\mathcal{A}$:
\begin{displaymath}\label{e:1}
\begin{tikzcd}
q(\Ker f) \arrow{r}{\nu_{\Ker f}} & \Ker f \arrow{r}{i} & B \arrow{d}{p} \arrow{r}{f} & B' \arrow{r}{q} & \Coker f\\
 & & \frac{B}{\tr_{\mathcal{B}}(\Ker f)}\arrow{r}{h} \arrow{d}{\tilde{f}} & q(\Img f) \arrow{u}{j\nu_{\Img f}} \arrow{d}{\nu_{\Img f}} & \\
 & & \frac{B}{\Ker f} \arrow{r}{g} & \Img f & \\
\end{tikzcd}
\end{displaymath}
where $g$ is the canonical isomorphism. By  Subsection \ref{ss:reflective-coreflective}, we know that $i\nu_{\Ker f}$ is the kernel of $f$ in $\mathcal{B}$. By Lemma \ref{l:Existence_trace} (1), we know that the image of $i \nu_{\Ker f}$ is $\tr_{\mathcal{B}}(\Ker f)$. Now, since $\mathcal{B}$ is closed under taking cokernels by Corollary \ref{c:AbelianCategories}, $p$ is the cokernel of $i\nu_{\Ker f}$ in $\mathcal{B}$. Reasoning similarly, we get that $j\nu_{\Img f}$ is the image of $f$ in $\mathcal{B}$. The conclusion is that $h$ is the canonical map from the coimage of $f$ to the image of $f$ in $\mathcal{B}$. 

Note that, since $\mathcal A$ is abelian, $g$ is an isomorphism. By the previous paragraph, $\mathcal{B}$ is abelian if and only if $h$ is an isomorphism. If $h$ is an isomorphism, then $\nu_{\Img f}h=g\tilde{f}=\overline f$ is a $\mathcal{B}$-coreflection. Conversely, if $\overline f=\nu_{\Img f}h$ is a $\mathcal{B}$-coreflection, then $h$ is an isomorphism,  since $\nu_{\Img f}$ is also a $\mathcal{B}$-coreflection and $\mathcal{B}$-coreflections are unique, up to unique isomorphism. 
\end{proof}

\section{The category $\PU$}

In this section we consider the special case of the full additive subcategory $\PU$ of an AB3 abelian category $\mathcal{A}$ formed starting from a set of objects $\mathcal{U}$. The motivation of this explicit study is the following.

\begin{proposition} \label{prop.abelian-implies-PU}
Let $\mathcal{A}$ be an AB3 abelian category and let $\mathcal{B}$ be any coreflective subcategory of $\mathcal{A}$ that is abelian. Then $\mathcal{B}$ has a set of generators if and only if  $\mathcal{B}=\PU$ for some set $\mathcal{U}$ of objects.
\end{proposition}

\begin{proof}
If $\mathcal{B}=\PU$, then $\mathcal{U}$ is a set of generators of $\mathcal{B}$. 

Conversely, suppose that $\mathcal{B}$ has a set of generators, $\mathcal{S}$, and let us see that $\Pres(\mathcal{S})=\mathcal{B}$. The inclusion $\textrm{Pres}(\mathcal{S}) \subseteq \mathcal{B}$ follows since every object in $\textrm{Pres}(\mathcal{S})$ is a cokernel in $\mathcal A$ of a morphism between objects in $\mathcal{B}$, and $\mathcal{B}$ is closed under taking cokernels by Corollary \ref{c:AbelianCategories}. In order to prove the inclusion $\mathcal B \subseteq \textrm{Pres}(\mathcal S)$, let $B \in \mathcal{B}$. Since $\mathcal{S}$ is a set of generators, there exists an epimorphism, equivalently a cokernel map, $f:\coprod_{i \in I}S_i \rightarrow B$ in $\mathcal{B}$, where all $S_i$ are in $\mathcal{S}$. Then $f$ is also an epimorphism in $\mathcal{A}$ by Corollary \ref{c:AbelianCategories}. Let $k:K \rightarrow \coprod_{i \in I}S_i$ be the kernel of $f$ in $\mathcal{B}$. Since $\mathcal{B}$ is abelian, $f$ is the cokernel of $k$ in $\mathcal{B}$ and, using that $\mathcal{B}$ is closed under taking cokernels, we conclude that $f$ is the cokernel of $k$ in $\mathcal{A}$. Now, since $K \in \mathcal{B}$, we can take an epimorphism $\coprod_{j \in J}S'_j \rightarrow K$, which gives the following exact sequence in $\mathcal{A}$:
\begin{displaymath}
\begin{tikzcd}
\coprod_{j \in J}S'_j \arrow{r} & \coprod_{i \in I}S_i \arrow{r} & B \arrow{r} & 0
\end{tikzcd}
\end{displaymath}
Thus, $B \in \textrm{Pres}(\mathcal{S})$.
\end{proof}

Let us recall some facts about $\PU$. First, there is a canonical way to induce $\Sum(\mathcal{U})$-precovers. For any object $A\in \mathcal{A}$, we denote $\hat{U}_A=\coprod_{U\in \mathcal{U}}U^{(\mathcal{A}(U,A))}$
and by $\varepsilon_A:\hat{U}_A\to A$ the canonical morphism. It is very easy to see that, for every object $A$ of $\mathcal{A}$, 
$\varepsilon_A:\hat{U}_A\to A$ is a ${\rm Sum}(\mathcal{U})$-precover.

\begin{lemma} \label{l:U-functor} Let $\mathcal{A}$ be an AB3 abelian category and let $\mathcal{U}$ be a set of objects of $\mathcal{A}$. 
Then the assignment $A\mapsto \hat{U}_A$ ($A\in \mathcal{A}$) is functorial.
\end{lemma}

\begin{proof}
This is mainly well-known, and straightforward. 
%
\end{proof}

Using these standard $\Sum(\mathcal{U})$-precovers, it is easy to show that $\PU$ is precovering:

\begin{lemma} \label{l:PU-precover} Let $\mathcal{A}$ be an AB3 abelian category and let $\mathcal{U}$ be a set of objects of $\mathcal{A}$. 
Then for every object $A\in \mathcal{A}$ the composed morphism 
$$\hat{\varepsilon}_A:\frac{\hat{U}_A}{{\rm tr}_{\mathcal{U}}(\Ker \varepsilon_A)}\twoheadrightarrow 
\frac{\hat{U}_A}{\Ker \varepsilon_A}\hookrightarrow A$$ is a functorial ${\rm Pres}(\mathcal{U})$-precover of $A$, whose kernel is isomorphic to $\frac{\Ker \varepsilon_A}{{\rm tr}_{\mathcal{U}}(\Ker \varepsilon_A)}$.
\end{lemma}

\begin{proof} Denote $\mathcal{B}={\rm Pres}(\mathcal{U})$. Let $f:B\to A$ be a morphism in $\mathcal{A}$ with $B\in \mathcal{B}$. Then there is an exact sequence 
$$\coprod_{U\in \mathcal{U}}U^{(J_U)}\longrightarrow \coprod_{U\in \mathcal{U}}U^{(I_U)}\longrightarrow B\longrightarrow 0.$$ 
Since $\hat{U}_A$ is a ${\rm Sum}(\mathcal{U})$-precover of $A$, we have the following induced commutative diagram with exact rows:
\[
\xymatrix{
 & \coprod_{U\in \mathcal{U}}U^{(J_U)} \ar[r] \ar[d]_h & \coprod_{U\in \mathcal{U}}U^{(I_U)} \ar[r] \ar[d]^g & B \ar[r] \ar[d]^f & 0 \\ 
0 \ar[r] & \Ker \varepsilon_A \ar[r]_u & \hat{U}_A \ar[r]_{\varepsilon_A} & A & 
}\]
Since $\Img h\subseteq {\rm tr}_{\mathcal{U}}(\Ker \varepsilon_A)$, we have the following induced commutative diagram with exact rows:
\[
\xymatrix{
 & \coprod_{U\in \mathcal{U}}U^{(J_U)} \ar[r] \ar[d]_{\tilde{h}} & \coprod_{U\in \mathcal{U}}U^{(I_U)} \ar[r] \ar[d]^g & 
 B \ar[r] \ar[d]^{\tilde{f}} & 0 \\ 
0 \ar[r] & {\rm tr}_{\mathcal{U}}(\Ker \varepsilon_A) \ar[r]_-{u'} & \hat{U}_A \ar[r] & 
\frac{\hat{U}_A}{{\rm tr}_{\mathcal{U}}(\Ker \varepsilon_A)} \ar[r] & 0
}\]
where $u'$ is the restriction of $u$ to ${{\rm tr}_{\mathcal{U}}(\Ker \varepsilon_A)}$. Then we clearly have the factorization 
$f=\hat{\varepsilon}_A\tilde{f}$. Hence $\hat{\varepsilon}_A:\frac{\hat{U}_A}{{\rm tr}_{\mathcal{U}}(\Ker \varepsilon_A)}\to A$ 
is a ${\rm Pres}(\mathcal{U})$-precover of $A$.

In order to see the functoriality of the precover, fix $f:A \rightarrow B$ in $\mathcal A$, and notice that the induced morphism $f_1:\hat U_A\rightarrow \hat U_B$ is functorial by the previous lemma. Then, $f_1$ induces (functorially) a morphism $f_2:\Ker \varepsilon_A \rightarrow \Ker \varepsilon_B$, as $\varepsilon$ is a natural transformation. Using that $\tr_{\mathcal U}$ is a functor, this $f_2$ induces (functorially, again) a morphism $f_3:\tr_{\mathcal U}(\Ker \varepsilon_A) \rightarrow \tr_{\mathcal U}(\Ker \varepsilon_B)$. Finally, since this $f_3$ is, essentially, the restriction of $f_1$, we obtain a morphism $f_4:\hat U_A/\tr_{\mathcal U}(\Ker \varepsilon _A) \rightarrow \hat U_B/\tr_{\mathcal U}(\Ker \varepsilon _B)$, which is functorial as it is obtained from the universal property of the cokernel.
\end{proof}

The following criterion is useful to check when a morphism is a $\PU$-coreflection, even without assuming that $\PU$ is coreflective:

\begin{lemma} \label{lem.coreflector}
Let $\mathcal{U}$ be a set of objects in the AB3 abelian category $\mathcal{A}$ and put $\mathcal{B}=\PU$. A morphism $p:B_A\longrightarrow A$ with $B_A \in \mathcal{B}$ is a $\mathcal{B}$-coreflection if and only if $\mathcal{A}(U,p)$ is an isomorphism for each $U \in \mathcal{U}$. In particular, a morphism $f:B\longrightarrow B'$ in $\mathcal{B}$ is an isomorphism if, and only if, $\mathcal{A}(U,f)$ is an isomorphism for each $U \in \mathcal{U}$.
\end{lemma}

\begin{proof}
We only need to prove the `if' part. Consider the subcategory $\mathcal{A}_p$ of $\mathcal{A}$ consisting of all those $A'\in\mathcal{A}$ such that $\mathcal{A}(A',p)$ is an isomorphism. It is clear that $\mathcal{A}_p$ is closed under taking coproducts. Moreover,  due to the left exactness of the contravariant Hom functor, $\mathcal{A}_p$ is also closed under taking cokernels. Since, by hypothesis, we have that $\mathcal{U}\subset\mathcal{A}_p$, we immediately get that $\mathcal{B}\subseteq\mathcal{A}_p$, and hence $p$ is a $\mathcal{B}$-coreflection. 

The last statement is clear since the morphism $f$ is a $\mathcal{B}$-coreflection if, and only if, it is an isomorphism. 
\end{proof}

Now we can characterize when $\PU$ is coreflective:

\begin{theorem} \label{t:PresCoreflective}
Let $\mathcal{A}$ be an AB3 abelian category and let $\mathcal{U}$ be a set of objects of $\mathcal{A}$. Denote by $\mathcal{B}=\PU$. Then the following are equivalent:
\begin{enumerate}
\item $\mathcal{B}$ is a coreflective subcategory of $\mathcal{A}$.
\item $\mathcal{B}$ is closed under taking cokernels in $\mathcal{A}$.
\item
\begin{enumerate} 
\item $\mathcal{B}$ is closed under taking direct summands in $\mathcal{A}$.
\item For every object $A\in \mathcal{A}$, 
$ \frac{{\rm tr}^2_{\mathcal{U}}(\Ker \varepsilon_A)}{{\rm tr}_{\mathcal{U}}(\Ker \varepsilon_A)}={\rm tr}_{\mathcal{U}}\left(\frac{\Ker \varepsilon_A}{{\rm tr}_{\mathcal{U}}(\Ker \varepsilon_A)}\right)$ 
is a direct summand of $\frac{\hat{U}_A}{{\rm tr}_{\mathcal{U}}(\Ker \varepsilon_A)}$.
\end{enumerate}
\item For every  object $A\in \mathcal{A}$ , the following conditions hold:
\begin{enumerate}
\item $\frac{\hat{U}_A}{{\rm tr}_{\mathcal{U}}^2(\Ker\varepsilon_A)}\in \mathcal{B}$;
 \item $\mathcal{A}\left(U,\frac{\Ker \varepsilon_A}{{\rm tr}^2_{\mathcal{U}}(\Ker \varepsilon_A)}\right)=0$, for all $U\in\mathcal{U}$.
\end{enumerate}
\end{enumerate}

In this  case, if  $\iota :\mathcal{B}\hookrightarrow\mathcal{A}$ is the inclusion functor, then, up to natural isomorphism,  its right adjoint  $q:\mathcal{A}\longrightarrow\mathcal{B}$ acts on objects as $q(A)=\frac{\hat{U}_A}{{\rm tr}_{\mathcal{U}}^2(\Ker \varepsilon_A)}$. 
\end{theorem}

\begin{proof}
$(1)\Longleftrightarrow (2)$ This follows from Corollary \ref{c:AbelianCategories} and Lemma \ref{l:PU-precover}.

$(1)\Longrightarrow (3)$ We just need to prove (3.b).  By Lemma \ref{l:PU-precover} the canonical morphism 
$\hat{\varepsilon}_A:\frac{\hat{U}_A}{{\rm tr}_{\mathcal{U}}(\Ker \varepsilon_A)}\to A$ is a $\mathcal{B}$-precover of $A$.  Note that it factors as the composition $\frac{\hat{U}_A}{{\rm tr}_{\mathcal{U}}(\Ker \varepsilon_A)}\stackrel{\pi}{\twoheadrightarrow}\frac{\hat{U}_A}{{\rm tr}^2_{\mathcal{U}}(\Ker \varepsilon_A)}\stackrel{\tilde{\varepsilon}_A}{\longrightarrow}A$, where $\pi$ is the canonical projection. Bearing in mind that $\Ker(\pi )=\frac{{\rm tr}^2_{\mathcal{U}}(\Ker \varepsilon_A)}{{\rm tr}_{\mathcal{U}}(\Ker \varepsilon_A)}$ is in $\Gen(\mathcal{U})$ and that $\mathcal{B}$ is closed under taking cokernels, we get that $\frac{\hat{U}_A}{{\rm tr}^2_{\mathcal{U}}(\Ker \varepsilon_A)}\in\PU$ and so $\tilde{\varepsilon}_A$ is also a $\mathcal{B}$-precover. 

Let now $q:\mathcal{A}\longrightarrow\mathcal{B}$ be the right adjoint to the inclusion functor $\iota :\mathcal{B}\hookrightarrow\mathcal{A}$, and we put $\nu_A:q(A)\longrightarrow A$ for the counit morphism, that is the $\mathcal{B}$-coreflection, and hence the $\mathcal{B}$-cover,  of $A$. We can then choose a morphism $g:\frac{\hat{U}_A}{{\rm tr}^2_{\mathcal{U}}(\Ker \varepsilon_A)}\longrightarrow q(A)$ such that $\nu_A g=\tilde{\varepsilon}_A$, and so $\hat{\varepsilon}_A=\nu_A g\pi$. By Lemma \ref{lem.cover-directsummand-of-precover}, we get that $g\pi$ is retraction and that there exists a subobject $X'$ of $\frac{\hat U_A}{\tr_{\mathcal{U}}(\Ker \varepsilon_A)}$ and an isomorphism $\varphi:X' \rightarrow q(A)$, such that $\frac{\hat U_A}{\tr_{\mathcal{U}}(\Ker \varepsilon_A)} = \Ker (g\pi) \coprod X'$ and $\hat \varepsilon_A = \nu_A\varphi p$, where $p$ is the projection onto $X'$. Now notice that $\Ker (g\pi)\subseteq\Ker (\nu_Ag\pi)=\Ker (\tilde{\varepsilon}_A\pi)=\Ker (\hat{\varepsilon}_A)=\frac{\Ker (\varepsilon_A)}{\tr_{\mathcal{U}}(\Ker \varepsilon_A)}$. Moreover, $\Ker (g\pi)$ is $\mathcal{U}$-generated since it is a direct summand of $\frac{\hat U_A}{\tr_{\mathcal{U}}(\Ker \varepsilon_A)}$. It then follows that
$\Ker g\pi \subseteq \tr_{\mathcal U}\left(\frac{\Ker \varepsilon_A}{\tr_{\mathcal{U}}(\Ker \varepsilon_A)}\right)$. But we also have that $\tr_{\mathcal{U}}\left(\frac{\Ker \varepsilon_A}{\tr_{\mathcal{U}}(\Ker \varepsilon_A)}\right) = \Ker \pi \subseteq \Ker(g\pi)$. This proves (3.b) and the composition $X'\stackrel{\varphi}{\cong}q(A)\stackrel{\nu_A}{\rightarrow} A$ gets canonically identified with 
$$\frac{\hat{U}_A/\tr_{\mathcal{U}}(\Ker \varepsilon_A)}{\Ker (g\pi)}=\frac{\hat{U}_A/\tr_{\mathcal{U}}(\Ker \varepsilon_A)}{\tr^2_{\mathcal{U}}(\Ker \varepsilon_A)/\tr_{\mathcal{U}}(\Ker \varepsilon_A)}\cong\frac{\hat{U}_A}{\tr^2_{\mathcal{U}}(\Ker \varepsilon_A)}\stackrel{\tilde{\varepsilon}_A}{\rightarrow} A.$$

This also proves the last statement of the theorem.

$(3)\Longrightarrow (4)$ By condition (3.b), we have a decomposition $\frac{\hat{U}_A}{{\rm tr}_{\mathcal{U}}(\Ker \varepsilon_A)}=\frac{{\rm tr}_{\mathcal{U}}^2(\Ker \varepsilon_A)}{{\rm tr}_{\mathcal{U}}(\Ker \varepsilon_A)}\coprod W$. By condition (3.a) we then get that $\frac{\hat{U}_A}{{\rm tr}_{\mathcal{U}}^2(\Ker \varepsilon_A)}\cong W$ is in $\mathcal{B}$, so that condition (4.a) holds.
  
But condition (3.b) also  implies that ${\rm tr}_{\mathcal{U}}\left(\frac{\Ker\varepsilon_A}{{\rm tr}_{\mathcal{U}}(\Ker\varepsilon_A)}\right)=\frac{{\rm tr}_{\mathcal{U}}^2(\Ker\varepsilon_A)}{{\rm tr}_{\mathcal{U}}(\Ker\varepsilon_A)}$ is  a direct summand of $\frac{\Ker\varepsilon_A}{{\rm tr}_{\mathcal{U}}(\Ker\varepsilon_A)}$. If we fix a decomposition $\frac{\Ker\varepsilon_A}{{\rm tr}_{\mathcal{U}}(\Ker\varepsilon_A)}=\frac{{\rm tr}_{\mathcal{U}}^2(\Ker\varepsilon_A)}{{\rm tr}_{\mathcal{U}}(\Ker\varepsilon_A)}\coprod V$, we conclude that any morphism $f:U \rightarrow V$ with $U \in \mathcal{U}$ satisfies that the composition $U\stackrel{f}{\longrightarrow}V\stackrel{inclusion}{\hookrightarrow}\frac{\Ker\varepsilon_A}{{\rm tr}_{\mathcal{U}}(\Ker\varepsilon_A)}$ has its image contained in  ${\rm tr}_{\mathcal{U}}\left(\frac{\Ker\varepsilon_A}{{\rm tr}_{\mathcal{U}}(\Ker\varepsilon_A)}\right)=\frac{{\rm tr}_{\mathcal{U}}^2(\Ker\varepsilon_A)}{{\rm tr}_{\mathcal{U}}(\Ker\varepsilon_A)}$, and hence in $V\cap\frac{{\rm tr}_{\mathcal{U}}^2(\Ker\varepsilon_A)}{{\rm tr}_{\mathcal{U}}(\Ker\varepsilon_A)}=0$. It follows that $f=0$ and hence $\mathcal{A}(U,V)=0$, for all $U\in\mathcal{U}$. Now, condition (4.b) follows from the fact that $V \cong \frac{\Ker(\varepsilon_A)}{\tr_{\mathcal{U}}^2(\Ker \varepsilon_A)}$.
  
$(4)\Longrightarrow (1)$ Let $A\in\mathcal{A}$ be arbitrary. Due to the fact that the canonical morphism $\varepsilon_A:\hat{U}_A\rightarrow A$ is a $\text{Sum}(\mathcal{U})$-precover, conditions (4.a) and (4.b) imply that the canonical morphism
$\tilde{\varepsilon}_A:\frac{\hat{U}_A}{{\rm tr}^2_{\mathcal{U}}(\Ker \varepsilon_A)}\rightarrow A$ satisfies that $\mathcal{A}(U,\tilde\varepsilon_A)$ is an isomorphism for each $U \in \mathcal{U}$. Then it is a $\mathcal{B}$-coreflection, by condition (4.a) and Lemma \ref{lem.coreflector}.
\end{proof}

Now we characterize when $\PU$ is coreflective and abelian.

\begin{theorem} \label{t.coreflective-abelian}
Let $\mathcal{A}$ be an AB3 abelian category, let $\mathcal{U}$ be a set of objects of $\mathcal{A}$ and denote $\mathcal{B}=\PU$. The following assertions are equivalent:
\begin{enumerate}
\item $\mathcal{B}$ is a coreflective subcategory that is abelian.
\item $\mathcal{B}$ is closed under taking cokernels and, for every morphism $f:B\to B'$ in $\mathcal{B}$, the induced morphism 
$\bar{f}:\frac{B}{{\rm tr}_{\mathcal{U}}(\Ker f)}\to \Img f$ satisfies that $\mathcal A(U,\overline f)$ is an isomorphism for each $U \in \mathcal U$.
\item $\mathcal{B}$ is closed under taking cokernels and, for each morphism $f:B\to B'$ in $\mathcal{B}$, the induced morphism 
$\bar{f}:\frac{B}{{\rm tr}_{\mathcal{U}}(\Ker f)}\to \Img f$ is a $\mathcal{B}$-coreflection.
\item The following two conditions hold:
\begin{enumerate}
\item $\mathcal{A}\left(U,\frac{\Ker\,\varepsilon_A}{\tr_{\mathcal{U}}(\Ker\,\varepsilon_A)}\right)=0$, for all $A\in\mathcal{A}$ and $U\in\mathcal{U}$;
\item For every morphism $g:\tilde{U}\longrightarrow B$, where $B\in\mathcal{B}$ and $\tilde{U}\in{\rm Sum}(\mathcal{U})$, the induced morphism $\bar{g}:\frac{\tilde{U}}{{\rm tr}_{\mathcal{U}}(\Ker g)}\longrightarrow \Img g$ satisfies that $\mathcal A(U,\overline g)$ is an isomorphism for each $U \in \mathcal U$. 
\end{enumerate}
\end{enumerate}

In this case, up to natural isomorphism, the right adjoint $q:\mathcal{A}\rightarrow\PU$ of the inclusion functor acts on objects as $q(A)=\frac{\hat{U}_A}{\tr_\mathcal{U}(\Ker \varepsilon_A)}$.
\end{theorem}

\begin{proof}
Recall that, by Theorem \ref{t:PresCoreflective}, $\mathcal{B}$ is coreflective if, and only if, it is closed under taking cokernels in $\mathcal{A}$. On the other hand, condition (4.a) is equivalent to say that $\frac{{\rm tr}_{\mathcal{U}}^2(\Ker \varepsilon_A)}{{\rm tr}_{\mathcal{U}}(\Ker \varepsilon_A)}=0$, which immediately implies assertion (4) of the mentioned theorem. Therefore we can assume that $\mathcal{B}$ is coreflective all throughout the proof. Let us denote by $q$ the right adjoint of the inclusion functor.

$(1)\Longleftrightarrow (3)$ follows from Theorem \ref{t:CoreflectiveAndAbelian} and $(3)\Longleftrightarrow (2)$ from Lemma \ref{lem.coreflector}.  

$(1),(2),(3)\Longrightarrow (4)$ Condition (4.b) follows from assertion (2). As for condition (4.a), recall that, by  Theorem \ref{t:PresCoreflective} again, $q(A)\cong\frac{\hat{U}_A}{\tr_{\mathcal{U}}^2(\Ker\,\varepsilon_A)}$. Applying assertion (3) with $f=\pi:\hat{U}_A\twoheadrightarrow\frac{\hat{U}_A}{\tr_{\mathcal{U}}^2(\Ker\,\varepsilon_A)}$ the canonical projection and bearing in mind that $\tr_{\mathcal{U}}(\Ker\,\pi)=\tr_{\mathcal{U}}(\tr^2_{\mathcal{U}}(\Ker\,\varepsilon_A))=\tr_{\mathcal{U}}(\Ker\,\varepsilon_A)$, we get that the induced map $\bar{\pi}:\frac{\hat{U}_A}{{\rm tr}_{\mathcal{U}}(\Ker \varepsilon_A)}\longrightarrow\frac{\hat{U}_A}{{\rm tr}_{\mathcal{U}}^2(\Ker \varepsilon_A)}$ is a $\mathcal{B}$-coreflection, and hence an isomorphism by Lemma \ref{lem.coreflector}. But this implies that $\frac{{\rm tr}_{\mathcal{U}}^2(\Ker \varepsilon_A)}{{\rm tr}_{\mathcal{U}}(\Ker \varepsilon_A)}=0$, which is equivalent to condition (4.a).

$(4)\Longrightarrow (2)$ Let $f:B\longrightarrow B'$ be any morphism in $\mathcal{B}$ and consider the canonical (epi) morphism $\varepsilon_B:\hat{U}_B\twoheadrightarrow B$. Note that  we then have an isomorphism $\Img(f\varepsilon_B)\cong \Img f$, that we see as an identity in the sequel. 
We then get the following commutative diagram with exact rows:  
\[\SelectTips{cm}{}
\xymatrix{
0 \ar[r] & {\rm Ker}(f\varepsilon_B) \ar@{^{(}->}[r] \ar[d]_{\varepsilon'} & \hat{U}_B \ar[r]^-{\bar{f}\varepsilon_B} \ar[d]^{\varepsilon_B}& 
\Img f \ar[r] \ar@{=}[d] & 0 \\
0 \ar[r] & {\rm Ker}(f) \ar@{^{(}->}[r] & B \ar[r] & \Img f \ar[r] & 0
}\]
Then $\varepsilon'$ is an epimorphism and we clearly have that $\varepsilon' ({\rm tr}_\mathcal{U}(\Ker(f\varepsilon_B)))\subseteq {\rm tr}_\mathcal{U}(\Ker f) $. Then we get another induced commutative diagram with exact rows
\begin{equation}\label{e:Diag}\tag{$\dagger$}
\SelectTips{cm}{}
\xymatrix{
0 \ar[r] & \frac{{\rm Ker}(f\varepsilon_B)}{{\rm tr}_{\mathcal{U}}({\rm Ker}(f\varepsilon_B))} \ar[r] \ar[d]_{\overline{\varepsilon'}} 
& \frac{\hat{U}_B}{{\rm tr}_{\mathcal{U}}({\rm Ker}(f\varepsilon_B))} \ar[r]^-{\overline{f\varepsilon_B}} \ar[d]^{\overline{\varepsilon_B}}
& \Img f \ar[r] \ar@{=}[d] & 0 \\
0 \ar[r] & \frac{{\rm Ker}(f)}{{\rm tr}_{\mathcal{U}}({\rm Ker}(f))} \ar[r]_-{\bar{j}} & 
\frac{B}{{\rm tr}_{\mathcal{U}}({\rm Ker}(f))} \ar[r]_-{\bar{f}} & \Img f \ar[r] & 0
}
\end{equation}
which implies, in particular, that the left square of the diagram is bicartesian.

Fix $U\in\mathcal{U}$ arbitrary. By assertion (4), we know that $\mathcal{A}(U,\overline{f\varepsilon_B})$ is an isomorphism. This in turn  implies   that $\mathcal{A}(U,\bar{f})$ is an epimorphism and the proof is reduced to check that $\mathcal{A}\left(U,\frac{\Ker f}{{\rm tr}_\mathcal{U}(\Ker f)}\right)=0$, for all $U\in\mathcal{U}$.
 Bearing in mind the left exactness of $\mathcal{A}(U,?):\mathcal{A}\longrightarrow\text{Ab}$, we also get a commutative diagram in $\text{Ab}$ with exact rows: 
\[\SelectTips{cm}{}
\xymatrix{
0 \ar[r] & \mathcal{A}\left(U,\frac{{\rm Ker}(f\varepsilon_B)}{{\rm tr}_{\mathcal{U}}({\rm Ker}(f\varepsilon_B))}\right) 
\ar[r] \ar[d] & \mathcal{A}\left(U,\frac{\hat{U}_B}{{\rm tr}_{\mathcal{U}}({\rm Ker}(f\varepsilon_B))}\right) 
\ar[r]^-{(\overline{f\varepsilon_B})_*} \ar[d]^{(\overline{\varepsilon_B})_*}
& \mathcal{A}\left(U,\Img f\right) \ar[r] \ar@{=}[d] & 0 \\
0 \ar[r] & \mathcal{A}\left(U,\frac{{\rm Ker}(f)}{{\rm tr}_{\mathcal{U}}({\rm Ker}(f))}\right) \ar[r]_-{(\bar{j})_*} & 
\mathcal{A}\left(U,\frac{B}{{\rm tr}_{\mathcal{U}}({\rm Ker}(f))}\right) \ar[r]_-{(\bar{f})_*} & \mathcal{A}\left(U,\Img f\right) \ar[r] & 0
}\]

It follows that the left square of this diagram is bicartesian and, since its upper left corner is the zero object, we conclude that the map $$\begin{pmatrix}(\bar{j})_* & (\bar{\varepsilon}_B)_*\end{pmatrix}:\mathcal{A}\left(U,\frac{\Ker f}{{\rm tr}_\mathcal{U}(\Ker f)}\right)\coprod\mathcal{A}\left(U,\frac{\hat{U}_B}{{\rm tr}_\mathcal{U}(\Ker (f\varepsilon_B))}\right)\longrightarrow\mathcal{A}\left(U,\frac{B}{{\rm tr}_\mathcal{U}(\Ker f)}\right)$$ is an isomorphism, for all $U\in\mathcal{U}$.

We next consider the natural transformation $$\tau=\mathcal{A}(?,\begin{pmatrix}\bar{j} &\bar{\varepsilon}_B\end{pmatrix}):\mathcal{A}\left(?,\frac{\Ker f}{{\rm tr}_\mathcal{U}(\Ker f)}\coprod \frac{\hat{U}_B}{{\rm tr}_\mathcal{U}(\Ker (f\varepsilon_B))}\right)\longrightarrow\mathcal{A}\left(?,\frac{B}{{\rm tr}_\mathcal{U}(\Ker f)}\right)$$ of contravariant representable functors $\mathcal{A}\longrightarrow\text{Ab}$, and take the subcategory $\mathcal{A}_\tau$ of $\mathcal{A}$ consisting of those objects $A$ such that $\tau_A$ is an isomorphism. Arguing as in the proof of Lemma \ref{lem.coreflector}, we get that $\mathcal{A}_\tau$ is closed under taking coproducts and cokernels. It then follows that $\mathcal{B}\subseteq\mathcal{A}_\tau$, since the previous paragraph shows that $\mathcal{U}\subset\mathcal{A}_\tau$. In particular, $\frac{B}{{\rm tr}_\mathcal{U}(\Ker f)}\in\mathcal{A}_\tau$, which means that the morphism $$\begin{pmatrix}\bar{j} &\bar{\varepsilon}_B\end{pmatrix}:\frac{\Ker f}{{\rm tr}_\mathcal{U}(\Ker f)}\coprod \frac{\hat{U}_B}{{\rm tr}_\mathcal{U}(\Ker (f\varepsilon_B))}\longrightarrow \frac{B}{{\rm tr}_\mathcal{U}(\Ker f)} $$ is a retraction in $\mathcal{A}$. 

By the bicartesian condition of (\ref{e:Diag}), we conclude that we have an isomorphism $$\Psi:\frac{\Ker f}{{\rm tr}_\mathcal{U}(\Ker f)}\coprod \frac{\hat{U}_B}{{\rm tr}_\mathcal{U}(\Ker (f\varepsilon_B))}\stackrel{\cong}{\longrightarrow} \frac{\Ker f\varepsilon_B}{{\rm tr}_\mathcal{U}(\Ker f\varepsilon_B)} \coprod \frac{B}{{\rm tr}_\mathcal{U}(\Ker f)}.  $$ But, by assertion (4), we know that $\mathcal{A}(U,?)$ vanishes on  $\frac{\Ker f\varepsilon_B}{{\rm tr}_\mathcal{U}(\Ker f\varepsilon_B)}$, for all $U\in\mathcal{U}$. This implies that we have $\mathcal{A}\left(\frac{\hat{U}_B}{{\rm tr}_\mathcal{U}(\Ker (f\varepsilon_B))}, \frac{\Ker f\varepsilon_B}{{\rm tr}_\mathcal{U}(\Ker f\varepsilon_B)}\right)=0$. Then,  expressing matricially the isomorphism $\Psi$, we have that $\Psi=\begin{pmatrix}\alpha & 0\\ \bar{j} & \bar{\varepsilon}_B \end{pmatrix}$, for some morphism $\alpha :\frac{\Ker f}{{\rm tr}_\mathcal{U}(\Ker f)}\longrightarrow  \frac{\Ker f\varepsilon_B}{{\rm tr}_\mathcal{U}(\Ker f\varepsilon_B)} $. The inverse $\Psi^{-1}$ then has a matricial expression $\begin{pmatrix}\alpha'& \beta'\\ \gamma' & \delta' \end{pmatrix}$. The matricial equality $$\begin{pmatrix}\alpha'& \beta'\\ \gamma' & \delta' \end{pmatrix} \begin{pmatrix}\alpha & 0\\ \bar{j} & \bar{\varepsilon}_B \end{pmatrix}=\begin{pmatrix} 1 & 0\\ 0 & 1 \end{pmatrix}$$ gives then that $\delta'\bar{\varepsilon}_B=1$. Then $\bar{\varepsilon}_B$ is a section, which implies that it is an isomorphism, since it is obviously an epimorphism. It follows that $\bar{\varepsilon}': \frac{\Ker f\varepsilon_B}{{\rm tr}_\mathcal{U}(\Ker f\varepsilon_B)}\longrightarrow\frac{\Ker f}{{\rm tr}_\mathcal{U}(\Ker f)}$ is also an isomorphism, and so $\mathcal{A}\left(U,\frac{\Ker f}{{\rm tr}_\mathcal{U}(\Ker f)}\right)=0$ for all $U\in\mathcal{U}$, as it was desired. 
\end{proof}

Next we want to study when $\PU$ is abelian exact. First, we give some lemmas.



\begin{lemma} \label{l:l2} Let $\mathcal{A}$ be an AB3 abelian category and let $\mathcal{U}$ be a set of objects of $\mathcal{A}$. 
If the canonical morphism $\varepsilon_B:\hat{U}_B\to B$ has its kernel in $\Gen(\mathcal{U})$ 
for every $B\in \PU$, then $\PU$ is closed under taking cokernels.
\end{lemma}

\begin{proof} Let $f:B'\to B$ be a morphism in $\PU$. In order to prove that $\Coker f\in \PU$, we may assume without loss of generality that $B'\in {\rm Sum}(\mathcal{U})$, say $B'=\coprod_{i\in I}U_i$.

Consider the pullback of $f$ and $\varepsilon_{B}$ in order to get the following commutative diagram with exact rows,
\begin{displaymath}
\begin{tikzcd}
V \arrow{r}{v} \arrow{d}{\pi} \arrow[phantom]{dr}{\textrm{BC}} & \hat U_B \arrow{r} \arrow{d}{\varepsilon_B} & \Coker f \arrow{r} \arrow[equal]{d} & 0\\
\coprod_{i \in I}U_i \arrow{r}{f} & B \arrow{r} & \Coker f \arrow{r} & 0
\end{tikzcd}
\end{displaymath}
and notice that the left square is also a pushout. Since $\varepsilon_{B}$ is a ${\rm Sum}(\mathcal{U})$-precover of $B$, 
there is a morphism $u:\coprod_{i\in I}U_i\to \hat{U}_{B}$ such that $\varepsilon_{B}u=f$. 
Then the pullback property yields a unique morphism $w:\coprod_{i\in I}U_i\to V$ such that 
$vw=u$ and $\pi w=1_{\coprod_{i\in I}U_i}$, and so $\pi$ is a retraction. Then we have that
$$V\cong \Ker \varepsilon_{B}\coprod \left(\coprod_{i\in I}U_i\right).$$ for some family of objects $\{U_i \mid i \in I\} \subseteq \mathcal U$. Since $\Ker \varepsilon_{B}\in \Gen(\mathcal{U})$,
there is an epimorphism $\coprod_{j\in J}U'_j\to \Ker \varepsilon_{B}$. 
Then we have an induced exact sequence 
$$\left(\coprod_{j\in J}U'_j\right)\coprod \left(\coprod_{i\in I}U_i\right)\longrightarrow \hat{U}_{B}
\longrightarrow \Coker f\longrightarrow 0,$$
which shows that $\Coker f\in \PU$. 
\end{proof}

\begin{lemma} \label{l:kerf-PU} Let $\mathcal{A}$ be an AB3 abelian category and let $\mathcal{U}$ be a set of objects of $\mathcal{A}$. 
Assume that  every morphism $f:\hat{U}\to \tilde{U}$ in ${\rm Sum}(\mathcal{U})$ has its kernel in
$\Gen(\mathcal{U})$. Then for every morphism $f:\coprod_{i\in I}U_i\to A$ in $\mathcal{A}$, with each $U_i\in \mathcal{U}$, we have that
${\rm tr}_{\mathcal{U}}(\Ker f)\in {\rm Pres}(\mathcal{U})$.
\end{lemma}

\begin{proof} Let $f:\coprod_{i\in I}U_i\to A$ be a morphism in $\mathcal{A}$ with each $U_i\in \mathcal{U}$. 
Then there is an epimorphism $d:\coprod_{j\in J}U'_j\to {\rm tr}_{\mathcal{U}}(\Ker f)$. 
Denote $v=ltd:\coprod_{j\in J}U'_j\to \coprod_{i\in I}U_i$, where $l:\Ker f\to \coprod_{i\in I}U_i$ 
and $t:{\rm tr}_{\mathcal{U}}(\Ker f)\to \Ker f$ are the inclusions. Then $\Ker d=\Ker v\in \Gen(\mathcal{U})$ 
by hypothesis, and so ${\rm tr}_{\mathcal{U}}(\Ker f)\in {\rm Pres}(\mathcal{U})$. 
\end{proof}



We are now ready to determine when $\PU$ is an abelian exact subcategory of $\mathcal{A}$. 

\begin{theorem} \label{t:abexact} Let $\mathcal{A}$ be an AB3 abelian category and let $\mathcal{U}$ be a set of objects of $\mathcal{A}$. Denote $\mathcal{B}=\PU$. Then, the following are equivalent:
\begin{enumerate}
\item $\mathcal B$ is a coreflective abelian exact  subcategory of $\mathcal{A}$.

\item Each morphism $f:\hat{U}\longrightarrow\tilde{U}$ in $\mathrm{Sum}(\mathcal{U})$ has its kernel in $\Gen(\mathcal{U})$ and one (resp. all) of the following conditions hold:

\begin{enumerate}
\item  For each $\tilde{U}\in\Sum(\mathcal{U})$ and each subobject $K$ of $\tilde{U}$, we have $\mathcal{A}\left(U,\frac{K}{{\rm tr}_\mathcal{U}(K)}\right)=0$, for all $U\in\mathcal{U}$.
\item $\mathcal{A}\left(U,\frac{\Ker \varepsilon_A}{{\rm tr}_\mathcal{U}(\Ker \varepsilon_A)}\right)=0$, for all $A\in\mathcal{A}$;
\item For each $B\in\mathcal{B}$, the canonical morphism $\varepsilon_B:\hat{U}_B\longrightarrow B$ has its kernel in $\Gen(\mathcal{U})$.
\end{enumerate}
\end{enumerate}
\end{theorem}

\begin{proof}
We clearly have that  $(2.a)\Longrightarrow (2.b)$. The scheme of the proof will go  as follows. We shall prove that $(1)\Longrightarrow (2.b)\Longrightarrow (2.c)\Longrightarrow (1)$, and we will finally prove that all these equivalent conditions imply (2.a). 

$(1)\Longrightarrow (2.b)$  By Theorem \ref{t.coreflective-abelian} we get condition (2.b). That each morphism in $\text{Sum}(\mathcal{U})$ has its kernel in $\Gen(\mathcal{U})$ is a consequence of the fact that $\mathcal{B}$ is closed under taking kernels in $\mathcal{A}$.

$(2.b)\Longrightarrow (2.c)$  Let $B \in \mathcal B$. By Lemma \ref{l:PU-precover}  we know that the induced map $\hat{\varepsilon}_B:\frac{\hat{U}_B}{{\rm tr}_\mathcal{U}(\Ker \varepsilon_B)}\longrightarrow B$ is a $\mathcal{B}$-precover. Using Lemma \ref{lem.cover-directsummand-of-precover}, we get that $\hat{\varepsilon}_B$ is a retraction, so that we have a decomposition $\frac{\hat{U}_B}{{\rm tr}_\mathcal{U}(\Ker \varepsilon_B)}\cong\frac{\Ker \varepsilon_B}{{\rm tr}_\mathcal{U}(\Ker \varepsilon_B)}\coprod B$.  This implies that $\frac{\Ker \varepsilon_B}{{\rm tr}_\mathcal{U}(\Ker \varepsilon_B)}\in\Gen(\mathcal{U})$, which under the hypothesis (2.b) implies that  $\frac{\Ker \varepsilon_B}{{\rm tr}_\mathcal{U}(\Ker \varepsilon_B)}=0$ and so $\Ker \varepsilon_B\in\Gen(\mathcal{U})$.

$(2.c)\Longrightarrow (1)$   By  Lemma \ref{l:l2} and Theorem \ref{t:PresCoreflective}, we know that $\mathcal{B}$ is coreflective. So we need to prove that $\mathcal{B}$ is closed under taking kernels in $\mathcal{A}$. We then fix a morphism $f:B'\longrightarrow B$ in $\mathcal{B}$ and will prove that $\Ker f\in\mathcal{B}$ in several steps.

{\it Step 1: It is enough to assume that $B'\in\mathrm{Sum}(\mathcal{U})$.}

  Fix  an exact sequence $\coprod_{U\in \mathcal{U}}U^{(J_U)}\longrightarrow 
\coprod_{U\in \mathcal{U}}U^{(I_U)}\stackrel{p}\longrightarrow B'\longrightarrow 0$.
We have the following induced commutative diagram with exact rows and columns: 
\begin{displaymath}
\begin{tikzcd}
 & \coprod_{U \in \mathcal{U}}U^{(J_U)} \arrow[equal]{r} \arrow{d} & \coprod_{U \in \mathcal{U}}U^{(J_U)} \arrow{d} & \\
 0 \arrow{r} & \Ker fp \arrow{r} \arrow{d} \arrow[phantom]{dr}{\textrm{BC}} & \coprod_{U \in \mathcal{U}}U^{(I_U)} \arrow{d}{p} \arrow{r}{fp} & B \arrow[equal]{d} \\
 0 \arrow{r} & \Ker f \arrow{r} \arrow{d} & B' \arrow{r}{f} \arrow{d} & B \\
 & 0 & 0 & \\
\end{tikzcd}
\end{displaymath}
in which the lower left square is bicartesian. If $\Ker fp\in \mathcal{B}$, then $\Ker f\in \mathcal{B}$,
because $\mathcal{B}$ is closed under taking cokernels.

{\it Step 2: It is enough to prove it for epimorphisms $\tilde{U}\twoheadrightarrow B$, where $\tilde{U}\in\mathrm{Sum}(\mathcal{U})$ and $B\in\mathcal{B}$.} 

Let $f:\coprod_{U\in \mathcal{U}}U^{(K_U)}\to B$ be a morphism with $B\in \mathcal{B}$. 
Then there is an exact sequence $\coprod_{U\in \mathcal{U}}U^{(J_U)}\longrightarrow 
\coprod_{U\in \mathcal{U}}U^{(I_U)}\longrightarrow B\longrightarrow 0$. 
We have the following induced commutative diagram with exact rows:

\[\SelectTips{cm}{}
\xymatrix{
 & \coprod_{U\in \mathcal{U}}U^{(J_U)} \ar@{=}[r] \ar[d] & \coprod_{U\in \mathcal{U}}U^{(J_U)} \ar[d] & \\
0 \ar[r] & \Ker \alpha \ar[r] \ar[d] & \coprod_{U\in \mathcal{U}}U^{(I_U)} \ar[r]^-{\alpha} \ar[d] & 
\Coker f \ar@{=}[d] \ar[r] & 0 \\ 
0 \ar[r] & \Img f \ar[r] \ar[d] & B \ar[r] \ar[d] & \Coker f \ar[r] & 0 \\
 & 0 & 0 & 
}\]
in which $\Coker f\in \mathcal{B}$, because $\mathcal{B}$ is closed under taking cokernels. If $\Ker \alpha\in \mathcal{B}$, 
then $\Img f\in \mathcal{B}$, because $\mathcal{B}$ is closed under taking cokernels. 

Using again that the kernel of any epimorphism  $\tilde{U}\twoheadrightarrow B$ with $\tilde U \in  \Sum (\mathcal U)$ belongs to $\Gen (\mathcal{U})$, and that we have an epimorphism $\bar{f}:\coprod_{U\in \mathcal{U}}U^{(K_U)}\longrightarrow\Img f$ whose kernel is precisely $\Ker f$, we conclude that $\Ker f$ belongs to $\mathcal B$ as well. 

 {\it Step 3: We prove the result for epimorphisms $\tilde{U}\twoheadrightarrow B$, where $\tilde{U}\in\mathrm{Sum}(\mathcal{U})$ and $B\in\mathcal{B}$:} 
 
Let $p:\coprod_{U\in \mathcal{U}}U^{(X_U)}\to B$ be an epimorphism with $B\in \mathcal{B}$. 
Then the pullback of $p:\coprod_{U\in \mathcal{U}}U^{(X_U)}\to B$ and $\varepsilon_B:\hat{U}_B\to B$ 
induces the following commutative diagram with exact rows and columns:
\begin{displaymath}
\begin{tikzcd}
 & & 0 \arrow{d} & 0 \arrow{d} & \\
 & & \Ker \varepsilon_B \arrow[equal]{r} \arrow{d} & \Ker \varepsilon_B \arrow{d} & \\
0 \arrow{r} & \Ker p \arrow{r} \arrow[equal]{d} & K \arrow{r} \arrow{d} \arrow[phantom]{dr}{\textrm{PB}} & \hat U_B \arrow{r} \arrow{d}{\varepsilon_B} & 0\\
0 \arrow{r} & \Ker p \arrow{r} & \coprod_{U \in \mathcal{U}}U^{(X_U)} \arrow{d} \arrow{r}{p} & B \arrow{r} \arrow{d} & 0\\
 & & 0 & 0 & \\
\end{tikzcd}
\end{displaymath}
Since $\varepsilon_B$ is a ${\rm Sum}(\mathcal{U})$-precover of $B$, there is a morphism 
$v:\coprod_{U\in \mathcal{U}}U^{(X_U)}\to \hat{U}_B$ such that $\varepsilon_B v=p$. Then the pullback property 
implies that the middle column splits. 

By condition (2.c), we know that 
 $\Ker \varepsilon_B={\rm tr}_{\mathcal{U}}(\Ker \varepsilon_B)\in \Gen(\mathcal{U})$, which in turn implies that  $\Ker \varepsilon_B\in {\rm Pres}(\mathcal{U})$ due to Lemma \ref{l:kerf-PU}. 
Thus we have: $$K\cong \Ker \varepsilon_B\coprod \left(\coprod_{U\in \mathcal{U}}U^{(X_U)}\right)\in {\rm Pres}(\mathcal{U}).$$
 Then there is an exact sequence $\coprod_{j\in J}U'_j\to \coprod_{i\in I}U_i\to K\to 0$.  
We can construct the following commutative diagram with exact rows and columns: 
\begin{displaymath}
\begin{tikzcd}
 & \coprod_{j \in J}U'_j \arrow[equal]{r} \arrow{d} & \coprod_{j \in J}U'_j \arrow{d} & & \\
 0 \arrow{r} & L \arrow{r} \arrow{d} \arrow[phantom]{dr}{\textrm{PB}} & \coprod_{i \in I}U_i \arrow{d} \arrow{r} & \hat U_B \arrow[equal]{d} \arrow{r} & 0 \\
 0 \arrow{r} & \Ker p \arrow{r} \arrow{d} & K \arrow{r} \arrow{d} & \hat U_B \arrow{r} & 0\\
 & 0 & 0 & & \\
\end{tikzcd}
\end{displaymath}
in which the lower left square is a pullback. But $L \in \Gen(\mathcal U)$, by assertion (2), so that $L\in \mathcal{B}$ by Lemma \ref{l:kerf-PU}.
Since $\mathcal{B}$ is closed under taking cokernels, it follows that $\Ker p\in \mathcal{B}$ as desired.

$(1),(2.b),(2.c)\Longrightarrow (2.a)$  Note that proving (2.a) is equivalent to proving that, given any epimorphism $f:\coprod_{U\in\mathcal{U}}U^{(I_U)}\twoheadrightarrow A$, for some family $\{I_U\mid U\in\mathcal{U}\}$ of sets, one has that $\mathcal{A}\left(U,\frac{\Ker f}{{\rm tr}_\mathcal{U}(\Ker f)}\right)=0$.  
Fix such a morphism $f$ and define, in each $I_U$ a relation $\sim$ by the rule: 
$$i\sim j \Leftrightarrow f\iota_{U,i}=f\iota_{U,j},$$
where, for $U\in\mathcal{U}$ and $i\in I_U$,  $\iota_{U,i}$ is the corresponding inclusion of  $U$ into $\coprod_{U\in \mathcal{U}}U^{(I_U)}$.
For every $U\in \mathcal{U}$, take a set $I'_U$ of representatives of the indices in $I_U$ with respect to $\sim$. 
Now take $f':\coprod_{U\in \mathcal{U}}U^{(I'_U)}\to A$ to be the unique morphism satisfying that $f'\iota'_{U,i}=f\iota_{U,i}$,
where $\iota'_{U,i}$ is the inclusion,  for every $U\in \mathcal U$ and $i\in I'_U$, of $U$ into $\coprod_{U\in \mathcal{U}}U^{(I'_U)}$.

Let $\varphi:\coprod_{U\in \mathcal{U}}U^{(I_U)}\to \coprod_{U\in \mathcal{U}}U^{(I'_U)}$ be the morphism given by the property that, 
for every $U\in \mathcal{U}$ and $i\in I_U$, $\varphi \iota_{U,i}=\iota'_{U,j}$, where $j\in I'_U$ is such that $j\sim i$. 
Note that $\varphi$ is a retraction. Indeed, if for each $j\in I'_U$ we denote by $[j]$ the equivalence class of $j$ with respect to $\sim$, then we have an induced summation map $s_j:U^{([j])}\longrightarrow U$, which is clearly a retraction. We then get an induced morphism $\varphi_U:U^{(I_U)}=\coprod_{j\in I'_U}U^{([j])}\longrightarrow\coprod_{j\in I'_U}U=U^{(I'_U)}$ which is the coproduct of all summation maps $s_j$ as $j$ runs through $I'_U$. Therefore $\varphi_U$ is also a retraction, and it is left as an easy exercise to check that $\varphi$ is just the coproduct map $\varphi =\coprod_{U\in\mathcal{U}}\varphi_U:\coprod_{U\in\mathcal{U}}U^{(I_U)}\longrightarrow\coprod_{U\in\mathcal{U}}U^{(I'_U)}$. Therefore $\varphi$ is a retraction.

On the other hand, by definitions of $\varphi$ and $f'$, we get that, for each $U\in\mathcal{U}$ and each $i\in I_U$, we have $f'\varphi\iota_{U,i}=f'\iota'_{U,j}=f\iota_{U,j}$, where $j\in I'_U$ is the only element such that $i\sim j$.  By the definition of this equivalence relation, we have that $f\iota_{U,j}=f\iota_{U,i}$. It follows that $f'\varphi\iota_{U,i}=f\iota_{U,i}$, for all $U\in\mathcal{U}$ and all $i\in I_U$. By the universal property of coproducts, this means that $f'\varphi =f$. Hence $f'$ is an epimorphism since so is $f$.

Consider next the pullback of ${\rm ker}\,f'$ and $\varphi$ in order to get the following commutative diagram with exact rows and columns:
\begin{displaymath}
\begin{tikzcd}
 & 0 \arrow{d} & 0 \arrow{d} & & \\
 & V \arrow[equal]{r} \arrow{d} & V \arrow{d} & & \\
 0 \arrow{r} & \Ker f \arrow{r} \arrow{d} \arrow[phantom]{dr}{\textrm{PB}} & \coprod_{U \in \mathcal{U}}U^{(I_U)} \arrow{d}{\varphi} \arrow{r}{f} & A \arrow[equal]{d} \arrow{r} & 0 \\
 0 \arrow{r} & \Ker f' \arrow{r} \arrow{d} & \coprod_{U \in \mathcal{U}}U^{(I'_U)} \arrow{r}{f'} \arrow{d} & A \arrow{r} & 0\\
 & 0 & 0 & & \\
\end{tikzcd}
\end{displaymath}
Since, by our last comments, the central column of this diagram is a split exact sequence, it follows that so is the left column. In particular, we have that 
 $\Ker f\cong V\coprod \Ker f'$. 
Then we have $$\frac{\Ker f}{{\rm tr}_{\mathcal{U}}(\Ker f)}\cong 
\frac{V}{{\rm tr}_{\mathcal{U}}(V)}\coprod \frac{\Ker f'}{{\rm tr}_{\mathcal{U}}(\Ker f')}.$$
But since morphisms in $\text{Sum}(\mathcal{U})$ have kernels in $\Gen(\mathcal{U})$, we get that  $V={\rm tr}_{\mathcal{U}}(V)$ and so $$\frac{\Ker f}{{\rm tr}_{\mathcal{U}}(\Ker f)}\cong 
\frac{\Ker f'}{{\rm tr}_{\mathcal{U}}(\Ker f')}.$$

Note that $f'$ has no redundancy in the coproduct exponent of its domain, 
since for every $U\in \mathcal{U}$, we have $f'\iota'_{U,i}\neq f'\iota'_{U,j}$ for every $i,j\in I'_U$ with $i\neq j$. 
Replacing $I_U$ by $I'_U$ and $f$ by $f'$ if necessary, we may and shall assume that $f:\coprod_{U\in \mathcal{U}}U^{(I_U)}\to A$ has the property that, for every $U\in \mathcal{U}$, 
$f\iota_{U,i}\neq f\iota_{U,j}$ whenever $i,j\in I_U$ and $i\neq j$. 
This means that the map $I_U\to \mathcal{A}(U,A)$, given by $i\mapsto f\iota_{U,i}$ is injective for all $U\in\mathcal{U}$, something that allows us to assume, without loss of generality, that
 $I_U\subseteq \mathcal{A}(U,A)$ for every $U\in \mathcal{U}$. 

Now, consider the following square:
\[\SelectTips{cm}{}
\xymatrix{
\coprod_{U\in \mathcal{U}}U^{(I_U)} \ar[r]^-f \ar[d]_w & A \ar@{=}[d] \\
\coprod_{U\in \mathcal{U}}U^{(\mathcal{A}(U,A))} \ar[r]_-{\varepsilon_A} \ar[r] & A
}\]
where $w$ is the inclusion, which is a section. Then for every $U\in \mathcal{U}$ and every $\alpha\in I_U\subseteq \mathcal{A}(U,A)$, we have
$\varepsilon_A w\iota_{U,\alpha}=\varepsilon_A\iota_{U,\alpha}=\alpha=f\iota_{U,\alpha},$ 
where $\iota_{U,\alpha}$ is first considered into 
$\coprod_{U\in \mathcal{U}}U^{(I_U)}$ and then into $\coprod_{U\in \mathcal{U}}U^{(\mathcal{A}(U,A))}$. Hence the square is commutative. 

Consider the pullback of ${\rm ker}\,\varepsilon_A$ and $w$ in order to get the following commutative diagram 
with exact rows:
\begin{displaymath}
\begin{tikzcd}
0 \arrow{r} & \Ker f \arrow{r} \arrow{d}{k} \arrow[phantom]{dr}{\textrm{BC}} & \coprod_{U \in \mathcal{U}} U^{(I_U)} \arrow{r}{f} \arrow{d}{w} & A \arrow{r} \arrow[equal]{d} & 0\\
0 \arrow{r} & \Ker \varepsilon_A \arrow{r} & \coprod_{U \in \mathcal{U}} U^{(\mathcal{A}(U,A))} \arrow{r}{\varepsilon_A} & A \arrow{r} & 0
\end{tikzcd}
\end{displaymath}
in which the left square is also a pushout. Then we have a short exact sequence:
$$0\longrightarrow \Ker f\stackrel{k}\longrightarrow \Ker \varepsilon_A\longrightarrow 
\coprod_{U\in \mathcal{U}}U^{(\mathcal{A}(U,A)\setminus I_U)}\longrightarrow 0.$$
Bearing in mind that $\mathcal{B}$ is an abelian exact subcategory of $\mathcal{A}$, we further get the following commutative diagram with exact rows:
\[\SelectTips{cm}{}
\xymatrix{
0 \ar[r] & (iq)(\Ker f) \ar[r]^-{(iq)(k)} \ar[d]_{\nu_{\Ker f}} & (iq)(\Ker \varepsilon_A) 
\ar[r] \ar[d]^{\nu_{\Ker \varepsilon_A}} & (iq)\left(\coprod_{U\in \mathcal{U}}U^{(\mathcal{A}(U,A)\setminus I_U)}\right) \ar@{=}[d] & \\
0 \ar[r] & \Ker f \ar[r] & \Ker \varepsilon_A \ar[r] & \coprod_{U\in \mathcal{U}}U^{(\mathcal{A}(U,A)\setminus I_U)} \ar[r] & 0 
}\]
This induces the following commutative diagram with exact rows:
\[\SelectTips{cm}{}
\xymatrix{
0 \ar[r] & (iq)(\Ker f) \ar[r]^-{(iq)(k)} \ar[d]_{\nu_{\Ker f}} & (iq)(\Ker \varepsilon_A) 
\ar[r] \ar[d]^{\nu_{\Ker \varepsilon_A}} & \Coker((iq)(k)) \ar[d] \ar[r] & 0 \\
0 \ar[r] & \Ker f \ar[r] & \Ker \varepsilon_A \ar[r] & \coprod_{U\in \mathcal{U}}U^{(\mathcal{A}(U,A)\setminus I_U)} \ar[r] & 0 
}\]
in which the last vertical arrow is a monomorphism. 
By the Kernel-Cokernel Lemma, it follows that there is an exact sequence 
$0\longrightarrow \Coker(\nu_{\Ker f})\longrightarrow \Coker(\nu_{\Ker \varepsilon_A})$. But $\Img \nu_B=\tr_\mathcal{U}(B)$, for all objects $B\in\mathcal{A}$ (see Lemma \ref{l:Existence_trace}(1)), so that we get a monomorphism $\frac{\Ker f}{{\rm tr}_\mathcal{U}(\Ker f)}\rightarrow\frac{\Ker \varepsilon_A}{{\rm tr}_\mathcal{U}(\Ker \varepsilon_A)}$. Since we are assuming that assertion 2(b) holds, we get that $\mathcal{A}\left(U,\frac{\Ker f}{{\rm tr}_\mathcal{U}(\Ker f)}\right)=0$, for all $U\in\mathcal{U}$, thus ending the proof. 
\end{proof}

We finish this section by characterizing when a coreflective subcategory of an AB5 abelian category is Grothendieck. 

\begin{proposition} \label{prop.Grothendieck-coreflective}
Let $\mathcal{A}$ be an AB5 abelian category and let $\mathcal{B}$ be an additive subcategory of $\mathcal{A}$. The following assertions are equivalent:
\begin{enumerate}
\item $\mathcal{B}$ is a Grothendieck coreflective subcategory of $\mathcal{A}$.
\item There exists a set $\mathcal{U}$ of objects of $\mathcal{A}$ such that $\mathcal{B}=\Pres(\mathcal{U})$ and:
\begin{enumerate}
\item $\mathcal{B}$ is closed under taking cokernels.

\item For every morphism $f:B \rightarrow B'$ in $\mathcal{B}$, the induced morphism $\overline f:\frac{B}{\tr_{\mathcal{U}}(\Ker f)} \rightarrow \Img f$ satisfies that $\mathcal A(U,\overline f)$ is an isomorphism for each $U \in \mathcal U$.

\item For every direct system $(u_i:B_i \rightarrow B'_i \mid i \in I)$ of morphisms in $\mathcal{B}$ satisfying that $\mathcal{A}(?,\Ker u_i)_{| \mathcal{U}}=0$,  for all $i \in I$, we also have that $\mathcal{A}\left(?,\varinjlim \Ker u_i\right)_{| \mathcal{U}}=0$. 
\end{enumerate}
\end{enumerate}
\end{proposition}

\begin{proof}
It is a consequence of Theorem \ref{t.coreflective-abelian}, Proposition \ref{prop.abelian-implies-PU} and Corollary \ref{cor:exactness of colimits}. 
\end{proof}

\begin{corollary} \label{cor.Grothendieck-abelian-exact}
Let $\mathcal{A}$ be an AB5 abelian category and let $\mathcal{B}$ be an abelian exact subcategory closed under taking coproducts. Then $\mathcal{B}$ is  AB5, and it is a Grothendieck category if and only if  $\mathcal{B}=\PU$, for some set $\mathcal{U}$ of objects. In this latter case $\mathcal{B}$ is a coreflective subcategory. 
\end{corollary}

\begin{proof}
Since $\mathcal{B}$ is closed under taking cokernels and coproducts it is closed under taking all colimits, in particular under taking direct limits. Then, using the exactness and fully faithful condition of the inclusion functor $\iota :\mathcal{B}\rightarrow\mathcal{A}$, we deduce that $\mathcal{B}$ is AB5.

The rest follows from Proposition \ref{prop.abelian-implies-PU} and   Theorem \ref{t:PresCoreflective}. 
\end{proof}

\setcounter{section}{7}

\renewcommand{\thesection}{\arabic{section}*}

\section{Dualizing the results: the category $\textrm{Copres}(\mathcal U)$}

The main results of Section 8 admit duals, where the reject is involved. But they are not immediately visible since the trace and reject are not dual to each (see Subsection \ref{sec.trace-reject}).  For this reason, in this section we make explicit those duals without proofs.

All throughout the section $\mathcal A$ is an AB3$^*$ abelian category and $\mathcal U$ a set of objects of $\mathcal A$. As usual, for each morphism $f:A\rightarrow B$ in $\mathcal A$, we denote by $f^c:B\rightarrow\Coker f$ its cokernel.  We denote by $\underline U_A$ the object $\prod_{U \in \mathcal U}U^{\mathcal A(A,U)}$, for each object $A \in \mathcal A$, and by $\mu_A$ the canonical morphism from $A$ to $\underline U_A$. Notice that the kernel of $\mu_A$ is precisely $\rej_{\mathcal U}(A)$ and that $\Img \mu_A\subseteq (\mu_A^c)^{-1}(\rej_{\mathcal U}(\Coker \mu_A))$.

We start with the dual of Lemma \ref{l:PU-precover}:

\begin{lemma}
Let $\mathcal A$ be an AB3$^*$ abelian category and let $\mathcal U$ be a set of objects of $\mathcal A$. Then, for every object $A$ of $\mathcal A$, the composed morphism
\begin{displaymath}
\lambda_A:A \rightarrow \Img \mu_A \rightarrow (\mu_A^c)^{-1}(\rej_{\mathcal U}(\Coker \mu_A))
\end{displaymath}
is a functorial $\Copres \mathcal U$-preenvelope, whose cokernel is isomorphic to $\rej_{\mathcal U}(\Coker \mu_A)$.
\end{lemma}

 Now we state the dual of Theorem \ref{t:PresCoreflective}:

\begin{theorem}
Let $\mathcal A$ be an AB3$^*$ abelian category and let $\mathcal U$ be a set of objects of $\mathcal A$. Denote by $\mathcal B=\Copres(\mathcal U)$. Then, the following are equivalent:
\begin{enumerate}
\item $\mathcal B$ is reflective in $\mathcal A$.

\item $\mathcal B$ is closed under taking kernels in $\mathcal A$.

\item
\begin{enumerate}
\item $\mathcal B$ is closed under taking direct summands.

\item For each $A \in \mathcal A$, $(\mu_A^c)^{-1}(\rej_\mathcal{U}^2(\Coker \mu_A))$ is a direct summand of  $(\mu_A^c)^{-1}(\rej_\mathcal{U}(\Coker \mu_A))$.
\end{enumerate}
\item For every $A \in \mathcal A$ the following conditions hold:
\begin{enumerate}
\item 
$(\mu_A^c)^{-1}(\rej_{\mathcal U}^2(\Coker \mu_A))\in\mathcal B$.

\item $\mathcal A(\rej_{\mathcal U}^2(\Coker \mu_A),U)=0$, for all $U \in \mathcal U$.
\end{enumerate}
\end{enumerate}
Moreover, if all these conditions are satisfied, then the left adjoint functor $\sigma$ of the inclusion functor $\iota:\mathcal B \rightarrow \mathcal A$ acts on objects, up to isomorphism, as $\sigma(A)=(\mu_A^c)^{-1}(\rej^2_{\mathcal U}(\Coker \mu_A))$, where $\mu_A^c$ is the cokernel of $\mu_A$.
\end{theorem} 

\begin{remark}
Note that condition (3.b) in the last theorem is equivalent to saying that 
the composition
\begin{displaymath}
\begin{tikzcd}
(\mu_A^c)^{-1}(\rej_{\mathcal U}(\Coker \mu_A))  \arrow{r}{\mu_A^c} & \rej_{\mathcal U}(\Coker \mu_A) \arrow{r}{p} & \frac{\rej_{\mathcal U}{(\Coker \mu_A)}}{\rej^2_{\mathcal U}(\Coker \mu_A)}
\end{tikzcd}
\end{displaymath}
is a retraction, where $\mu_A^c$ is induced by the cokernel of $\mu_A$ and $p$ is the canonical projection.
\end{remark}
We next establish the dual of Theorem \ref{t.coreflective-abelian}:

\begin{theorem}
Let $\mathcal A$ be an AB3$^*$ abelian category and let $\mathcal U$ be a set of objects of $\mathcal A$. Denote by $\mathcal B=\Copres(\mathcal U)$. Then, the following are equivalent:
\begin{enumerate}
\item $\mathcal B$ is a reflective subcategory of $\mathcal A$ which is abelian.

\item $\mathcal B$ is closed under taking kernels and, for every morphism $f:B \rightarrow B'$ in $\mathcal B$, the inclusion $i:\Img f \rightarrow (f^c)^{-1}(\rej_{\mathcal U}(\Coker f))$, where $f^c$ is the cokernel of $f$, satisfies that $\mathcal A(i,U)$ is an isomorphism for each $U \in \mathcal U$.

\item $\mathcal B$ is closed under taking kernels and, for each morphism $f:B \rightarrow B'$ in $\mathcal B$, the inclusion $i:\Img f \rightarrow (f^c)^{-1}(\rej_{\mathcal U}(\Coker f))$ is a $\mathcal B$-reflection.
\item
\begin{enumerate}
\item $\mathcal A(\rej_{\mathcal U}(\Coker \mu_A),U)=0$ for all $A \in \mathcal A$.

\item For every morphism $g:B \rightarrow \widetilde U$, where $B\in\mathcal B$ and $\widetilde U \in \Prod(\mathcal U)$, the inclusion $i:\Img g \rightarrow (g^c)^{-1}(\rej_{\mathcal U}(\Coker g))$ satisfies that $\mathcal A(i,U)$ is an isomorphism for all $U \in \mathcal U$.
\end{enumerate}
\end{enumerate}
Moreover, if all these conditions are satisfied, then the left adjoint functor $\sigma$ of the inclusion functor $\iota:\mathcal B \rightarrow \mathcal A$ acts on objects, up to isomorphisms, as $\sigma(A)=(\mu_A^c)^{-1}(\rej_{\mathcal U}(\Coker \mu_A))$.
\end{theorem}

Finally, we state the dual of Theorem \ref{t:abexact}:

\begin{theorem}
Let $\mathcal A$ be an AB3$^*$ abelian category and let $\mathcal U$ be a set of objects of $\mathcal A$. Denote by $\mathcal B=\Copres(\mathcal U)$. Then, the following are equivalent:
\begin{enumerate}
\item $\mathcal B$ is a reflective abelian exact subcategory of $\mathcal A$.

\item Each morphism $f:\widehat U \rightarrow \widetilde U$ in $\Prod(\mathcal U)$ has its cokernel in $\Cogen(\mathcal U)$ and one (resp. all) of the following conditions hold:
\begin{enumerate}
\item For each $\widetilde V \in \Prod(\mathcal U)$ and each subobject $K$ of $\widetilde V$, $\mathcal A(\rej_{\mathcal U}(\widetilde V/K),-)_{| \mathcal U}=0$.

\item $\mathcal A(\rej_{\mathcal U}(\Coker \mu_A),-)_{| \mathcal U}=0$, for all $A \in \mathcal A$.

\item For each $B \in \mathcal B$, $\mu_B$ has its cokernel in $\Cogen(\mathcal U)$.
\end{enumerate}
\end{enumerate}
\end{theorem}

\renewcommand{\thesection}{\arabic{section}}

\section{Other examples of coreflective subcategories of abelian categories}

\subsection{The abelian condition of $\Gen(\mathcal{U})$}

All throughout this section, $\mathcal{A}$ is an AB3 abelian category. If $\mathcal{U}$ is a set of objects in $\mathcal{A}$, then $\Gen(\mathcal{U})$ is a coreflective subcategory. Indeed, the trace functor $\tr_\mathcal{U} :\mathcal{A}\rightarrow\Gen(\mathcal{U})$, acting on objects as $A\rightsquigarrow \tr_\mathcal{U} (A)$, is a right adjoint of the inclusion $\iota :\Gen(\mathcal{U})\rightarrow\mathcal{A}$. The results of the previous sections allow us to determine when $\Gen(\mathcal{U})$ is abelian or abelian exact. 

\begin{proposition} \label{prop.GenU-abelian}
In the situation of the last paragraph, the following assertions are equivalent:

\begin{enumerate}
\item $\Gen (\mathcal{U})$ is abelian;
\item $\Gen(\mathcal{U})$ is an abelian exact subcategory of $\mathcal{A}$;
\item  $\Gen(\mathcal{U})$ is closed under taking subobjects in $\mathcal{A}$;
\item Each subobject of an object in $\mathrm{Sum} (\mathcal{U})$ is in $\Gen (\mathcal{U})$.
\end{enumerate}
When $\mathcal{A}$ is AB5, these assertions are also equivalent to:
\begin{enumerate}
\item[(5)] For each finite family $\{U_1,\ldots, U_n\}$ of objects of $\mathcal{U}$, all subobjects of $\coprod_{i=1}^nU_i$ are in $\Gen (\mathcal{U})$. 
\end{enumerate}
When $\mathcal{A}$ is AB5 and these equivalent conditions hold, $\Gen(\mathcal{U})$ is a Grothendieck category.
\end{proposition}

\begin{proof} 
$(1)\Longrightarrow (4)$ $\mathcal{U}$ is clearly a set of generators of $\Gen (\mathcal{U})$. Then, by Proposition \ref{prop.abelian-implies-PU} and its proof, we know that $\Gen (\mathcal{U})=\PU$ and, even more,  the projection $p:\coprod_{i\in I}U_i\twoheadrightarrow \coprod_{i\in I}U_i/K$ is the cokernel of its kernel in $\mathcal{B}$, where     $\{U_i\mid i \in I\} \subseteq \mathcal{U}$ and $K$ is a subobject of $\coprod_{i\in I}U_i$. But the kernel of $p$ in $\PU$ is the composition $\tr_\mathcal{U}(K)\stackrel{}{\hookrightarrow} K\hookrightarrow\coprod_{i\in I}U_i$. Since $\Gen (\mathcal{U})$ is closed under taking cokernels in $\mathcal{A}$, the cokernel of that kernel is then the projection $\coprod_{i\in I}U_i\twoheadrightarrow\coprod_{i\in I}U_i/ \tr_\mathcal{U}(K)$. Therefore the canonical (epi)morphism $\coprod_{i\in I}U_i/ \tr_\mathcal{U}(K)\rightarrow \coprod_{i\in I}U_i/K$ is an isomorphism, and so $\tr_\mathcal{U} (K)=K\in\Gen (\mathcal{U})$. 

$(4)\Longrightarrow (3)$ Let $\{U_i\mid i \in I\}$ be any family of objects of $\mathcal U$. Then, any subobject of a quotient $\coprod_{i\in I}U_i/K$ is, due to an obvious pullback construction,  an epimorphic image of a subobject of $ \coprod_{i\in I}U_i$ and, hence, it is in $\Gen (\mathcal{U})$. 

$(3)\Longrightarrow (2)\Longrightarrow (1)$  are clear. 

$(4)\Longrightarrow (5)$ is clear, even if $\mathcal{A}$ is not AB5.

$(5)\Longrightarrow (4)$ (when $\mathcal{A}$ is AB5) Let $\{U_i\mid i\in I\}$ be a family of objects of $\mathcal U$ and $K$ be a subobject of $\coprod_{i \in I}U_i$. Due to the AB5 condition,  $K$ is isomorphic to the direct limit of the intersections $K\cap U_F$, where $F$ is a finite subset of $I$ and $U_F:=\coprod_{i\in F}U_i$. By assertion (5), each  $K\cap U_F$ is in $\Gen (\mathcal{U})$ and so $K\in\Gen(\mathcal{U})$. 
 
The last statement of the proposition is a direct consequence of Corollary \ref{cor.Grothendieck-abelian-exact}.
\end{proof}

\subsection{The subcategory $\Gen ({\rm fg}(\mathcal{A}))$}

When $\mathcal{A}$ is an AB5 abelian category, an object $X$ is said to be \emph{finitely generated} if the functor $\mathcal{A}(X,?):\mathcal{A}\rightarrow\Ab$ preserves direct limits of direct systems $(A_i,u_{ij} \mid i < j \in I)$ such that each $u_{ij}$ is a monomorphism for each $i<j$ in $I$.  We denote by ${\rm fg}(\mathcal{A})$ the subcategory consisting of all finitely generated objects of $\mathcal{A}$. 

\begin{proposition} \label{prop.Gen(fgA)}
Let $\mathcal{G}$ be a Grothendieck category. The following assertions hold:

\begin{enumerate}
\item ${\rm fg}(\mathcal{G})$ is a skeletally small subcategory.
\item $\Gen ({\rm fg}(\mathcal{G}))$ is  a coreflective subcategory, and we have $$\Gen ({\rm fg}(\mathcal{G}))=\Pres ({\rm fg}(\mathcal{G}))=\varinjlim{\rm fg}(\mathcal{G}),$$ where $\varinjlim{\rm fg}(\mathcal{G})$ denotes the subcategory consisting of all objects that are isomorphic to direct limits of objects belonging to ${\rm fg}(\mathcal{G})$.
\item  $\Gen ({\rm fg}(\mathcal{G}))$ is abelian (exact) if, and only if, each subobject of a finitely generated object is an epimorphic image of a coproduct of finitely generated objects. In this case $\Gen ({\rm fg}(\mathcal{G}))$ is a locally finitely generated Grothendieck category.
\end{enumerate}
\end{proposition}

\begin{proof}
(1) Fix a generator $G$ of $\mathcal{G}$. If $X\in{\rm fg}(\mathcal{G})$ and we fix an epimorphism $p:G^{(I)}\twoheadrightarrow X$, then $X$ is the direct union of all images $\Img(p i_F)$, where $F$ runs through the finite subsets of $I$ and $i_F:G^{(F)}\rightarrow G^{(I)}$ is the obvious section. We then get some $E\subseteq I$ finite such that $X=\Img(p i_E)$.  It follows that $X$ is isomorphic to $G^n/K$, for some $n\in\mathbb{N}$ and some subobject $K\subseteq G^n$. Since the subobjects of each $G^n$ form a set, assertion (1) follows.

(2) If $\mathcal{U}$ is a set of representatives of the isoclasses of objects in ${\rm fg}(\mathcal{G})$, then we have that $\Gen ({\rm fg}(\mathcal{G}))=\Gen (\mathcal{U})$, that is a coreflective subcategory by the previous subsection.  Let $T\in\Gen({\rm fg}(\mathcal{G}))$ and fix an epimorphism $q:\coprod_{i\in I}X_i\twoheadrightarrow T$, where all the $X_i$ are in ${\rm fg}(\mathcal{G})$. Putting $X_F:=\coprod_{i\in F}X_i$, for each finite subset $F\subseteq I$, and arguing as in the proof of assertion (1), we get that $T=\bigcup_{F\subseteq I\text{, }F\text{ finite}}\Img(p i_F)$, where $i_F:X_F\rightarrow\coprod_{i\in I}X_i$ is the obvious section. But ${\rm fg}(\mathcal{G})$ is closed under taking quotients (see \cite[Lemma V.3.1]{Stenstrom}), so that $T_F:=\Img(pi_F)\in{\rm fg}(\mathcal{G})$ for all $F\subseteq I$ finite. It follows that $\Gen ({\rm fg}(\mathcal{G}))\subseteq\varinjlim{\rm fg}(\mathcal{G})$. But, by the explicit construction of colimits, we have that $\varinjlim{\rm fg}(\mathcal{G})\subseteq\Pres ({\rm fg}(\mathcal{G}))$. Since we obviously have that $\Pres ({\rm fg}(\mathcal{G}))\subseteq\Gen({\rm fg}(\mathcal{G}))$ assertion (2) immediately follows. 

(3) This is a direct consequence of the equivalence of assertions (1), (2) and (5) in Proposition \ref{prop.GenU-abelian} and the fact that the class $\mathcal{U}$ above is actually a set consisting of finitely generated generators of $\Gen({\rm fg}(\mathcal{G}))$. 
\end{proof}

\subsection{Abelian coreflective subcategories with a set of projective generators}

In this subsection we want to identify the coreflective abelian subcategories of any AB3 abelian category $\mathcal{A}$ that have a set of projective generators, and to determine those of them that are module categories over a small preadditive category or over an (associative unital) ring. We need the following notions:

\begin{definition} \label{def.selfsmall-quasiprojective}
Let $\mathcal{U}$ be a set of objects in an AB3 abelian category $\mathcal{A}$. We say that:

\begin{enumerate}
\item $\mathcal{U}$ is \emph{self-small}, if for each family $\{U_i\mid i\in I\} \subseteq \mathcal U$ and each object $U$ of $\mathcal{U}$, the canonical map $\coprod_{i\in I}\mathcal{A}(U,U_i)\rightarrow\mathcal{A}(U,\coprod_{i\in I}U_i)$ is an isomorphism.

\item $\mathcal{U}$ is \emph{$\Sigma$-quasi-projective}, if any epimorphism $p:\tilde{U}\twoheadrightarrow A$ in $\mathcal A$, with $\tilde{U}\in\mathrm{Sum}(\mathcal{U})$, is an $\mathcal{U}$-epimorphism.

\item An object $X$ of $\mathcal{A}$ is \emph{$\mathcal{U}$-subgenerated}  when it is isomorphic to a subobject of an object in $\Gen (\mathcal{U})$.  We will denote by $\sigma[\mathcal{U}]$ the subcategory of $\mathcal A$ consisting of all $\mathcal{U}$-subgenerated objects.
\end{enumerate}
We say that an object $U$ is \emph{self-small} (resp. \emph{$\Sigma$-quasi-projective}) if the set $\{U\}$ is self-small (resp. $\Sigma$-quasi-projective).
\end{definition}

Now we characterize coreflective subcategories $\mathcal{B}$ of AB3 abelian categories $\mathcal{A}$ which are abelian and have a set of projective generators (in $\mathcal{B}$, but not necessarily in $\mathcal{A}$).

\begin{proposition} \label{prop.coreflabelian-with-projectgenerators}
Let $\mathcal{A}$ be an AB3 abelian category and $\mathcal{B}$ be an additive subcategory of $\mathcal{A}$. The following assertions are equivalent:

\begin{enumerate}
\item $\mathcal{B}$ is an abelian coreflective subcategory of $\mathcal A$ with a set of generators which are projective in $\mathcal B$.
\item $\mathcal{B}$ is closed under taking cokernels in $\mathcal{A}$ and there is a set $\mathcal{U}$ of objects that satisfies the following properties:

\begin{enumerate}
\item $\mathcal{B}=\PU$.
\item If $\tilde{U}\in\Sum (\mathcal{U})$ and $Y$ is a subobject of an object in $\PU$, then any epimorphism $p:\tilde{U}\twoheadrightarrow Y$ is an $\mathcal{U}$-epimorphism and  $\mathcal{A}\left(?,\frac{\Ker p}{\tr_\mathcal{U}(\Ker p)}\right)_{| \mathcal{U}}=0$.
\end{enumerate}
\end{enumerate}
\end{proposition}

\begin{proof}
$(1)\Longrightarrow (2)$ $\mathcal{B}$ is closed under taking cokernels since it is coreflective. By Proposition \ref{prop.abelian-implies-PU} and its proof, we know that $\mathcal{B}=\PU$, where $\mathcal{U}$ is the set of projective generators of $\mathcal{B}$. Then we automatically have condition 2(a). On the other hand, by assertion (2) in Theorem \ref{t.coreflective-abelian}, we have that the epimorphism $p$ of condition 2(b) satisfies $\mathcal{A}\left(?,\frac{\Ker p}{\tr_\mathcal{U}(\Ker p)}\right)_{| \mathcal{U}}=0$. 

It remains to see that $p$ is an $\mathcal{U}$-epimorphism. This is obvious when $Y\in\mathcal{B}$, since the objects of $\mathcal{U}$  are projective in $\mathcal{B}$.  This also proves that the canonical epimorphism $\pi:\tilde{U}\twoheadrightarrow\frac{\tilde{U}}{\tr_{\mathcal{U}}(\Ker p)}$ is an $\mathcal{U}$-epimorphism, even when $Y\not\in\mathcal{B}$. But, in this latter case,  assertion (2) of Theorem \ref{t.coreflective-abelian} implies that the induced morphism $\bar{p}:\frac{\tilde{U}}{\tr_\mathcal{U}(\Ker p)}\rightarrow Y$ is an $\mathcal{U}$-epimorphism.  Consequently, $p=\bar{p}\pi$ is an $\mathcal{U}$-epimorphism since it is the composition of two $\mathcal{U}$-epimorphisms. 

$(2)\Longrightarrow (1)$ Note that, by Theorem \ref{t:PresCoreflective},  $\mathcal{B}$ is coreflective since $\mathcal{B}=\PU$ is closed under taking cokernels. We check that assertion (2) of Theorem \ref{t.coreflective-abelian} holds, i.e., for each morphism $f:B\rightarrow B'$ in $\mathcal{B}$,  the induced map $\overline{f}:\frac{B}{\tr_\mathcal{U}(\Ker f)}\rightarrow \Img f$ satisfies that $\mathcal A(U,\overline f)$ is an isomorphism for each $U \in \mathcal U$. Note that, by our condition 2(b), the composition $\hat{U}_B\stackrel{\varepsilon_B}{\rightarrow}B\stackrel{\pi}{\rightarrow}\frac{B}{\tr_\mathcal{U}(\Ker f)}\stackrel{\overline{f}}{\rightarrow}\Img f$, where $\pi$ is the projection,  is an $\mathcal{U}$-epimorphism, which implies that so is $\overline{f}$. It remains to prove that $\mathcal{A}\left(U,\frac{\Ker f}{\tr_\mathcal{U}(\Ker f)}\right)=0$, for all $U\in\mathcal{U}$. 

Note that condition 2(b) gives that the objects of $\mathcal{U}$ are projective in $\mathcal{B}$. Note also that $\varepsilon_B$ induces an isomorphism $\frac{\hat{U}_B}{\Ker (f\varepsilon_A)}\cong\frac{B}{\Ker f}\cong \Img f$. We then get the following commutative diagram,  with exact rows that are kept exact by applying the functor $\mathcal{A}(U,?)$, for all $U\in\mathcal{U}$:
\begin{displaymath}
\begin{tikzcd}
0 \arrow{r} & \frac{\Ker(f \varepsilon_B)}{\tr_{\mathcal{U}}(\Ker(f\varepsilon_B))} \arrow{r} \arrow{d} & \frac{\widehat{U}_B}{\tr_{\mathcal{U}}(\Ker(f \varepsilon_B))} \arrow{r} \arrow{d} & \Img f \arrow{r} \arrow[equal]{d} & 0\\
0 \arrow{r} & \frac{\Ker f}{\tr_{\mathcal{U}}(\Ker f)} \arrow{r} & \frac{B}{\tr_{\mathcal{U}}(\Ker f)} \arrow{r}{\overline f} & \Img f \arrow{r} & 0
\end{tikzcd}
\end{displaymath}
The left square is then bicartesian, in particular the two left most vertical arrows are epimorphisms, and these properties are also preserved when applying any functor  $\mathcal{A}(U,?):\mathcal{A}\rightarrow\Ab$, with $U\in\mathcal{U}$. This implies that the induced morphism $\mathcal{A}\left(U,\frac{\Ker (f\varepsilon_B)}{\tr_\mathcal{U}(\Ker (f\varepsilon_B))}\right)\rightarrow\mathcal{A}\left(U,\frac{\Ker f}{\tr_\mathcal{U}(\Ker f)}\right)$ is an epimorphism. But, by our condition 2(b), we know that $\mathcal{A}\left(U,\frac{\Ker (f\varepsilon_B)}{\tr_\mathcal{U}(\Ker (f\varepsilon_B))}\right)=0$, and hence $\mathcal{A}\left(U,\frac{\Ker f}{\tr_\mathcal{U}(\Ker f)}\right)=0$, for all $U\in\mathcal{U}$.
\end{proof}

As an immediate consequence, we get:

\begin{corollary} \label{cor.coreflective-modulecategory}
Under the hypotheses of the last proposition, the following assertions are equivalent:
\begin{enumerate}
\item $\mathcal{B}$ is a coreflective subcategory that is (equivalent to) the module category over a small preadditive category.
\item $\mathcal{B}$ is closed under taking cokernels in $\mathcal{A}$ and there is a self-small set of objects   $\mathcal{U}$   that satisfies condition (2) of the mentioned proposition.
\end{enumerate}
In such situation $\mathcal{B}$ is the module category over a ring if, and only if, the set $\mathcal{U}$ in assertion (2) may be chosen to contain just one object. 
\end{corollary}

\begin{proof}
It immediately follows by Proposition \ref{prop.coreflabelian-with-projectgenerators} and \cite[Proposition 3.14]{ParraSaorin}.
\end{proof}

As a straightforward consequence of Theorem \ref{t:abexact}, we can characterize the coreflective abelian exact subcategories of an AB3 abelian category which have a set of projective generators.

\begin{corollary} \label{cor.abexact-with-set-projgenerators}
Let $\mathcal{A}$ be an AB3 abelian category and let $\mathcal{B}$ be an additive subcategory of $\mathcal{A}$. The following assertions are equivalent:

\begin{enumerate}
\item $\mathcal{B}$ is a coreflective abelian exact subcategory with a set of projective generators.
\item There is a set $\mathcal{U}$ of objects that satisfies the following properties:

\begin{enumerate}
\item $\mathcal{B}=\PU$;
\item Each epimorphism $p:\tilde{U}\twoheadrightarrow B$, with $\tilde{U}\in\Sum (\mathcal{U})$ and $B\in\mathcal{B}$,  is an $\mathcal{U}$-epimorphism with kernel in $\Gen(\mathcal{U})$;
\item Each morphism in $\Sum(\mathcal{U})$ has its kernel in $\Gen (\mathcal{U})$. 
\end{enumerate}
\end{enumerate}

In such case, $\mathcal{B}$ is a module category over a preadditive category (resp. over a ring) if, and only if, the set $\mathcal{U}$ may be chosen to be self-small (resp. self-small and having just one object).
\end{corollary}

We end this subsection by considering the case when $\mathcal{U}$ is a $\Sigma$-quasi-projective set of objects, in which case some of the conditions of the preceding results get rather simplified. We will use the following result, which is well-known in module categories (see \cite[Subsection 9.45.11 and Theorem 3.18.3]{Wisbauer}), and its proof is valid also in our context. 

\begin{lemma} \label{lem.Sigma-QP are projective}
Let $\mathcal{U}$ be a set of objects in an AB3 abelian category $\mathcal A$. The subcategory $\sigma[\mathcal{U}]$ is a coreflective abelian exact subcategory of $\mathcal{A}$. Moreover, the following assertions are equivalent:
\begin{enumerate}
\item $\mathcal{U}$ is $\Sigma$-quasi-projective in $\mathcal{A}$.
\item $\mathcal{U}$ is a set of projective objects of $\sigma [\mathcal{U}]$.
\end{enumerate}
\end{lemma}

\begin{proposition} \label{prop.PresU-for-SigmaQP}
Let $\mathcal{A}$ be an AB3 abelian category, let $\mathcal{U}$ be a $\Sigma$-quasi-projective set of objects (e.g., a set of projective objects) and let $\mathcal{B}=\PU$. The following assertions hold true:
\begin{enumerate}
\item $\mathcal{B}$ is a coreflective abelian subcategory with a set of projective generators.
\item $\mathcal{B}$ is an abelian exact subcategory if, and only if, any morphism in $\Sum(\mathcal{U})$ has its kernel in $\Gen(\mathcal{U})$.
\item When $\mathcal{U}$ is self-small,  $\mathcal{B}$ is the module category over a small preadditive category. Actually, $\mathcal{B}\cong \mathrm{Mod}\textrm{-}\mathcal{U}$.
\item If $\mathcal{U}$ is finite and self-small, then $\mathcal{B}$ is the module category over a ring. Actually,  $\mathcal{B}\cong \mathrm{Mod}\textrm{-}\mathrm{End}_\mathcal{A}(\tilde{U})$ for $\tilde{U}=\coprod_{U\in\mathcal{U}}U$. 
\end{enumerate}
\end{proposition}

\begin{proof}
(1) Due to Lemma \ref{lem.Sigma-QP are projective},  it is enough to prove that $\PU$ is an abelian coreflective subcategory of $\sigma [\mathcal U]$. Since the latter is also AB3 abelian (see Lemma \ref{lem.Sigma-QP are projective}), replacing $\mathcal{A}$ by $\sigma [\mathcal{U}]$ if necessary,  we can assume from now on in the proof of this assertion that $\mathcal{U}$ consists of projective objects. We claim that $\mathcal{A}\left(-,\frac{A}{\tr_{\mathcal{U}}(A)}\right)_{| \mathcal{U}}=0$, for all  objects $A$ of $\mathcal{A}$. This implies that   $\PU$  is coreflective by (4) of   Theorem \ref{t:PresCoreflective}.  Then, it is closed under cokernels and trivially satisfies  conditions 2(a) and 2(b) of Proposition \ref{prop.coreflabelian-with-projectgenerators}. 

 To settle our claim, let  $A$ be any object of $\mathcal{A}$ and take any morphism $f:U \rightarrow \frac{A}{\tr_{\mathcal{U}}(A)}$. There exists a morphism $g:U \rightarrow A$ such that $\pi g=f$, where $\pi:A \rightarrow \frac{A}{\tr_{\mathcal{U}}(A)}$ is the canonical projection. But $\Img g$ is contained in $\tr_{\mathcal{U}}(A)$, which implies that $f=0$.
 
 (2) We have just proved that $\mathcal A\left(U,\frac{A}{\tr_{\mathcal U}(A)}\right)=0$ for each $A \in \mathcal A$ and $U \in \mathcal U$. Then the result follows from  Theorem \ref{t:abexact}. 
 
 (3) It is a consequence of Corollary \ref{cor.coreflective-modulecategory}, taking into account the classical proof of Mitchell theorem (see \cite[Corollary 3.6.4]{Popescu}) and the fact that, by the proof of \cite[Proposition 3.14]{ParraSaorin},  the objects of $\mathcal{U}$ are small in $\PU$. 
 
  (4) If   $\tilde{U}=\coprod_{U\in\mathcal{U}}U$, then we have that $\mathcal{B}=\text{Pres}(\tilde{U})$ and $\tilde{U}$ is a (self-)small projective generator of $\mathcal{B}$. Then $\mathcal{B} (\tilde{U},?):\mathcal{B}\rightarrow\text{Mod-}\text{End}_\mathcal{A}(\tilde{U})$ is an equivalence of categories (see \cite[Corollaries 3.6.4 and 3.7.4]{Popescu}).
\end{proof}

\begin{example} \label{ex.non-abelian exact}
Examples of ($\Sigma$-quasi-)projective sets $\mathcal{U}$ for which $\PU$ is not abelian exact are abundant. For instance, if $R$ is a right Artinian ring of finite right global dimension and  $P$ is an indecomposable projective right $R$-module with maximal Loewy length and such that $\mathrm{End}(P_R)$ is not a division ring, then $\Pres (P)$ is not abelian exact. 
As an example of this situation, take $K$ to be a field and $R$  the finite dimensional $K$-algebra  given by the quiver $\xymatrix{ Q: 1 \ar@<.5ex>[r]^-{\beta} & 2\ar@<.5ex>[l]^-{\alpha} }$ and relation $\alpha\beta =0$. Then $\Pres (e_2R)$ is not abelian exact.
\end{example}

\begin{proof}
 If $\Pres (P)$ were abelian exact, then, by Proposition \ref{prop.PresU-for-SigmaQP},  all morphisms in $\Sum (P)$ would have kernel in $\Gen(P)$. Assuming this latter condition and taking any non-invertible nonzero endomorphism
$f:P\rightarrow P$, we would have that the minimal projective resolution of $M:=\Coker f$ is of the form $$0\rightarrow P_n\rightarrow P_{n-1}\rightarrow ...\rightarrow P\stackrel{f}{\rightarrow}P\rightarrow M\rightarrow 0, $$ with all the $P_i$ in $\text{sum} (P)$. But, with the obvious abuse of terminology, we have that $P_n\subseteq P_{n-1}J(R)$, where $J(R)$ is the Jacobson radical. It would follow that $LL(P)=LL(P_n)\leq LL(P_{n-1}J(R))=LL(P)-1$, where $LL$ denotes the Loewy length. 

As for the particular example, we readily see that $\Omega (S_1)\cong S_2$ and $\Omega (S_2)\cong e_1R$, where $S_i=\frac{e_iR}{e_iJ(R)}$ for $i=1,2$. Therefore, $\mathrm{gl.dim}(R)<\infty$. Moreover, the assignment $x\rightsquigarrow x\beta\alpha$ gives a non-invertible endomorphism $f:e_2R\rightarrow e_2R$.
\end{proof}
 
\subsection{A generalization of a theorem of Gabriel and De la Pe\~na}
 
Now we apply our results to extend a famous result of Gabriel and De la Pe\~na (see \cite[Theorem 1.2]{GabrieldelaPena}).

\begin{lemma} \label{lem.fullyexact-gener.by-a-set}
Let $\mathcal{G}$ be an AB4* Grothendieck category and let $\mathcal{V}$ be a set of objects of $\mathcal{G}$. The smallest fully exact subcategory of $\mathcal{G}$ that contains $\mathcal{V}$ is a bireflective subcategory. 
\end{lemma}

\begin{proof}
 We denote by $\mathcal{B}_V$ the smallest fully exact subcategory that contains $\mathcal{V}$. The proof will be an adaptation of that of \cite[Lemma 1.3]{GabrieldelaPena}, and the first paragraph of it shows that each object $B\in\mathcal{B}_V$ admits an epimorphism $B'\twoheadrightarrow B$, where $B'\in\mathcal{B}_V$ and $B'$ is a subobject of a product of objects in $\mathcal{V}$. 

 Let us fix a generator $G$ of $\mathcal{G}$ and consider, for each $M\in\mathcal{G}$, the canonical morphism  $u_M:M\rightarrow\prod_{V\in\mathcal{V}}V^{\mathcal{G}(M,V)}:=\hat{V}_M$, which is a $\Prod (\mathcal{V})$-preenvelope. We put $u:=u_G$ and denote by $W$ the intersection of all subobjects $X$ of $\hat{V}:=\hat{V}_G$ such that $\Img u\subseteq X$ and $X\in\mathcal{B}_V$. Note that $W\in\mathcal{B}_V$ since $\mathcal{B}_V$ is closed under taking limits in $\mathcal{G}$. 
The goal is to prove that each $B\in\mathcal{B}_V$ which is a subobject of a product of objects in $\mathcal{V}$ is an epimorphic image of a coproduct of copies of $W$. Note that if this is proved then, by the first paragraph of this proof and  Corollary \ref{cor.Grothendieck-abelian-exact}, we will have that $\mathcal{B}_V$ is coreflective and a Grothendieck category. In particular, it has a cogenerator and, by the Special Adjoint Functor Theorem (see \cite[Theorem 3.3.4]{Borceux}), it will follow that $\mathcal{B}_V$ is also reflective and will end the proof.

Let $B\in\mathcal{B}_V$ be an object which is a subobject of a product of objects in $\mathcal{V}$ and let $f:G\rightarrow B$ be any  morphism. The composition $u_Bf$ factors through $u$ since this is a $\Prod (\mathcal{V})$-preenvelope. We fix a morphism $\hat{f}:\hat{V}_G\rightarrow\hat{V}_B$ such that $\hat{f}u=u_Bf$. Note that $u_B$ is a monomorphism and, since $\mathcal{B}_V$ is closed under taking cokernels, we get that $B':=\Coker u_B\in\mathcal{B}_V$. We denote by $\pi$ the projection $\hat{V}_B\twoheadrightarrow B'$, so that we have $\pi \hat{f}u=\pi u_B f=0$, and hence $\Img u\subseteq\Ker (\pi \hat{f})$. But $\Ker (\pi \hat{f})\in\mathcal{B}_V$ since $\mathcal{B}_V$ is closed under taking kernels. It then follows that $W\subseteq\text{Ker}(\pi \hat{f})$, so that $\hat{f}(W)\subseteq\Ker \pi$. This means that we get a map $\tilde{f}:W\rightarrow B\cong\Ker(\pi)$ such that $\hat{f}_{| W}=u_B \tilde{f}$. If we denote by $\tilde{u}$ the obvious composition $G\twoheadrightarrow\Img u\hookrightarrow W$, we get that $u_B \tilde{f}\tilde{u}=\hat{f}_{| W} \tilde{u}=\hat{f} u=u_B f$, which implies that 
 $\tilde{f} \tilde{u}=f$, since $u_B$ is monic. Bearing in mind that $G$ is a generator of $\mathcal{G}$, we get that $$B=\sum_{f\in\mathcal{G}(G,B)}\Img f=\sum_{f\in\mathcal{G}(G,B)}\Img(\tilde{f}\tilde{u})\subseteq\sum_{f\in\mathcal{G}(G,B)}\tilde{f}(W)\subseteq B.$$ These inclusions are then equalities, and we get that the unique morphism $\psi :W^{(\mathcal{A}(G,B))}\rightarrow B$ such that $\psi \iota_f=\tilde{f}$, where $\iota_f:W\rightarrow W^{(\mathcal{A}(G,B))}$ is the canonical injection for all $f\in\mathcal{G}(G,B)$, is an epimorphism. 
\end{proof}

Recall that to any small preadditive category $\mathcal{C}$ one canonically associates a ring with enough idempotents $R_\mathcal{C}$, sometimes called the \emph{(right) functor ring} of $\mathcal{C}$, such that the category $\text{Mod-}\mathcal{C} =[\mathcal{C}^{\rm op},\Ab]$ of right $\mathcal{C}$-modules is equivalent to the category $\textrm{Mod-}R_{\mathcal{C}}$ of unitary right $R_{\mathcal{C}}$-modules, i.e. those right $R_{\mathcal{C}}$-modules $M$ such that $M=MR_{\mathcal{C}}$  (see \cite[Section 3 and Theorem 3.1]{Saorin} for an extended version).  On the other hand, if $R=\bigoplus_{i,j}e_iRe_j$ is a ring with enough idempotents, where $\{e_i\mid i\in I\}$ is a family of nonzero orthogonal idempotents which we fix from now on, one canonically associates to it a unital ring $\hat{R}$, usually called the \emph{Dorroh extension (or overring)}  of $R$. As an additive abelian group, $\hat{R}=R\times\mathbb{Z}$, and the multiplication in $\hat{R}$ is given by the rule $(\alpha ,m)(\beta, n)=(\alpha\beta +n\alpha +m\beta, mn)$ for each $(\alpha,m),(\beta,n)\in \hat R$. Then $R$ is identified with the idempotent two-sided ideal $R\times 0$ of $\hat{R}$ and $\text{Mod-}R$ gets canonically identified with the subcategory $\mathcal{X}$ of $\text{Mod-}\hat{R}$ consisting of those $\hat{R}$-modules $M$ such that $M(R\times 0)=M$ (see \cite[Sections 1.1.5 and 2.6.3]{Wisbauer}). Note that $\mathcal{X}$ is a hereditary torsion class in $\text{Mod-}\hat{R}$ that is the first component of a \emph{TTF triple} $(\mathcal{X},\mathcal{Y},\mathcal{Z})$ in $\text{Mod-}\hat{R}$, i.e. a triple such that $(\mathcal{X},\mathcal{Y})$ and $(\mathcal{Y},\mathcal{Z})$ are both torsion pairs (see \cite[Section 2]{Jans}, where the author uses the term `torsion torsionfree theory'). Abusing the notation, in our next result we will identify $R$ with $R\times 0$ and $\text{Mod-}R$ with $\mathcal{X}=\{X\in\text{Mod-}\hat{R}\text{: }XR=X\}$.

\begin{lemma} \label{lem.bireflectives form a set}
Let $R$ be a ring with enough idempotents and $\hat{R}$ be its Dorroh extension. To each fully exact subcategory $\mathcal{B}$ of $\mathrm{Mod}\textrm{-}R$ we associate the subcategory $\hat{\mathcal{B}}$ of $\mathrm{Mod}\textrm{-}\hat{R}$ consisting of those $\hat{R}$-modules $M$ such that $MR\in\mathcal{B}$. The following assertions hold:

\begin{enumerate}
\item $\hat{\mathcal{B}}$ is a fully exact subcategory of $\mathrm{Mod}\textrm{-}\hat{R}$.
\item If ${\rm Fex}(R)$ and ${\rm Fex}(\hat{R})$ denote the classes of fully exact subcategories of $\mathrm{Mod}\textrm{-}R$ and $\mathrm{Mod}\textrm{-}\hat{R}$ respectively, then the assignment $\mathcal{B}\rightsquigarrow\hat{\mathcal{B}}$ defines an injective map ${\rm Fex}(R)\rightarrowtail {\rm Fex}(\hat{R})$. 
\item ${\rm Fex}(R)$ is a set (as opposite to a proper class). 
\end{enumerate}
\end{lemma}

\begin{proof}
Assertion (3) follows from the other two and from \cite[Theorem 1.2 and Lemma 1.4]{GabrieldelaPena}.  

(1) With the obvious abuse of notation, we have that $e_iR=e_i\hat{R}$ and also that $MR=\tr_P(M)$ is the trace of $P=\coprod_{i\in I}e_i\hat{R}$ in $M$, for each $M\in\text{Mod-}\hat{R}$. In particular, the assignment $M\rightsquigarrow MR$ is the definition on objects of the torsion radical $t:\text{Mod-}\hat{R}\rightarrow\text{Mod-}\hat{R}$ associated to the hereditary torsion class $\mathcal{X}=\text{Mod-}R$. As a consequence, $t$ is left exact by \cite[Proposition VI.3.1]{Stenstrom}. But $t$ also preserves epimorphisms, since $P$ is a projective $\hat{R}$-module. In conclusion, $t$ is an exact functor which clearly preserves coproducts. On the other hand, the induced functor $t:\text{Mod-}\hat{R}\rightarrow\text{Mod-}R$ is right adjoint of the inclusion functor $\text{Mod-}R=\mathcal{X}\hookrightarrow\text{Mod-}\hat{R}$ which implies that  $t:\text{Mod-}\hat{R}\rightarrow\text{Mod-}R$ also preserves products. 

Finally, the fact that the induced functor $t:\text{Mod-}\hat{R}\rightarrow\text{Mod-}R$ is exact and preserves products and coproducts clearly implies that $\hat{\mathcal{B}}=\{M \in \text{Mod-}\hat{R} \text{: } t(M) \in \mathcal{B}\}$ is a fully exact subcategory of $\text{Mod-}\hat{R}$.

(2) With the obvious abuse of notation, we clearly have that $\mathcal{B} =\hat{\mathcal{B}}\cap\text{Mod-}R$, for any $\mathcal{B}\in\text{Fex}(R)$. The assertion follows immediately from this fact.
\end{proof}

\begin{remark}
A caution is in order in the proof of the preceding result. Given an infinite family of $R$-modules $\{X_i\mid i\in I\}$,  its product in $\text{Mod-}\hat{R}$ and in $\text{Mod-}R$ differ. While the first one is the usual cartesian product $\prod_{i\in I}X_i$, the second one is precisely $(\prod_{i\in I}X_i)R=\tr_P(\prod_{i\in I}X_i)$.
\end{remark}
 
 \begin{definition} \label{def.epimorphism-preadditive}
 Let $f:\mathcal{C}\rightarrow\mathcal{C}'$ be an additive functor between small preadditive categories. We say that it is an \emph{epimorphism of preadditive categories} when the restriction of scalars $f_*:\mathrm{Mod}\textrm{-}\mathcal{C}'\rightarrow \mathrm{Mod}\textrm{-}\mathcal{C}$, taking $M\rightsquigarrow M f$, is a fully faithful functor.
 \end{definition}
 
 The following is our promised generalization of Gabriel-De la Pe\~na's Theorem.
 
 \begin{theorem} \label{t:Gabriel-De la Pena}
 Let $\mathcal{C}$ be a small preadditive category and let $\mathcal{B}$ be an additive subcategory of $\mathrm{Mod}\textrm{-}\mathcal{C}$. The following assertions are equivalent:
 
 \begin{enumerate}
 \item $\mathcal{B}$ is fully exact;
 \item $\mathcal{B}$ is bireflective;
 \item There is a small preadditive category $\mathcal{C}'$ and an epimorphism of preadditive categories $f:\mathcal{C}\rightarrow\mathcal{C}'$ (that can be taken to be surjective on objects) such that $\mathcal{B}$ is the essential image of $f_*:\mathrm{Mod}\textrm{-}\mathcal{C}'\rightarrow \mathrm{Mod}\textrm{-}\mathcal{C}$.
 \end{enumerate}
 \end{theorem}
 
 \begin{proof}
$(1)\Longrightarrow (2)$ It goes exactly as the proof of \cite[Lemma 1.5]{GabrieldelaPena}, by using our Lemmas \ref{lem.fullyexact-gener.by-a-set} and \ref{lem.bireflectives form a set}.
 
$(2)\Longrightarrow (3)$ Let $\lambda :\text{Mod-}\mathcal{C}\rightarrow\mathcal{B}$ and $\rho :\text{Mod-}\mathcal{C}\rightarrow\mathcal{B}$, respectively, be the left and right adjoint of the inclusion functor $\iota :\mathcal{B}\hookrightarrow\text{Mod-}\mathcal{C}$. Note that this last functor is exact, i.e. $\mathcal{B}$ is an abelian exact subcategory, since it is closed under taking finite coproducts, kernels and cokernels. By using the adjunction $(\lambda ,\iota)$, it immediately follows  that the image by $\lambda$ of a set of small projective generators of $\text{Mod-}\mathcal{C}$ is a set of small projective  generators of $\mathcal{B}$. The canonical set of small projective generators of $\text{Mod-}\mathcal{C}$ is $\{H_x\text{: }x\in\mathcal{C}\}$, where $H_x=\mathcal{C}(?,x):\mathcal{C}^{\rm op}\rightarrow\Ab$ is the representable right $\mathcal{C}$-module, for each $x\in\mathcal{C}$. But it may happen that $\lambda (H_x)=0$, for some $x\in\mathcal{C}$. This happens exactly when $B(x)\cong (\mathrm{Mod}\textrm{-}\mathcal{C})(H_x,B)=0$, for all $B\in\mathcal{B}$. We then take as a small set of generators of $\mathcal{B}$ the subcategory 
   $$\mathcal{U}:=\{U_x=\lambda (H_x)\text{: }x\in\mathcal{C}\text{ and }B(x)\neq 0\text{, for some }B\in\mathcal{B}\}\cup\{0\}.$$
 
 We can view $\mathcal{U}$ as a preadditive subcategory of $\text{Mod-}\mathcal{C}$ and we have an obvious additive functor $f:\mathcal{C}\stackrel{H}{\rightarrow}\text{Mod-}\mathcal{C}\stackrel{\lambda}{\rightarrow}\mathcal{B}$, where $H$ is the Yoneda embedding,  such that  $f(\mathcal{C})=\mathcal{U}$. We still denote by $f$ the induced  additive functor $\mathcal{C}\rightarrow\mathcal{U}$, which is then surjective on objects. Recall that, by the proof of Mitchell's Theorem (see \cite[Corollary 3.6.4]{Popescu}), the assignment $B\rightsquigarrow\mathcal{B}(?,B)_{|\mathcal{U}}$ gives an additive equivalence of categories $F:\mathcal{B}\stackrel{\cong}{\rightarrow}\text{Mod-}\mathcal{U}$.  We will be done once we prove that the composition functor $\mathcal{B}\stackrel{F}{\rightarrow}\text{Mod-}\mathcal{U}\stackrel{f_*}{\rightarrow}\text{Mod-}\mathcal{C}$ is naturally isomorphic to the inclusion functor, for then $f_*$ will be fully faithful and $\Img (f_*)=\Img(\iota)=\mathcal{B}$.
 
The mentioned composition satisfies $(f_* F)(B)(x)=\mathcal{B}(\lambda (H_x),B)$ for each $B \in \mathcal B$ and $x \in \mathcal C$. By adjunction and Yoneda's lemma, we have that
\[\mathcal{B}(\lambda (H_x),B)\cong (\text{Mod-}\mathcal{C})(H_x,\iota (B))\cong \iota (B)(x).\]
It is routine to derive from this that $f_* F\cong\iota$. 
  
$(3)\Longrightarrow (2)$ It is a well-known fact that $f_*:\text{Mod-}\mathcal{C}'\rightarrow\text{Mod-}\mathcal{C}$ has a left adjoint, namely, the extension of scalars functor $f^*:\text{Mod-}\mathcal{C}\rightarrow\text{Mod-}\mathcal{C}'$, and a right adjoint functor $f^!:\text{Mod-}\mathcal{C}\rightarrow\text{Mod-}\mathcal{C}'$ (see \cite[Exercise 4, Section 3.6]{Popescu}). Due to the fully faithful condition of $f_*$ and the fact that $\Img(f_*)=\mathcal{B}$, we conclude that $\mathcal{B}$ is a bireflective subcategory of $\text{Mod-}\mathcal{C}$.
 \end{proof}
 
 \begin{remark} \label{rem.Nicolas-Saorin}
 The last theorem  was first proved in \cite{NicolasSaorin}. However that proof is much longer and involved, while the one here is quite elementary.
 \end{remark}
 
 We end the section with  a question that, after Lemma \ref{lem.fullyexact-gener.by-a-set}, is quite natural.
 
\begin{question} \label{ques.fullyexact-in-Grothendieck}
Given a Grothendieck category $\mathcal{G}$, is any fully exact subcategory bireflective? Is it so when $\mathcal{G}$ is AB4*? Note that the proof of $(1)\Longleftrightarrow (2)$ in the last theorem can be reproduced mutatis mutandis if one is able to prove that the fully exact subcategories of $\mathcal{G}$ form a set.
\end{question}

\subsection{The canonical finitely accessible subcategory of a Grothendieck category}

 The following is an extended version of \cite[Lemma 1.11]{ParraSaorinVirili}.

\begin{proposition}[Generalized Lazard's Trick] \label{prop.Lazard}
Let $\mathcal{A}$ be an AB5 abelian category, let $(X_i,u_{ij}:X_i\to X_j\mid i\leq j \in I)$ and  $(Y_\lambda,v_{\lambda\mu}:Y_\lambda\to Y_\mu\mid \lambda\leq\mu \in \Lambda)$ be two direct systems in $\mathcal{A}$ and let $f:\varinjlim X_i\to\varinjlim Y_\lambda$ be any morphism satisfying the following property:

$(\dagger)$ For each $j\in I$, there is a $\mu =\mu (j)\in\Lambda$ such that $f u_j$ factors through $v_\mu$, where $u_j:X_j\to\varinjlim X_i$ and $v_\mu :Y_\mu\to\varinjlim Y_\lambda$ are the canonical maps associated to the direct limit.

Further, suppose that either one of the following two conditions holds:
 
\begin{enumerate}
 \item all the $v_{\lambda\mu}$ are monomorphisms;
 \item all the objects $X_i$ are finitely generated.
\end{enumerate}

Then, there exists a directed set $\Omega$ and a direct system of morphisms $(g_\omega :X_\omega\to Y_\omega \mid \omega\in\Omega)$ satisfying the following properties:

\begin{enumerate}
\item[(a)] $X_\omega\in\{X_i\text{: }i\in I\}$ and $Y_\omega\in\{Y_\lambda\text{: }\lambda\in\Lambda\}$, for all $\omega\in\Omega$;
\item[(b)] the map $\varinjlim g_\omega :\varinjlim X_\omega\to\varinjlim Y_\omega$ is isomorphic to $f$ (in the category $\Mor(\mathcal{A})$ of morphisms of $\mathcal{A}$). 
\end{enumerate}
\end{proposition}

\begin{proof}
We consider the set $\Omega$ of all triples $(i,\lambda ,g)$, where $i\in I$, $\lambda\in\Lambda$ and $g:X_i\to Y_\lambda$ is a morphism satisfying $v_\lambda g=f u_i$. We define in $\Omega$ the relation $\preceq$ given by $$(i,\lambda,g)\preceq (j,\mu, h)\hspace*{0.5cm}\Longleftrightarrow\hspace*{0.5cm} i\leq j\text{, }\lambda\leq\mu\text{ and }v_{\lambda\mu} g=h u_{ij}.$$ It is clear that $\preceq$ is a partial order in $\Omega$. We prove that $(\Omega,\preceq )$ is a directed set.

Suppose that condition (1) holds. Then, the AB5 condition implies that the canonical map $v_\mu :Y_\mu\to\varinjlim Y_\lambda$ is a monomorphism, for all $\mu\in\Lambda$. Consider two elements $(i,\lambda ,g),(j,\mu,h)\in\Omega$ and fix any $k\in I$ such that $i,j\leq k$. By condition $(\dagger)$, $fu_k$ factors through some $v_\nu :Y_\nu\to\varinjlim Y_\lambda$. Without loss of generality, we can assume that $\lambda ,\mu\leq\nu$. Then, we have a morphism $s:X_k\to Y_\nu$ such that $v_\nu s=f u_k$. Consequently:  $$v_\nu s u_{ik}=f u_k u_{ik}=f u_i=v_\lambda g=v_\nu v_{\lambda\nu} g,$$ from which we get, using that $v_\nu$ is monic, that $s u_{ik}=v_{\lambda\nu} g$. It follows that $(i,\lambda ,g)\preceq (k,\nu ,s)$ in $\Omega$. The same argument replacing $(i,\lambda ,g)$ by $(j,\mu, h)$ gives that  $(j,\mu ,h)\preceq (k,\nu ,s)$, thus proving that $\Omega$ is a directed set. 

Now we prove that $\Omega$ is directed under condition (2). Take  $(i,\lambda ,g)$, $(j,\mu, h)\in\Omega$ and choose elements $k \geq i,j$ and $\nu \geq \lambda, \mu$, and a morphism $s:X_k\to Y_\nu$ satisfying $v_\nu s =fu_k$ as in the preceding paragraph. This last identity implies that $v_\nu(v_{\mu\nu}h-su_{jk})=v_\nu(v_{\lambda\nu}g-su_{ik})=0$ and, consequently, that both $\Img(v_{\mu\nu}h-su_{jk})$ and $\Img(v_{\lambda\nu}g-su_{ik})$ are contained in $\Ker v_\nu$. Since, by the AB5 condition, we have that $\Ker v_\nu = \bigcup_{\nu \leq \rho}\Ker v_{\nu\rho}$, and both $X_i$ and $X_j$ are finitely generated objects, we can find a $\rho \geq \nu$ such that the images of $v_{\mu\nu}h-su_{jk}$ and of $v_{\lambda \nu}g-su_{ik}$ are contained in $\Ker v_{\nu\rho}$. In particular, this implies that $v_{\mu\rho}h=v_{\nu\rho}su_{jk}$ and $v_{\lambda\rho}g=v_{\nu\rho}su_{ik}$. Using that the triple $(k,\rho,v_{\nu\rho}s)$ actually is an element of $\Omega$, these two identities say that $(i,\lambda,g) \preceq (k,\rho,v_{\nu\rho}s)$ and $(j,\mu,h) \preceq (k,\rho,v_{\nu\rho}s)$. This proves that $\Omega$ is directed.

Finally, for each $w=(i,\lambda ,g)\in\Omega$ set $X_\omega:=X_i$, $Y_\omega :=Y_\lambda$ and $g_\omega:=g$, so that we get a direct system of morphisms, $(g_\omega:X_\omega \rightarrow Y_\omega\mid\omega \in \Omega)$, satisfying (a) of the statement. Condition ($\dagger$) guarantees that, for each $i\in I$, there is an $\omega\in\Omega$ such that $X_\omega =X_i$. Moreover, if $(i,\lambda ,g)\in\Omega$ and $\lambda\leq\mu$ for some $\mu \in \Lambda$, then $(i,\mu ,v_{\lambda\mu} g)\in\Omega$. This says that the canonical map $\Omega\to\Lambda$ (given by $(i,\lambda ,g)\rightsquigarrow\lambda$) has as image a cofinal subset of $\Lambda$. This implies that $\varinjlim g_\omega$ is isomorphic to $f$ in the category of morphisms of $\mathcal{A}$, which concludes the proof.
\end{proof}


Our next result shows how to associate in a natural way a finitely accessible additive category to any Grothendieck category. 

\begin{proposition} \label{prop.lim-fpG}
Let $\mathcal{A}$ be an AB5 abelian category. The following assertions hold:

\begin{enumerate}
\item $\fp (\mathcal{A})$ is closed under taking cokernels and, when $\mathcal{A}$ is a Grothendieck category, $\fp (\mathcal{A})$ is skeletally small and closed under taking extensions.
\item There is an equality of subcategories $\Pres(\fp (\mathcal{A}))=\varinjlim\fp (\mathcal{A})$.
\item When $\mathcal{A}$ is a Grothendieck category, the subcategory $\varinjlim\fp (\mathcal{A})$ is a coreflective subcategory, which is finitely accessible.
\end{enumerate}
\end{proposition}

\begin{proof}
(1) is well-known (see, e.g., \cite[Lemma 3.5.9]{Popescu}).

(2) By the explicit construction of the colimit in any category with coproducts and cokernels, we have that the inclusion  $\text{Pres}(\fp (\mathcal{A}))\supseteq\varinjlim\fp (\mathcal{A})$ holds. On the other hand, if $A\in\text{Pres}(\fp (\mathcal{A}))$ and we fix an exact sequence $\coprod_{i\in I}U_i\stackrel{f}{\to}\coprod_{\lambda\in\Lambda}V_\lambda\to A\to 0$, where $U_i,V_\lambda\in\fp (\mathcal{A})$, for all $i\in I$ and $\lambda\in\Lambda$, then we can apply the generalized Lazard's Trick to the morphism $f$ and the direct systems $(X_F,i_{FG}\mid F\subseteq G \subseteq I\textrm{ finite})$ and $(Y_{\Gamma},j_{\Gamma\Theta} \mid \Gamma \subseteq \Theta \subseteq \Lambda \textrm{ finite})$, where $X_F=\coprod_{i\in F}U_i$, $Y_{\Gamma}=\coprod_{\lambda\in\Gamma}V_\lambda$ and $i_{FG}$ and $j_{\Gamma\Theta}$ are inclusions, for all finite subsets $F\subseteq G \subseteq I$ and $\Gamma\subseteq \Theta \subseteq \Lambda$. Indeed, here each $X_F$ is a finitely presented, whence finitely generated, object and all $j_{\Gamma\Theta}$ are sections, whence monomorphisms. It easily follows that condition $(\dagger)$ in the last proposition holds. 

Then we have a certain direct system of morphisms, $(g_\omega :X_\omega\to Y_\omega \mid \omega\in\Omega)$, satisfying the properties of Proposition \ref{prop.Lazard}, in particular that $\varinjlim g_\omega$ is isomorphic to $f$. It then follows that $\Coker f\cong\varinjlim \Coker f_\omega$. But $\Coker f_\omega \in\fp (\mathcal{A})$, since $X_\omega ,Y_\omega\in\fp (\mathcal{A})$ and $\fp (\mathcal{A})$ is closed under taking cokernels. 

(3) Let $\mathcal{U}$ be a set of representatives of the isoclasses of the objects in $\fp (\mathcal{A})$.  We then have that $\PU=\varinjlim\fp (\mathcal{A})$ and we only need to prove that this subcategory is coreflective. By Theorem \ref{t:PresCoreflective}, it is enough to prove that it is closed under taking cokernels. Let then $f:\varinjlim X_i\to\varinjlim Y_\lambda$ be any morphism, where  $(X_i,u_{ij}:X_i\to X_j\mid i\leq j \in I)$ and  $(Y_\lambda,v_{\lambda\mu}:Y_\lambda\to Y_\mu \mid \lambda\leq\mu \in \Lambda)$ are direct systems in $\fp (\mathcal{A})$. By the definition of a finitely presented object, property $(\dagger)$ of Proposition \ref{prop.Lazard} holds. Moreover, each $X_i$ is finitely generated. Therefore we have a directed set $\Omega$ and a direct system $(g_\omega :X_\omega\to Y_\omega \mid \omega \in \Omega)$ of morphisms in $\mathcal{A}$ satisfying conditions (a) and (b) of the mentioned proposition. Since colimits are right exact we get that $\Coker f=\varinjlim\Coker g_\omega$. But, by assertion (1), we have that $\Coker g_\omega \in\fp (\mathcal{A})$, for all $\omega\in\Omega$. We then get that $\Coker f\in\varinjlim\fp (\mathcal{A})$, so that  $\PU=\varinjlim\fp (\mathcal{A})$ is closed under taking cokernels. 
\end{proof}

\end{document}